\documentclass[10pt]{amsart} 
\usepackage{graphicx}
 \usepackage{caption}
 
\usepackage{psfrag}
\usepackage{float}
\usepackage{subfigure}
\usepackage{psfrag}
\usepackage{epsfig}
\usepackage{amsmath}
\usepackage{amscd}

\usepackage{amsthm, amsfonts, amsxtra, amssymb, amscd}
\usepackage[all]{xy}
\xyoption{2cell}
\usepackage{enumerate}
\def\piff{\eta}
    \usepackage[unicode]{hyperref}

\theoremstyle{plain}
\newtheorem*{lemma*}{Lemma}
\newtheorem{lemma}[subsection]{Lemma}
\newtheorem*{theorem*}{Theorem}
\newtheorem{theorem}{Theorem}
\newtheorem*{proposition*}{Proposition}
\newtheorem{proposition}[subsection]{Proposition}
\newtheorem*{corollary*}{Corollary}

\newtheorem{corollary}[subsection]{Corollary}
\theoremstyle{definition}
\newtheorem*{definition*}{Definition}

\newtheorem{definition}[subsection]{Definition}
\newtheorem*{example*}{Example}
\newtheorem{example}[subsection]{Example}
\theoremstyle{remark}
\newtheorem*{remark*}{Remark}\newtheorem*{proof*}{Proof}
\newtheorem{remark}[subsection]{Remark}

\def\e{{\varepsilon}}

\def\del{{\delta}}
\def\A{{\mathcal A}}
\def\Z{{\mathbb Z}}\def\Q{{\mathbb Q}}
\def\R{{\mathbb R}}
\def\C{{\mathbb C}}
\def\ome{{\omega}}

\def\N{{\mathbb N}}

\newcommand{\T}{\mathbb{T}}
\newcommand{\ii}{{\rm i}}
\begin{document}

 \title{  The energy graph of   the non linear Schr\"odinger  equation}
\author{ M.  Procesi*,
 C.  Procesi{**},   \and Nguyen Bich Van{***}.  }
 \thanks{ *Universit\`a di Roma,   La Sapienza,  supported by ERC grant HamPDEs under FP7,   {**} {***} Universit\`a di Roma,   La Sapienza }\maketitle
\begin{abstract}
We discuss the stability of a class of normal forms of the completely resonant non--linear Schr\"odinger  equation on a torus described in \cite{PP}.   The discussion is essentially combinatorial and algebraic in nature.
 \end{abstract}

\tableofcontents
\section{Introduction} 
In this paper we study the completely resonant cubic Nonlinear Schr\"odinger equation (NLS):
\begin{equation}\label{nonn}
\ii u_t -\Delta u= |u|^2 u
\end{equation}
on the $n$ dimensional torus $\T^n$.
More precisely we analize the quadratic Normal form Hamiltonian,  introduced in \cite{PP},  of the NLS equation \eqref{nonn}, with the purpose of proving 
{\em non-degeneracy} and stability results for its dynamics.  Our dynamical results are summarized in Propositions \ref{teoM} and \ref{ello} which in turn follow from our main Theorem \ref{Lase} . This theorem, whose lenghtly proof occupies  most of the paper, is of  algebraic, combinatorial and geometric  nature, and can in principle be formulated with no previous knowledge of the NLS. In the first ten pages we recall,  for convenience of the reader,  the  results on the NLS normal form  proved in \cite{PP}, and we show how to deduce our dynamical results from Theorem \ref{Lase}. 
Let us briefly- and somewhat na\"ively-- recall the theory 
 of {\em Poincar\'e-Birkhoff Normal Form}. The Birkhoff normal form reduction was  developed in order to study the long-time behaviour of the solutions of a dynamical system  close to an equilibrium and represents a non-linear analog to the {\em canonical form} for matrices. For a classical introduction see \cite{ArGm}, \cite{Bir}, \cite{PoH}, \cite{HZ}; for the application to PDEs see for instance \cite{BG}.
    
  At a purely formal level, consider a non-linear Hamitonian dynamical system  with an elliptic fixed point:
  $$ 
  H(p,q)=\sum_{j \in I} \lambda_j (p_j^2+q_j^2)+ H^{>2}(p,q)\,,\quad \lambda_j\in \R
  $$
  here the index  set  $I$ is finite or possibly denumberable while $H^{>2}(p,q)$ is some polynomial with minimal degree $>2$.
  By definition   the {\em normal form reduction} at order $N$  is a symplectic change of variables $\Psi_N$  which  reduces $H$ to its resonant terms:
    $$ 
    H(p,q)\circ \Psi_N = \sum_j \lambda_j (p_j^2+q_j^2)+ H^{>2}_{Res}(p,q) + H^N(p,q)
  $$   
where $H^{>2}_{Res}$   Poisson commutes with $ \sum_j \lambda_j (p_j^2+q_j^2)$ while $H^N(p,q)$ is a formal power series of minimal degree $>N+1$.

There are two classes of problems in this scheme:

(i) Even though $H^N$ is of minimal order $N+1$ its norm may diverge  as $N\to \infty$, due to the presence of small divisors.

(ii) If $I$ is an infinite set it is not trivial, even when $N=1$, to show that $\Psi_N$ is an analytic change of variables.

Note that if the $\lambda_j$ are rationally independent then the normal form  $H_{Birk}=\sum_j \lambda_j (p_j^2+q_j^2)+ H^{>2}_{Res}(p,q)$ is integrable, a feature which is used in proving for instance long time stability results.

If the $\lambda_j$ are resonant then $H_{Birk}$ may not be integrable but it is possible that its dynamics is simpler than the one of the original Hamiltonian.

In particular in many examples, including the NLS, one can see that  $H_{Birk}$ has  invariant tori of the form
\begin{equation}
\label{pollo} p_i^2+q_i^2= \xi_i \,,\quad i\in S\subset I;\quad   p_j=q_j=0\,,\quad j\in S^c:=I\setminus S
\end{equation}
on which the dynamics is of the form $\psi\to \psi+\omega(\xi)t$ with $\omega(\xi)$ a diffeomorphism.

One wishes to obtain information on solutions of the complete Hamiltonian  close to these tori. As is well known in order to obtain results one needs to study the Hamilton equations of $H$ linearized 
at this family of invariant  tori.  That is one needs to study the dynamics induced on the normal bundle to these tori.   This is described by a family of linear  operators (between normal spaces) parametrized by the family and the points on the tori.

In terms of equations this is described by a quadratic Hamiltonian with coefficients depending on the parameters $\xi$ and on the angle variables of the tori.  The matrix obtained by linearizing $H_{Birk}$ at the solutions \eqref{pollo} is referred to as the normal form matrix (or normal form).  One of the main results of \cite{PP}  exhibits, for the NLS and for generic choices of $S$, a symplectic change of variables which removes the dependence from the angles, this decouples the dynamics into the one on the tori and one on the normal space.  Moreover in our infinite dimensional case  the matrices of the   normal form  are block diagonal with blocks uniformly bounded.  Thus one has a reduction to an infinite list of decoupled linear equations (depending on the parameters $\xi$).\smallskip

In order to perform perturbation theory algorithms, to obtain informations on the solutions of $H$, one generally uses {\em non-degeneracy } conditions. One of the strongest   requirements is that the matrix  of the normal  form  has non-zero and distinct eigenvalues.
This property is an instance of {\em structural stability}.  In this paper we prove that this condition is satisfied for the   normal forms of the NLS previously introduced provided the parameters $\xi$ are taken outside a countable union of  real hypersurfaces.

\subsection{Structural Stability}\quad Structural stability, for an orbit of a dynamical system or a solution of a differential equation is  a basic, and delicate, question both  for theoretical and practical reasons. It essentially  means that the qualitative behavior of the trajectories, close to the given solutions, is unaffected by small perturbations both of the initial data and of the system itself. 

In the simplest case of the class of  linear differential equation $\dot x=Ax$, where $A$ is a real $n\times n$ matrix, the nature of the orbits depends upon the Jordan canonical form of $A$.  
In particular the discriminant of $A$  is an hypersurface (in the space of all matrices) which contains all special normal forms; its complement is the set of matrices with distinct eigenvalues which  decomposes into connected components. On each such component the number of real eigenvalues is constant, thus these regions are the regions of structural stability. Of course if the matrix $A$ is subject to some restrictions (as being symmetric, symplectic etc.)  the normal forms  are further constrained \cite{A}.

\subsubsection{Stability for the NLS} 

 The normal form of the NLS is described  by  an infinite dimensional Hamiltonian which determines a  linear operator $ad  (N)$,    depending on a finite number of parameters $\xi_i$     (the actions of certain excited frequencies), and acting   on a certain infinite dimensional vector space $F^{  (0,  1)}$    (see \ref{spaceF} ) of functions.

  Stability for this infinite dimensional operator will be interpreted in the same way as it appears for finite dimensional linear systems,   that is  the property that the linear operator is semisimple with distinct eigenvalues.

  This will be shown to be true outside a zero measure set of parameters,   further on  a smaller   set of positive measure we shall show that the dynamic is elliptic.     This condition in a more precise quantitative form (which will be discussed elsewhere) in the Theory of dynamical systems is referred to as the {\em second Melnikov condition}.
We shall apply this  in \cite{PP2} in order to prove,   by a KAM
algorithm,    the existence and stability of quasi--periodic
solutions for the NLS   (not just the normal form).\smallskip

  The fact that this non-degeneracy condition makes at all sense depends on the fact that the normal form matrix decomposes into an infinite direct sum  of finite dimensional blocks. Furthermore, these finite dimensional blocks  are described by translating,  with suitable scalars, a finite number of combinatorially defined matrices,   constructed from certain combinatorial objects called {\em marked colored graphs} (cf. Definition  \ref{mcg} and Remark \ref{mcg1}). Thus the matrices appearing as blocks of the normal form matrix can be combinatorially classified and, in principle, computed. Indeed given a specific graph computing the associated matrix block  is quite simple, so that the question is essentially that of classifying the possible graphs which describe blocks of the normal form. 
  
The characteristic polynomials $\det  (t-ad  (N)_\Gamma)$ of  the normal form operator $ad  (N)$ restricted to the infinitely many blocks $\Gamma$  are all polynomials in the variables $\xi_i$ and $t$ with integer coefficients.  The issue is thus to prove that a  rather complicated infinite list of polynomials in a variable $t$, of degree  increasing with the space dimension,  and with coefficients polynomials in the parameters $\xi_i$  have distinct roots  for generic values of the parameters.

In general, in order to prove that a single polynomial has distinct roots, one has to prove the non--vanishing of its discriminant, for two polynomials to have different roots the condition is the  non--vanishing of the resultant.  In our case we can consider all the characteristic polynomials as having coefficients  in the field of rational functions  in the parameters $\xi_i$, its algebraic closure is a {\em field of algebraic functions}.   Thus if  the discriminant $D(\xi)$ of a given polynomial and the resultant $R(\xi)$ of two distinct    polynomials in $\Q  (\xi_1,  \ldots,  \xi_m)[t]$ are  non--zero as  
polynomials in the $\xi$ we have that   outside the real hypersurfaces $R(\xi)=0,\, D(\xi)=0$ the
two polynomials have distinct roots. Although  both the discriminant and the resultant can be computed by explicit formulas   a  proof of their non--vanishing for the infinite list of complicated polynomials appearing seems to be  a hopeless task.

We thus followed a different approach. Remark that, if we have a  list of different polynomials in one variable $t$, with coefficients in a field $F$ of characteristic 0,  a sufficient condition  that all their roots (in the algebraic closure $\overline F$ of $F$) be distinct is that they are all {\em irreducible} (over $F$) and distinct. This follows immediately from the fact that an irreducible polynomial  $f(t)$ is uniquely determined as the minimal polynomial of each of its roots (cf. \cite{MAr}) and, in characteristic 0, its derivative $f'(t)$ is non--zero. By the irreducibility of $f(t)$ the greatest common divisor between $f(t),f'(t)$ is 1 so all the roots of $f(t)$  are distinct.\smallskip

 Therefore by a rather complex    induction (setting some variables $\xi_i$ equal to zero) we prove:
   \begin{theorem}[Separation and Irreducibility Theorem]\label{Lase}
The characteristic polynomials of the possible graphs giving blocks of the normal form of the NLS are all {\em distinct},  and {\em irreducible}  as polynomials with integer coefficients, that is  in $\Z[\xi_1,\ldots,\xi_m,t]\subset \Q  (\xi_1,  \ldots,  \xi_m)[t]$.
\end{theorem}
In general proving that a polynomial in several variables is irreducible is not an easy task, few general methods are available and none of these seems to apply to our case.  For a given polynomial  with integer coefficients  there exist reasonable computer algebra algorithms to test irreducibility  but this is not a practical method in our case  where the polynomials are infinite and their degrees also tend to infinity.  Fortunately the combinatorics comes to our help as follows.  We start from one of the matrices  describing the Hamiltonian for a block  associated to a given graph $\Gamma$.  If we set one of the parameters $\xi_i=0$ it is easy to verify that the matrix specializes to a direct sum of smaller blocks  of  the same type  for less parameters (cf. Corollary \ref{laffa}).  This remark gives a powerful tool for induction. The characteristic polynomial specializes to the product of the characteristic polynomials of the blocks and, by induction, we may assume that these factors are irreducible. We thus obtain a factorization for the specialized polynomial.

We repeat the argument with a different variable obtaining a different specialization and a different factorization.  It is possible  that these two factorizations cannot arise  from the same factorization of the given polynomial. If this happens we are sure that the polynomial we started with is irreducible.  This is the method we follow in order to prove Theorem  \ref{Lase} and it is the content of Part 2.

Unfortunately this still  
requires a rather tedious and lengthy case analysis and a reduction to some basic cases which we treat by computer algebra algorithms.

The fact that the polynomials are distinct (cf. lemma \ref{seplem})  is based by induction on the irreducibility theorem and it is   relatively easy  to prove.\smallskip

There is another delicate point in this proof, in order for the induction   to work we need to have a complete control  on the graphs that may appear, which is not proved in \cite{PP} and which we do not know for $q>1$. We need to know  that the possible  graphs  satisfy a {\em geometric non-degeneracy} or {\em non resonance}  restriction, given by Proposition \ref{affind}. Precisely one of the presentations of our graphs  is by describing the vertices as integral vectors (in $\Z^m$), then the non degeneracy condition is that these vectors are affinely independent.    The possible graphs are obtained by associating to the combinatorial graphs  a system of  $d$ linear and quadratic equations, in $n$ variables,  which depend on the tangential sites in a quadratic way, where $d+1$ is the number of vertices.  
The graph is thus admissible if and only if these equations have solutions in $\Z^n\setminus S$, this arithmetic analysis is too difficult to perform and we study wether   they have solutions in $\R^n\setminus S$.
The idea is that if these equations are independent then they can be at most $n$. In fact for a   geometrically non degenerate graph the condition of independence is fulfilled when $d\leq n$, the case $d>n$ has been treated completely by methods of algebraic geometry in \cite{PP}, in the same paper we proved only a partial result  on degenerate graphs. Here, by restricting to the case $q=1$, we are able to show that, for generic choices of $S$,  a resonant graph gives a system which has no solutions in  $\R^n\setminus S$.
Note that a {\em resonance}, namely a relation between the vertices of the graph, implies a linear relation among the linear terms of the system of equations.  Such a relation may correspond either to a relation on the equations or an incompatibility condition for the system. So first  we  reduce to    minimal cases (only one resonance), and then we study   those graphs for which the equations are generically compatible.  This produces two cases, either the system has only solutions in $S$  or only in $\C^n\setminus \R^n$, this concludes the proof.

The strategy follows these steps: First we reduce to the case of trees and describe the resonance in terms of edges (instead of vertices). Next we analyze in a combinatorial way all the possible minimal resonances (in this analysis the hypothesis  $q=1$ is essential).  Then we prove that we can essentially reduce to those trees in which all the edges contribute to the resonance.
Finally we show that such trees have at most two trivalent vertices (that is a vertex from which 3 edges originate), the other vertices have valency $1,2$.  At this point one can deduce from the system a simple equation which has   only solutions in $S$  or only in $\C^n\setminus \R^n$ by inspection.\smallskip

The proof of Proposition \ref{affind}  is the content of   Part  1, the proof we found is rather complex  and takes a good 20 pages of detailed combinatorial analysis. 

\subsubsection{Dynamical consequences}
From the fact that the characteristic  polynomials of the matrix blocks are described through finitely many graphs we shall be able to show the existence of a {\em discriminant variety} also in the infinite dimensional setting and show:
  \begin{corollary} \label{teoM}There exists   an algebraic hypersurface $\A$, in the space $\R^m$ of the parameters $\xi$, and a finite number of algebraic functions $\theta_i  (\xi)$ homogeneous of degree 1 on the region $\R^m\setminus \A$,   so that  the eigenvalues of $Q:=-\frac 12\ii ad  (N)$ on $F^{0,  1}$ are of the form $n+\theta_j  (\xi)+a  (\xi)$,    $a  (\xi)=\sum_jn_j\xi_j,  \ n_j\in\Z,   \ \sum_jn_j=-1,  \ n\in\N$.  In particular the eigenvalues are   all distinct and non--zero   outside the countable union of hypersurfaces $\theta_i  (\xi)-\theta_j  (\xi)-a  (\xi)\neq 0$ for all $i\neq j$ and $a  (\xi)$.

\end{corollary}
\begin{proof}
We know that $F^{0,  1}$ decomposes into the direct sum of infinitely many  blocks
corresponding to the connected components of the graph
$\Lambda_S$ defined in \ref{gliegg}.

From Theorem  \ref{Lase} we have   that the
characteristic polynomials of the matrices $ad  (N)$  in the
various blocks are irreducible and distinct.      In our case we have seen that,   for two distinct blocks,   this
produces a non zero polynomial  whose non vanishing is equivalent
to the condition that the two blocks have distinct eigenvalues. In
principle this gives countably many hypersurfaces.   Since we know
that our infinite list of matrices is obtained from a finite list
by adding a scalar matrix of the form $  (n+\sum_in_i\xi_i)I$ we
obtain a finite number of distinct algebraic function $\theta_i
(\xi)$,   outside an algebraic hypersurface $\A$,   which are the
eigenvalues of all the combinatorial blocks.  The condition is
$\theta_i  (\xi)-\theta_j  (\xi)-a  (\xi)\neq 0$ for all $i\neq j$
and $a  (\xi)=\sum_jn_j\xi_j,  \ n_j\in\Z,   \ \sum_jn_j=0$.   
\end{proof}
In \cite{PP2}  we shall refine this Theorem by exhibiting a region of positive measure where the eigenvalues are explicitly bounded away from 0.\smallskip

By construction of the matrix $Q$, real eigenvalues of $Q$ correspond to imaginary eigenvalues of $ad(N)$.
We  have seen that outside a real hypersurface the eigenvalues of
all the combinatorial blocks are distinct.   Thus outside this
hypersurface the cone of the $\xi_i$ decomposes into open regions
where the number of real roots is constant.  We can furthermore  show (see \S \ref{EO})that
\begin{proposition}\label{ello}
The open region  where all the eigenvalues of $Q$ are real is non empty.
\end{proposition}

As a consequence of Proposition \ref{teoM}  one easily sees that one can perform a symplectic coordinate change so that the Hamiltonian is in {\em diagonal canonical form}, that is we have an infinite sum  $\sum_k\theta_k|z_k|^2$  corresponding to the real eigenvalues, plus a (possibly empty, depending on the connected region of $O_\del\setminus \A$ where $O_\del$ is a small hypercube), finite sum of hyperbolic terms corresponding to the complex eigenvalues. Then Proposition \ref{ello} ensures that on an open region of parameters the Hamiltonian is diagonal and elliptic.
\smallskip

 \begin{remark}
 No  knowledge of the NLS  is
necessary in order to understand the Theorems of this paper which may be formulated as  purely geometric questions. 
\end{remark}\begin{remark}
 We should remark that  only finitely many
of the infinite blocks   are not self adjoint matrices.  If one restricts
the analysis to the self adjoint blocks the proofs simplify
drastically,   in particular this is true for the first part which
admits a far reaching generalization   (cf.  Theorem \ref{ridma}).
\end{remark}
\begin{remark}
The restriction to $q=1$  plays a major role in both parts of the paper. However for any $q$ and dimension $n=1$  all the results of this paper have been proved in the Ph. D. Thesis of Nguyen Bich Van.
\end{remark}
\begin{remark}
In general ($q>1,n>1$) although we do not know that the eigenvalues are distinct we can use a {\em Fitting decomposition}  with blocks corresponding to distinct eigenvalues. It turns out that these blocks are uniformly bounded for generic $S$.
\end{remark}
\begin{remark}
In Proposition \ref{ello} we have pointed out the existence of an elliptic region.  It is easy to exhibit large regions where there are complex eigenvalues, which however can be at most a finite number bounded by a function of $n,m$.
\end{remark}
      
\section{Preliminaries}
{\em We start by presenting an elementary geometric problem which originates from the NLS  but can be explained and treated in a completely independent way. Then we briefly describe the NLS normal form and show the origin and importance of the geometric problem in this context.}
\subsection{An elementary geometric problem}  Given a point $p$ in a sphere in Euclidean space  $\R^n$ we can consider its {\em antipode} or {\em mirror}  point $p'$.  A  similar construction holds in  the case of two parallel hyperplanes $H_1,  H_2$.  Given a point $p$ in one of them,   say for instance $H_1$,   we can construct a {\em mirror point} $p'\in H_2$ by drawing the line $r$ perpendicular to $H_1$ through $p$  and taking as $p'$  the point of intersection between $r$ and $H_2$.   If we have several spheres $S_1,  \ldots,  S_a$  and pairs of parallel hyperplanes $  (H_1^1,  H_2^1),  \ldots,     (H_1^b,  H_2^b)$  we have, for a point in the intersection of $h$ such hypersurfaces, $h$  mirror points.   Each of them in turn could  have several mirror points.   The combinatorics resulting is encoded by a 2-colored graph,   having as vertices the points of $\R^n$  and two types of edges; the edges  colored black represent mirror pairs in parallel hyperplanes while edges colored red represent  antipode points in one of the spheres.  The edges are understood as purely combinatorial and not as segments of $\R^n$.  The combinatorics of this graph can be extremely complicated and reflects partially the complex relative positions of all the given hypersurfaces.  

In our case a configuration of previous type is
  associated to a set $S$ (the tangential sites) as follows:   given
two distinct elements $v_i,  v_j\in S$  construct the sphere
$S_{i,  j}$  having  the two vectors as opposite points of a
diameter  and the two Hyperplanes,   $H_{i,  j},  \ H_{j,  i}$,
passing through $v_i$  and $v_j$ respectively,   and perpendicular
to the line though the two vectors  $v_i,  v_j. $
\smallskip

From this configuration of spheres and pairs of parallel
hyperplanes  we deduce, by the previous rules, a {\em combinatorial colored graph},
denoted by $\Gamma_S$,   with vertices  the points in $\R^n$ and
two types of edges,   which we call {\em black} and {\em red}.

\begin{itemize}\item A black edge connects  two points $p\in H_{i,  j},  \ q\in H_{j,  i}$,   such that the line $p,  q$ is orthogonal to the two hyperplanes,   or in other words $q=p+v_j-v_i$.

\item A red edge connects  two points $p,  q\in S_{i,  j} $ which are opposite points of a diameter   ($p+q=v_i+v_j$).

\end{itemize}
{\bf The Problem}\quad The problem consists in the study of the connected components of this  graph.   Of course  the nature of the graph depends upon the choice of $S$  but one expects a relatively simple behavior for $S$  {\em generic}.

\smallskip

It is immediate by the definitions  that the points in $S$ are all
pairwise connected by black and red edges  and it is not hard to
see  that,   for generic values of $S$,    the set $S$ is itself a
connected component which we call the {\em special component}.

 What we expect to have,   as explained in \S\ref{gra} and proved in Part 1,   is:

\begin{proposition}\label{affind}
For generic choices of $S$  the  connected components of this
graph,   different from the special component,   are formed by
affinely independent points.

In particular each component has at most $n+1$ points.
\end{proposition}
In the next paragraph we explain how this problem arises in the NLS. The NLS considered in \cite{PP} depend upon an integer parameter $q$  but here we concentrate in the simplest case when $q=1$, which is connected to the previous geometric problem,  and we have the {\em cubic NLS} the remaining cases are essentially open.

\subsection{Some background}

The  cubic  NLS on a torus is a Hamiltonian system,   the
symplectic variables are the Fourier  coefficients of the
functions $u( \varphi):= \sum_{k\in \Z^n} u_k e^{\ii (k, \varphi)}$, the symplectic structure is $\ii \sum_{k\in \Z^n} d u_k\wedge d \bar u_k $
 and the Hamiltonian is

\begin{equation}\label{Ham}H:=\sum_{k\in \Z^n}|k|^2 u_k \bar u_k \pm  \sum_{k_i\in \Z^n: \sum_{i=1}^{4}  (-1)^i k_i=0}\hskip-30pt u_{k_1}\bar u_{k_2}u_{k_3}\bar u_{k_4}.  \end{equation}  We shall choose the sign $+$  for simplicity of notations.  We  perform a step of ``Resonant Birkhoff  normal form''.  Denote by $K:=  \sum_{k\in \Z^n}|k|^2 u_k \bar u_k$. A monomial  $\prod_iu_{k_i}^{\alpha_i}\bar u_{k_i}^{\beta_i}$ in the $u_k,  \bar u_k$ is an eigenvector  for $\{K,  -\}$ of eigenvalue $\sum_i(\alpha_i-\beta_i)|k_i|^2$ and such a step is a symplectic change of variables under which we cancel all or some of the quartic terms which do not Poisson commute with $K$,   to the cost of introducing higher order terms which are then treated as {\em a perturbation}.  The condition of commuting with $K$ is  $  \sum_{i=1}^{4}  (-1)^i |k_i|^2= 0$.   Dropping the perturbation one has a {\em restricted model}.  
\begin{equation}\label{Ham1}H:=\sum_{k\in \Z^n}|k|^2 u_k \bar u_k +  \sum_{k_i\in \Z^n: \sum_{i=1}^{4}  (-1)^i k_i=0,  \atop { \sum_{i=1}^{4}  (-1)^i |k_i|^2= 0}}\hskip-30pt u_{k_1}\bar u_{k_2}u_{k_3}\bar u_{k_4} .  \end{equation}
Note that the two conditions $\sum_{i=1}^{4}  (-1)^i k_i=0,  \    \sum_{i=1}^{4}  (-1)^i |k_i|^2= 0 $ have a geometric interpretation, that is the four points $k_1,k_2,k_3,k_4$ are the vertices of a {\em rectangle}.

As it is well known   (cf.  Keel--Tao \cite{KTa} and
Gr\'ebert--Thomann \cite{GT}) this restricted model admits infinitely many
invariant subspaces  defined by requiring $u_k=0$ for all $k\notin
S$ where $S=\{v_1,  \ldots,  v_m\}$,   {\em tangential sites}, is
some   (arbitrarily large) subset of $\Z^n$ satisfying a {\em
completeness condition}  (cf. \cite{PP}, 2.1.1).  The dynamics on these subspaces depends
in a subtle way on the geometric properties of $S$ and,  for generic
choices of $S$ the behavior is integrable  (cf. \cite{PP}, Proposition 1).
 In order to understand how to pass from solutions of the restricted model to  true solutions of the NLS  one has to have some structural stability result that is, as we explained before, control of  the dynamics on the normal bundle to the family of invariant tori in the given invariant subspace.   In coordinates
we set \begin{equation} \label{chofv}u_k:= z_k \;{\rm for}\; k\in
S^c\,  ,  \quad u_{v_i}:= \sqrt {\xi_i+y_i} e^{\ii x_i}= \sqrt
{\xi_i}  (1+\frac {y_i}{2 \xi_i }+\ldots  ) e^{\ii x_i}\;{\rm for}\;  i=1,  \dots m,
\end{equation} 
considering  the $\xi_i>0$ as parameters, with $|y_i|<\xi_i$, while $y,  x,  w:=  (z,  \bar z)$ are dynamical variables.  In these variables the Hamiltonian can be decomposed as $$ H\circ\Phi_\xi=   (\ome  (\xi),  y) +\sum_{k\in S^c} |k|^2|z_k|^2 + {\mathcal Q}  (\xi,x,  w)+   P  (\xi,  y,  x,  w)=N+P. $$
Where $N:=  (\ome  (\xi),  y) +\sum_{k\in S^c} |k|^2|z_k|^2 +
{\mathcal Q}  (\xi,x,  w)$, with ${\mathcal Q}  (\xi, x, w)$ quadratic,  is the {\em normal form} and $P$ the {\em
perturbation}.    
\smallskip

  We  use systematically the fact that this Hamiltonian commutes with {\em momentum $M$} and {\em mass $L$}:
  \begin{equation}\label{moma0}
M=\sum_i\xi_i v_i+\sum_i   y_i v_i+\sum_{k\in S^c}k|z_k|^2,  \quad
L=\sum_i\xi_i  +\sum_i   y_i  +\sum_{k\in S^c} |z_k|^2\,  ,
\end{equation}

We have, after some renormalizing,
   $\ome_i  (\xi):  = |v_i|^2 -2\xi_i 
$.
Finally the quadratic form is \begin{equation}
 \mathcal Q  (\xi,x,w)= 4\sum^*_{  1\leq i\neq j\leq m    \atop h,  k \in S^c}\sqrt{\xi_{i}\xi_{j}}e^{\ii  (x_{i}-x_{j})}z_{h}\bar z_{k } +
\end{equation} $$ + 2\sum^{**}_{ 1\leq i< j\leq m   \atop h,  k \in S^c }\sqrt{\xi_{i}\xi_{j}}e^{-\ii  (x_{i}+x_{j})}z_{h} z_{k } +
  2\sum^{**}_{ 1 \leq i<j\leq m    \atop h,  k \in S^c }\sqrt{\xi_{i}\xi_{j}}e^{\ii  (x_{i}+x_{j})}\bar z_{h}\bar  z_{k }.  $$  Here $\sum^*$ denotes that $  (h,  k,  v_i,  v_j)$ satisfy:
 $$ \{  (h,  k,  v_i,  v_j)\,  |\,    {h+v_i= k+v_j},  \  { |h|^2+|v_i|^2=| k|^2+|v_j|^2}\}. $$
  and $\sum^{**}$,   that  $  (h,  v_i,   k,  v_j)$ satisfy:
   $$ \{  (h,  v_i,  k,  v_j)\,  |\,    {h+k= v_i+v_j},  \  { |h|^2+|k|^2=| v_i|^2+|v_j|^2}\}. $$
Notice that in the sums  $  \sum^{**}$ each term appears twice.   These constraints describe exactly the two types of rectangles in which two vertices lie in $S$  and the others in $S^c$, thus   these last two vertices are joined, by definition, by a black edge in the first case (in which they are vertices of a side of the rectangle) and a red in the second (in which they are opposite vertices  of the rectangle). Note that the edges correspond to interacting sites. \vskip10pt

We have described   a very complicated infinite dimensional quadratic
Hamiltonian which we wish to decompose   into infinitely many decoupled finite dimensional blocks,  corresponding to the components of the geometric graph $\Gamma_S$ defined in the previous paragraph.    
 In  \cite{PP}  we show that
this  is possible and we 
also   proved the existence of  a symplectic change of
variables which makes the angles disappear.  
 \subsection{The operator $ad  (N)$\label{LaMa}}
  \begin{definition}\label{glispa}Denote by  ${\Z^m}:=\{\sum_{i=1}^ma_ie_i,  \  a_i\in\Z\}$  the lattice  with basis the elements $e_i$.

 Consider   the {\em mass $\eta$} and the {\em momentum $\pi$} (the name comes from dynamical considerations): $$ \eta:{\Z^m}\to \Z,  \  \eta  (e_i):=1,\quad \pi:\Z^m\to \Z^n,  \ \pi_S=\pi:e_i\mapsto v_i.$$
\end{definition}

At this point it is useful to formalize the idea of {\em energy
transfer} in a combinatorial way. Let    $S^2[{\Z^m}]:=\{\sum_{i,
j=1}^ma_{i,  j}e_ie_j\},  \  a_{i,  j}\in\Z$ be   the  polynomials
of degree $2$  in the $e_i$ with integer coefficients.  We extend
the map $\pi$ and introduce a linear map from ${\Z^m}$ to $S^2
({\Z^m})$ denoted $a\mapsto a^{  (2)}$ as:
 \begin{equation}
\label{adue} \pi  (e_i) = v_i,  \quad
\pi  (e_ie_j):=  (v_i,  v_j),    \quad *^{  (2)}:{\Z^m}\to S^2
({\Z^m}),  \ e_i\mapsto e_i^2 .
\end{equation}

We have $   \pi  (AB)=  (\pi  (A),  \pi  (B)),  \forall A,
B\in{\Z^m}. $

\begin{remark}
Notice that we have   $a^{  (2)}=a^2$ if and only if  $a$ equals 0 or one of the variables $e_i$.
\end{remark}
\subsubsection{The space  $F^{0,  1}$\label{spaceF}}  We start from the space  $V^{0,  1}$ of functions with basis the elements $$ e^{\ii \sum_j\nu_jx_j}z_k ,  \quad   e^{-\ii \sum_j\nu_jx_j}\bar z_k,  \quad k  \in S^c  .$$

In this space the conditions of  commuting with momentum, resp. with mass select the elements, called {\em frequency basis}
 \begin{equation}\label{llab}
 F_B= e^{\ii \sum_j\nu_jx_j}z_k ,  \quad   e^{-\ii \sum_j\nu_jx_j}\bar z_k ;   \quad k  \in S^c \end{equation}$$  \sum_j\nu_jv_j+k=\pi  (\nu)+k=0 , \qquad \text{resp.}\quad \sum_j\nu_j+1=0 .$$
Denote by $F^{0,  1}$  the subspace of   $V^{0,  1}$ commuting
with momentum and mass.\footnote{this convention is different from \cite{PP} where we only impose commutation with momentum}
  \smallskip

   An element of $F_B$  is completely determined by the value of
$\nu$  and the fact that the $z$ variable may or may not be
conjugated.   By
construction $\nu\in \Z^m_c$ where
 \begin{equation}\label{zmc}
\Z^m_c:=\{\mu\in\Z^m\,  |\,   -\pi  (\mu)\in S^c\}\,  .
\end{equation}
Denote by  $\Theta\subset \Z^m$   the kernel of $\pi_S:e_i\mapsto v_i$ then, by Formula \eqref{zmc}, we have $\Z^m_c=\Z^m\setminus\bigcup_i -e_i+\Theta$.

  Now 
$ad  (N)$ acts on $F^{0,  1}$, its matrix representation, in the frequency basis, decomposes into    infinitely many finite dimensional  blocks described by matrices with coefficients
quadratic polynomials in the variables $\sqrt{\xi_i}$.   One
easily sees that in the characteristic polynomial of each one of
these matrices the square roots disappear (Lemma \ref{spr}).  \smallskip

\subsection{The Cayley  graphs\label{Cg}} We recall how we have found useful to cast some of the description of the operator $ad(N)$ into the language of group theory and in particular of the {\em Cayley graph} (cf. \cite{MKS}).   In fact  to a matrix $C=(c_{i,j})$ we can always associate a graph, with vertices the indices of the matrix, and an edge between $i,j$ if and only if  $c_{i,j}\neq 0$. For the matrix of $ad  (N)$ in the frequency basis the relevant graph comes from a special Cayley graph.\medskip

Let $G$ be a group and $X=X^{-1}\subset G$ a subset.
\begin{definition}\label{mcg}
An $X$--marked graph is an oriented graph $\Gamma$ such that each oriented edge is marked with an element $x\in X$.
$$\xymatrix{ &a\ar@{->}[r]^{x} &b&  &a\ar@{<-}[r]^{\ x^{-1}} &b&   }$$
We mark the same edge,   with opposite orientation,   with
$x^{-1}$.  Notice that if $x^2=1$  we may drop the orientation of
the edge.
\end{definition}
 A typical way to construct an $X$--marked graph is the following.  Consider an action $G\times A\to A$ of $G$ on a set $A$,   we then define.
 \begin{definition}[Cayley graph] The graph $A_X$ has as  vertices   the elements of $A$ and,   given $a,  b\in A$ we join them by an oriented edge $\xymatrix{a\ar@{->}[r]^{x} &b }      $,   marked $x$,   if $b=xa,  \ x\in X$.
 \end{definition}

In our setting the relevant group is   the group of transformations of $\Z^m$ generated by translations $a:x\mapsto x+a$ and {\em sign change} $\tau:x\mapsto -x$. Thus $G:={\Z^m}\rtimes\Z/  (2)$ is the  {\em
semidirect product},   and $\tau:=  (0,  -1),\,G=\Z^m \cup \Z^m\tau$ and the product rule is $a\tau=-\tau a,\ \forall a\in\Z^m$ (notice that this implies $(a\tau)^2=(0,1)$).  We think of an element $a=    e^{\ii \sum_j\nu_jx_j}z_k$ as being associated to the group element  which, by abuse of notation, we still denote by $a=\sum_j\nu_je_j\in \Z^m$. Then $ \bar a= e^{-\ii \sum_j\nu_jx_j}\bar z_k$ is associated to the group element  $a\tau=(\sum_j\nu_je_j)\tau\in \Z^m\tau$. 
\smallskip

 Thus the  frequency basis  is indexed by   elements of $G^1\setminus\bigcup_i \{-e_i+\Theta,(-e_i+\Theta)\tau\}$  where $$G^1:=\{ a, a\tau,\ a\in \Z^m\,|\, \eta(a)=-1\}.$$

We now consider the Cayley graph $G_X$
 of $G$
with respect to the elements $$X^0:=\{e_i-e_j,  \ i\neq j\in[1,  \ldots,  m]\},  \quad X^{-2}:=\{  (-e_i-e_j)\tau,  \  i\neq j\in[1,  \ldots,  m]\}. $$
 If $p\in\Z$ it is easily seen that the set  $G_p:= \{a,  \ \eta  (a)=0,  \    a\tau\,  |\,   \eta  (a)=p\}$ form a subgroup.  In particular \begin{remark}\label{mcg1}
\label{G2}$G_{-2}$ is  generated by the elements  $X:=X^0\cup X^{-2}$, its right cosets are the connected components of the Cayley graph.

In the action of $G_{-2}$ on $\Z^m$ the orbit of 0 is identified to $G_{-2}$ and it is formed by the elements $a\in\Z^m\,|\,\eta(a)\in\{0,-2\}$.  We can thus identify the Cayley graph on $G_{-2}$ with the corresponding graph on this set of elements.

We distinguish the edges by {\em color}, as $X^0$  to be {\em black} and $X^{-2}$  {\em red}, hence the Cayley graph is accordingly colored; by convention we represent red edges with an unoriented double line: $g=(-e_i -e_j)\tau ,\  \xymatrix{a\ar@{=}[r]^{g} &ga }$ (recall that $g=g^{-1}$).

The set $G^1$  is also a right coset of $G_{-2}$   and thus it is also a connected component of the Cayley graph $G_X$.
\end{remark}

\subsubsection{The matrix structure of $ad  (N):=2\ii Q$}  This  is encoded in part
by the Cayley graph $G_X$
 of $G$
with respect to the elements $X:=\{e_i-e_j,    (-e_i-e_j)\tau\}$.

   Given $a=\sum_ia_ie_i ,\ \sigma=\pm 1$ set for $u=(a,\sigma)$
   \begin{equation}\label{ricon}
 C  ((a,\sigma)):= \frac{\sigma }{2}  (a^2+a^{  (2)})=  \frac{\sigma }{2}  ((\sum_ia_ie_i)^2+\sum_ia_ie_i^2) , \end{equation}$$     K  ((a,\sigma)): =\pi  (C  (u) )=\frac{\sigma }{2}   (|\sum_ia_iv_i|^2+\sum_ia_i|v_i|^2).$$
 Sometimes we call $K  (u)$ the \emph{quadratic energy} of $u$, notice that $C  (u)$ has integer coefficients. In particular if $a\in\Z^m_c$ we have $K(a\tau)=-K(a)$ and we set for $a ,b\in\Z^m_c$
\begin{equation}\label{maent}
 Q_{a,  a}=K(a) -\sum_ja_j \xi_j ,  \quad Q_{  a\tau ,   a\tau}=K(a\tau)+\sum_ja_j \xi_j\end{equation}
\begin{equation}\label{maent1} Q_{  a\tau,    b\tau}=-2\sqrt{\xi_i\xi_j}, \  Q_{a,  b}=2\sqrt{\xi_i\xi_j},  \quad\text{if}\ a,  \ b\ \text{are connected by an edge}\ e_i-e_j \end{equation} \begin{equation}\label{maent2}   Q_{a,  b\tau}= -2\sqrt{\xi_i\xi_j},\   Q_{a\tau,  b}=2\sqrt{\xi_i\xi_j},  \quad\text{if}\ a,  \ b\tau\ \text{are connected by an edge}\   (-e_i-e_j)\tau
\end{equation}

We have shown in \cite{PP} that the blocks $Q$ on $F^{0,  1}$  come into pairs of
conjugate Lagrangian blocks  $\Gamma,   \Gamma\tau$.  With respect to the frequency basis
the blocks are described as the connected components of a graph
$\Lambda_S$ which we now describe.

\begin{definition}
Given an edge $\xymatrix{u\ar@{->}[r]^{x} &v    },  $ $u=  (a,
\sigma),  v=  (b,  \rho)=xu,  \ x\in X_q$,      we say that the
edge is {\em compatible} with $S$ or $\pi$ if $K  (u) =K  (v) $.
\end{definition}
Remark now that, if $g\in G$ we have $C(g)=0$ if and only if $g=  -e_i,-e_i \tau $.  We call the elements $\{-e_i,-e_i\tau\}$  the {\em special component}. \begin{definition}\label{gliegg} The  graph $\Lambda_S$    is  the subgraph of $G_X$  inside $G^1\setminus\bigcup_i \{-e_i+\Theta,(-e_i+\Theta)\tau\}$ in which we only keep the compatible edges.
 \end{definition}
   Observe that the  graph $\Lambda_S$    is  invariant under translations by $\Theta$.
 We then have
  \begin{theorem}\label{iblo}
The indecomposable blocks of the matrix $Q$ in the frequency basis correspond to the connected components of the graph $\Lambda_S$.

In a block the entries  of $Q$ are given by \eqref{maent}, \eqref{maent1}, \eqref{maent2}.
\end{theorem}
The fact that in the graph $\Lambda_S$ we keep only compatible
edges implies in particular  that the {\em scalar part}  $K((a,\sigma))$   (which
is an integer) is constant on each block.  On the other hand,   in
general,   there are infinitely many blocks with the same scalar
part. It will be convenient to ignore the scalar term $diag(K((a,\sigma)))$, given a compatible connected component $A$ we hence define the matrix $C_A= Q_A-diag(K((a,\sigma)))$. 

One of the main ingredients of our work is to understand the possible connected components $\Gamma$ of the graphs $\Lambda_S$ for $S$ generic (but not necessarily fixed), we do this  by choosing a vertex $u\in\Gamma$ which we call the {\em root} and analyzing such a component as a translation $\Gamma=Au$  where $A$ is  now a complete subgraph of the Cayley graph   contained in $G_{-2}$ and containing the element $(0,+)=0$. If $u\in\Z^m$  the   matrix  $C_{Au}$  is obtained   from  $C_A$ by adding the scalar matrix $-u(\xi)= -(u,\xi)$  while $C_{A\tau}=-C_A$.

\begin{example} Consider the following complete subgraph containing $(0,+)$.
 $$A= \xymatrix{ (-e_1-e_2,-) \ar@{=}^{\ \ \ (-e_1-e_2)\tau}[r]  &(0,+)\ar@{->}^{e_1-e_2\ \ }[r]&(e_1-e_2,+) }.$$
 A translation by an element $(u,+)$ is hence 
 $$ A(u,+)=\xymatrix{ (-e_1-e_2-u,-) \ar@{=}^{\ \  \quad (-e_1-e_2)\tau}[r]  &(u,+)\ar@{->}^{e_1-e_2\ \quad }[r]&(e_1-e_2+u,+) }$$\smallskip
 so we get that the matrices associated to these graphs are:
$$ C_A=\begin{pmatrix}
-\xi_1-\xi_2&2\sqrt{\xi_1\xi_2}&0\\ \\
-2\sqrt{\xi_1\xi_2}&0&2\sqrt{\xi_1\xi_2}\\\\0&2\sqrt{\xi_1\xi_2}&\xi_2-\xi_1
\end{pmatrix},$$$$C_{Au}=\begin{pmatrix}
-\xi_1-\xi_2-u(\xi)&2\sqrt{\xi_1\xi_2}&0\\ \\
-2\sqrt{\xi_1\xi_2}&-u(\xi)&2\sqrt{\xi_1\xi_2}\\\\0&2\sqrt{\xi_1\xi_2}&\xi_2-\xi_1-u(\xi)
\end{pmatrix}$$
\end{example}
  In particular  we  have shown (cf. \cite{PP}, \S 9) that $A$  can be chosen among a finite number of graphs which we call {\em combinatorial}.
Note that we do not impose the compatibility constraint on  $A$ but only on its translations. It is convenient, in drawing the graphs to drop the labels on the edges since they can be deduced  from the vertices. In a combinatorial graph the color of a vertex is black if its mass is 0 and red if it is $-2$.  Then  in the vertices we drop the sign $\pm$, since this information can be deduced from the mass or from the parity (number of red edges)  of the path connecting the vertex with the root.  So the graph of the previous example will be denoted by:
$$A= \xymatrix{ -e_1-e_2 \ar@{=}[r]  &0\ar@{->}[r]&e_1-e_2 }.$$ 
Note that in all the combinatorial graphs the root is by convention set to $0$.

Let us show that:
\begin{lemma}\label{spr}
The characteristic polynomial of a matrix $C_A$ is in $\Z[\xi_1,\ldots,\xi_m,t]$  (the square roots disappear).
\end{lemma} \begin{proof}  
By definition the determinant of an $n\times n$ matrix with entries $a_{i,j}$ is the sum with sign, over all permutations $\sigma$ of the $n$ indices, of the products $a_{1,\sigma(1)} \ldots a_{n,\sigma(n)} $. It is convenient to rearrange this product  using the cycle structure of $\sigma$,  each cycle  $(i_1,\ldots,i_k)$  determines a factor $a_{i_1,i_2} \ldots a_{i_k,i_1} $. Let us show that in each of these factors the square roots disappear. In fact, if the cycle is reduced to a single element it corresponds to a diagonal entry, which has no roots. Otherwise it corresponds to a sequence of edges forming a closed path. Then, by the definitions and compatibility, one sees that each index appearing in the edges appears an even number of times in such a closed path, hence the claim follows from the formula $\pm 2\sqrt{\xi_i\xi_j}$ of the entry corresponding to each edge.
\end{proof} 
\subsubsection{Proof of Proposition \ref{ello}\label{EO}}
We are ready to prove Proposition \ref{ello}:
\begin{proof}
We proceed by induction  on the number $m$ of the parameters, for
$m=1$ the statement is trivial,   so assume the statement is true
for $m-1$  parameters. Let $\Gamma$ be one of the combinatorial
graphs,   $A  (\Gamma)$ the corresponding matrix and $  (a_1,
\ldots,  a_k)$ the vertices of $\Gamma$.

Let $\bar A$ be the  matrix obtained from $A  (\Gamma)$ by setting
$\xi_m=0$.    We claim that this matrix is  the one associated to
the not necessarily connected colored graph $\bar \Gamma$ in $m-1$
coordinates obtained by dropping  the last coordinate in all the
vertices $a_i$,   this is just a consequence of the definitions
(see \S \ref{LaMa}).

The first thing to be verified is that the vertices of $\bar \Gamma$ are all distinct (as colored vertices).  In fact given a vertex $a\in\Z^m$ let $\bar a\in\Z^{m-1}$ be the vertex  obtained by dropping  the last coordinate  $a_m$.   We  can reconstruct $a$ from $\bar a$  and its color   using the mass  since $\eta  (a)=\eta  (\bar a)+a_m$.

Now we claim that the graphs appearing give characteristic polynomials which are distinct, for this we apply Proposition \ref{seplem}. If we had two connected components of $\bar \Gamma$ giving the same characteristic polynomial we should have two elements  $\bar a$ black and $\bar b$ red so that $\bar b=\tau \bar a=-\bar a\tau$ red. We have  $a= \bar a -\eta(\bar a)e_m$  while $\tau \bar a=-\bar a\tau$  comes from $b= (-\bar a+(\eta(\bar a)-2)e_m)\tau=(-a-2e_m)\tau  $.     Thus in the
graph  $\Gamma$  we cannot have these two vertices,   since the
presence of two vertices     $b+a=-2e_m $ implies that the graph is not allowable by Definition \ref{ilpunto1}.

Now we apply the fact that we know that  all the blocks  appearing
in $A  (\bar\Gamma)$ are distinct and depend on $m-1$ variables, furthermore  two different blocks have different characteristic polynomials by the previous remark and Lemma \ref{seplem}.
From the hypotheses made there is an open region $\mathcal B_{m-1}$ in the
complement of the discriminant variety  for $m-1$ variables where for each of the finitely many combinatorial blocks
all the eigenvalues  are distinct and real.

Now this condition is stable so that for  $A  (\Gamma)$  there is a non empty open region complement of the discriminant variety  for  $A  (\Gamma)$ where all the eigenvalues are distinct and real containing $\mathcal B_{m-1}$, since we have finitely many combinatorial graphs $\Gamma$  we find an open component of the  complement of the discriminant variety for  all graphs $\Gamma$, containing $\mathcal B_{m-1}$,   where all the eigenvalues are  real. We  further remove the resultants  and have that they are also all distinct.

\end{proof}
\vfill \eject

\part{Sphere and hyperplanes problem} In order to understand the possible components of the graph $\Lambda_S$ we relate it to the geometric graph $\Gamma_S$.

\section{The geometric problem}

The condition  for two points $p,  q$ to be the vertices of an
edge  is   given by algebraic equations.   Visibly $p\in H_{i, j}$
means that $  (p-v_i,  v_i-v_j)=0$,   the corresponding
$q=p+v_j-v_i$,    while $p\in S_{i,  j}$  is given by $  (p-v_i,
p-v_j)=0$  and the corresponding opposite point $q$  is  given by
$p+q=v_i+v_j$.

We thus have two types of constraints  describing when two points
are joined by an edge,   a linear  $q-p=v_j-v_i$ or $p+q=v_i+v_j$
and a quadratic constraint $  (p-v_i,  v_i-v_j)=0$ or $  (p-v_i,
p-v_j)=0$.   The fact that a point $x$  belongs to a component described by the combinatorial  graph is  thus expressed by a list of linear and quadratic equations for $x$ deduced by eliminating all the other vertices using the linear constraints.

We describe the linear constraints again through a Cayley graph. The group $G$ also as linear operators on $\R^n$ by setting
\begin{equation}\label{azione}
a  k:=  -\pi  (a)+ k,  \ k\in\R^n,  \ a\in \Z^m\,  ,  \quad \tau
k= - k
\end{equation}
  We then have that \begin{remark}
$X$ defines also a Cayley graph on $\R^n$  and  in fact the graph  $\Gamma_S$ is a subgraph of this graph.
\end{remark}

\subsection{Equations for the root\label{gra}} From the very construction of the graph it is convenient to {\em mark}  the edges by $v_j-v_i$ in the first case and $v_j+v_i$ in the second (notice the sign change due to Formula \eqref{llab}).   In fact we use a more combinatorial way of marking which is illustrated in the next example.  It  is then clear that each connected  component of this  graph has  a  combinatorial description  which encodes  the information on the various types of edges which connect the vertices of the component.

The connection with the graph $\Lambda_S$   comes from the fact that these equations are exactly the ones which define compatible edges.
%\begin{example}   $$
%    \xymatrix{ &x-v_1+v_3\ar@{<- }[d] _{2,  1}& &\\  &x-v_2+v_3& &\\  x\ar@{ ->}[ru] _{3,  2}\ar@{<-}[ruu] ^{  1,3}\ar@{=}[rr]_{1,  2} && -x+v_1+v_2 \ar@{=}[luu]_{2,  3} \ar@{=}[lu] ^{1,  3} &   }    \xymatrix{ &x-v_1-v_2+v_4+v_3 \ar@{<- }[d] _{4,  1}& &\\  &x-v_2+v_3& &\\  x\ar@{ ->}[ru] _{3,  2}\ar@{=}[rr]_{1,  2} &&-x+v_1+v_2 \ar@{=}[luu]_{4,  3} \ar@{=}[lu] ^{1,  3} &   }
% $$
% the equations that $x$ has to satisfy  are:
%$$\begin{matrix}
%  (x,  v_2-v_3)=|v_2|^2-  (v_2,  v_3)&&&  (x,  v_2-v_3)=|v_2|^2-  (v_2,  v_3)\\
%|x|^2-  (x,  v_1+v_2)=-  (v_1,  v_2)&&&|x|^2+  (x,  v_1+v_2)=-  (v_1,  v_2)\\
%  (x,  v_1-v_3)=|v_1|^2-  (v_2,  v_3)&&&  (x,  v_1-v_2-v_3-v_4)-|v_1|^2 +  ( v_1,  v_2) +  ( v_1,  v_3)  \\&&&-
%   ( v_2,  v_3) +  ( v_1,  v_4) -
%  (  v_2,  v_4) -   (v_3,  v_4)
%\end{matrix} $$\end{example}
\begin{example}  The equations that $x$ has to satisfy  are:   $$
    \xymatrix{ &x-v_1+v_3\ar@{<- }[d] _{2,  1}& &\\  &x-v_2+v_3& &\\  x\ar@{ ->}[ru] _{3,  2}\ar@{<-}[ruu] ^{  1,3}\ar@{=}[rr]_{1,  2} && -x+v_1+v_2 \ar@{=}[luu]_{2,  3} \ar@{=}[lu] ^{1,  3} &   }   \begin{matrix}&\\&\\&\\
  (x,  v_2-v_3)=|v_2|^2-  (v_2,  v_3) \\&\\
|x|^2-  (x,  v_1+v_2)=-  (v_1,  v_2) \\&\\
  (x,  v_1-v_3)=|v_1|^2-  (v_2,  v_3) 
\end{matrix}
 $$\end{example}
%
%$$ \xymatrix{ &x-v_1-v_2+v_4+v_3 & &\\  &x-v_2+v_3& &\\  x\ar@{ ->}[ru] _{3,  2}\ar@{=}[rr]_{1,  2} &&-x+v_1+v_2 \ar@{=}[luu]_{4,  3} \ar@{=}[lu] ^{1,  3} &   }\hskip-3cm\begin{matrix}
%    (x,  v_2-v_3)=|v_2|^2-  (v_2,  v_3)\\
% |x|^2+  (x,  v_1+v_2)=-  (v_1,  v_2)\\
%     (x,  v_1-v_2-v_3-v_4)-|v_1|^2 +  ( v_1,  v_2) +  ( v_1,  v_3)  \\ -
%   ( v_2,  v_3) +  ( v_1,  v_4) -
%  (  v_2,  v_4) -   (v_3,  v_4)
%\end{matrix} $$\end{example}
In fact it should be clear that a graph in $\Gamma_S$ is obtained starting from a point $x$ and then applying the elements of a complete sub  graph $A\subset G_X$    of the Cayley graph
containing 0. One the results of \cite{PP}   (Theorem 3) is that in this fashion we have always isomorphisms between components of  $\Lambda_S$ and components of $\Gamma_S$. 

The question is thus    to understand when,   given $x\in \R^n$,
the elements $hx,  \ h\in A$  describe the vertices of a
corresponding geometric graph with {\em root} $x$ in $\Gamma_S$.

One can easily verify that
\begin{proposition} The elements $hx,  \ h\in A$  describe the vertices in a component $C$ of the   geometric graph   $\Gamma_S$ if and only if,
 for each $h=  (a,  \sigma)\in A$  we  have:
\begin{equation}\label{bacos}
\begin{cases}
    (x,  \pi  (a))= K  (h )\quad \text{if}\   \sigma=1\\
 |x|^2+  (x,  \pi  (a))=  K  ( h )\quad \text{if}\   \sigma=-1\end{cases}
. \end{equation}
 \end{proposition}
   Therefore the question that we have to address is: for which graphs $A\subset\Gamma_X$  we  can say that these equations    have a solution in $\R^n\setminus S$ for generic values of the points $v_i$? Such a graph is called {\em compatible}.
   
    A  main result in \cite{PP} is that if  the edges of the combinatorial graph span a lattice of dimension $> n$  then the only geometric realizations of this graph can be in the special component $S$.  
   
   It remains to analyze graphs  with linearly dependent edges.
 In order to address this question we need to develop a more combinatorial approach.
  \subsection{Relations}
Take a connected complete subgraph $A$, in the subgroup   $G_{-2}$ of $G$ generated by $X$, of
the Cayley graph $G_X$.    By taking the first coordinates we
identify its vertices with a subset,   still denoted by $A$,   of
the set of elements in $\Z^m$ with $\eta  (a)=0,  -2$ (the orbit of 0 under $G_{-2}$).
\begin{definition}\label{degga}\begin{itemize}
\item A  graph $A$ with $k+1$ vertices is  said to be of {\em dimension} $k$.

\item    We call the dimension  of the affine space spanned by $A$  in $\R^m$
the {\em rank},   ${\rm rk}\,   A$,    of the graph $A$.
\item If the rank of $A$ is strictly less than the dimension of $A$ we say that $A$  {\em is degenerate}.
\end{itemize}

\end{definition}

Once we choose a root $r$ for $A$ we can translate $A$  so that
$r=0$  then instead of the affine space spanned by $A$ we may
consider the lattice spanned by the non--zero elements in $A$, it
is natural to color all remaining vertices with the rule that a
vertex $a$ is {\em black} if $\eta  (a)=0$ or,   equivalently, it
is   joined to the root by an even path and {\em red} otherwise.
if $\eta  (a)=-2$. Then we can extend the notion  of {\em black }
or {\em red}  rank,   and corresponding degeneracy. When we change
the root we have a simple way of changing colors that we leave to
the reader and the two ranks may just be exchanged.

If $A$ is degenerate  then there are non trivial relations,
$\sum_an_a a =0,  \ n_a\in\Z$  among the elements $ a\in A$.
\begin{remark}
\label{maxt} It is also useful to choose a maximal tree $T$ in $\Gamma$.  There is   a triangular change of coordinates from the vertices $ a$ to the markings of $T$.  Hence the relation can be also expressed as a relation between these markings.
\end{remark}

 We must have by linearity,   for every relation  $\sum_an_a  a=0,  \ n_a\in\Z$ that $0=\sum_an_a a^{  (2)},  \ 0=\sum_an_a \pi  (a)$ and moreover  we have:\begin{equation}
0=\sum_{a,  \ |\,   \eta  (a)=-2}n_a.
\end{equation}
Applying Formula \eqref{bacos}  we deduce that we must have
\begin{equation}\label{riso}
\sum_an_aK  (g_a )=2  (x,  \sum_an_a \pi  (g_a))+[\sum_{a \,   |\,
\eta  (a)=-2}n_a]  (x)^2  =2  (x,  \sum_an_a \pi  (g_a))=0.
\end{equation}
 The expression $\sum_an_aK  (g_a)  $ is a linear combination with integer coefficients of the scalar products $  (v_i,  v_j)$.  We can prevent the occurrence of the component $\Gamma$ by imposing it as avoidable resonance.     We   need to formalize  the  setting.

  Let us use for the elements of $G$ in the subgroup $G_2$ just their coordinate  $a\in\Z^m,\ \eta(a)\in\{0,-2\}$.
Then we have $\sum_an_aK  ( a) =\pi  (\sum_an_aC  ( a) )$ hence we easily deduce:

  \begin{proposition}
  The equation   \eqref{riso}  is a non trivial constraint if and only if $\sum_an_aC  (g_a) \neq 0$. In this case    we say that the graph has an  {\em avoidable resonance}.

\end{proposition}
\begin{corollary}
If   we have an avoidable resonance of previous type associated to
$\Gamma$ then,   for a generic choice of the $S:=\{v_i\}$,
$\Gamma$  as   no geometric realizations.
\end{corollary}  The main Theorem on this topic proved in \cite{PP} is:
 \begin{theorem}\label{ridma}
Given  a compatible connected $X$--marked graph,   with a chosen
root and of rank $k$ for a  given color,   then either it has
exactly $k $  vertices of that color or it produces an avoidable
resonance.   \end{theorem}\begin{proof} Let us recall the proof
for convenience of our treatment. Assume by contradiction that  we
can choose   $k+1$ distinct vertices $  (a_0,  a_1,  \ldots,  a_k)
$,  different from 0 of the given color  so that we have a non
trivial relation   $\sum_in_i  a_i=0 $ and the elements $a_i,  \
i=1,  \ldots,  k$ are linearly independent. Set $n_a=n_i,   $ if
$a=a_i$ and $n_a=0$ otherwise. If all these vertices   have sign
$+$,   we have  $ \sum_an_a a^2=0$. Similarly,   if  they are
have sign $-$ we have  $-\sum_an_a a=\sum_an_a\sigma  (a)a=0$ and
also $ \sum_an_a a^{  (2)}=0$ so  again $\sum_an_aa  ^2=0$.

We can consider thus the elements $x_i:=a_i,  i=1,  \ldots,  k$ as
new variables  and then we write  the relations   $ \sum_an_a a
=\sum_an_a a^2=0$     as
$$0= a_{k+1} +\sum_{i=1}^kp_ix_i,  \implies   (\sum_{i=1}^kp_ix_i)^2+\sum_{i=1}^kp_ix_i^2=0 . $$
Now   $\sum_{i=1}^kp_ix_i^2 $ does not contain any mixed terms
$x_hx_k,  \ h\neq k$  therefore  this equation can be verified if
and only if  the sum $\sum_{i=1}^kp_ix_i$ is reduced to a single
term   $p_ix_i$,    and then we have $p_i=-1$ and $ a_0 = a_i$, a
contradiction.  \end{proof}

Unfortunately there are examples  of unavoidable resonances as we shall discuss in the next paragraph.
 
 \subsection{Degenerate resonant graphs}

\begin{definition} We say that a graph $A$ is {\em degenerate--resonant},   if it is degenerate and, for
all the possible linear relations $\sum_in_ia_i=0$  among its
vertices we have also $\sum_in_iC  (a_i)=0.$
\end{definition}

What we claim is that a  degenerate--resonant graph $A$ has no geometric realizations outside the special component.

\begin{remark}
One may easily verify that the previous condition,   although
expressed using a chosen root,   does not depend on the choice of
the root.
\end{remark}
One of the obstacles we have is that the proof of Theorem
\ref{ridma}  breaks down in general since in fact there are non
trivial degenerate--resonant graphs,   the simplest of them is the
{\em minigraph}
\begin{equation}
\label{mig} \xymatrix{   \ar@{-}[d]\ar@{=}[r]   (   -e_2+e_1)&    ( -2e_1 ) \ar@{-}[d] \\   \ar@{=}[r]0&     ( -e_2-e_1 )    }
\quad \xymatrix{ & \ar@{-}[d]\ar@{=}[r]   (   -e_2+e_1)+a&    ( -2e_1 )-a \ar@{-}[d] \\ & \ar@{=}[r]a&     ( -e_2-e_1 )-a    }
\end{equation}
Relation is $  (   -e_2+e_1)-    ( -e_2-e_1 ) +  ( -2e_1 )=0$,
we have
$$ C  (   -e_2+e_1)=e_1^2-e_1e_2,  \quad  C  ( -e_2-e_1 ) =-e_1e_2,  \quad C  ( -2e_1 )=-e_1^2$$ $$ e_1^2-e_1e_2-  (-e_1e_2)-e_1^2=0.$$
A more complex example is

 $$ \xymatrix{&&&e_2-e_3\ar@{-}[d]^{ e_2-e_3}&\\
 -3e_1+e_2\ar@{=}[r]^{-e_1-e_4}&2e_1-e_2-e_4\ar@{-}[r]^{ e_1-e_4}&e_1-e_2\ar@{-}[r]^{e_1-e_2}&0\ar@{= }[r]^{-e_2-e_3}&-e_2-e_3}.$$
What is common of these two examples is that in each there is a pair of vertices $a,b$, of distinct colors,  with $a+b=-2e_i$ for some index $i$.
\begin{definition}\label{ilpunto1}
We shall say that a connected graph  $G$ is {\em allowable}  if there is no pair of vertices $a,c\in G$   with $a c^{-1}=  c^{-1}a=(-2 e_i  ,\tau),$ or $   (-3 e_i+ e_j,\tau)$, otherwise it is {\em not allowable}.
\end{definition}
We may assume $a\in\Z^m$ black and $c=b\tau,\ b\in\Z^m$ red. We then easily see that 
\begin{proposition}\label{ilpunto0}
If a graph is not allowable then it has no geometric realization outside the special component (i.e. it is not compatible).\end{proposition}
\begin{proof}  
We write the quadratic equation \eqref{bacos}, for a vertex $x$, corresponding to the root $a$, given by the vertex $b=-2e_i$. Since $C(-2e_i)=-e_i^2,\ K(-2e_i)=-|v_i|^2$      we get $$0= |x|^2+(x,\pi(-2e_i))-K(-2e_i)=|x|^2-2(x, v_i )+|v_i|^2= |x-v_i|^2.$$  Hence the only real solution of $|x-v_i|^2=0$ is $x=v_i$. Then we apply  Remark 15 of \cite{PP}  where we have shown that the  special component is an isolated component of the graph.

In the other case $x$ is in a sphere whose square radius is  $\pi(A)$  
$$A=\frac{(-3e_i+e_j)^2}4+C(-3e_i+e_j)=-\frac 14[{(-3e_i+e_j)^2} +2{ (-3e_i^2+e_j^2) } ] $$
  $$ = -\frac 14[9e_i^2-6e_ie_j+e_j^2-6e_i^2+2e_j^2]= -\frac 34[e_i-e_j]^2$$
clearly $\pi(A)=-\frac 34|v_i-v_j|^2<0,\ \forall v_i\neq v_j.$

\end{proof}

  What we conjectured and shall prove in this paper is (cf. \S \ref{TR}):
\begin{theorem}\label{MM}
A degenerate--resonant graph $A$ is {\em not allowable} hence it has no geometric realizations outside the special component.
\end{theorem}
From this Theorem Proposition \ref{affind} follows.
%Recall  that our goal is to prove
% Theorem \ref{MM}, {\em
%A degenerate resonant graph  has no geometric realization outside the special component.
% }  This  follows immediately from
% \begin{proposition}\label{ilpunto}
%A degenerate resonant is not allowable.
%\end{proposition}
%%\begin{proof}[Proof of Theorem \ref{MM}] We write the quadratic equation \eqref{bacos}, for a vertex $x$, corresponding to the root $a$, given by the vertex $b=-2e_i$. Since $C(-2e_i)=-e_i^2,\ K(-2e_i)=-|v_i|^2$      we get $$0= |x|^2+(x,\pi(-2e_i))-K(-2e_i)=|x|^2-2(x, v_i )+|v_i|^2= |x-v_i|^2.$$  Hence the only real solution of $|x-v_i|^2=0$ is $x=v_i$. Then we apply  Remark 15 of \cite{PP}  where we have shown that the  special component is an isolated component of the graph.
%%\end{proof} So t
%This will be shown to be an immediate consequence of Proposition \ref{ilpunto}. The proof of this proposition occupies the first half of the paper.
%\begin{remark}
%One may easily verify that the previous condition,   although
%expressed using a chosen root,   does not depend on the choice of
%the root.
%\end{remark}
%

\section{Resonant graphs} \subsection{Encoding graphs} In order to understand relations,   consider   the complete graph $T_m$ on the vertices  $1,  \ldots,  m$.    If we are given a marked graph $\Gamma$ we associate to it   the subgraph $\Lambda$ of  $T_m$,   called its {\em encoding graph} in which we join the vertices $i,  j$ with a black edge  if $\Gamma$ contains an edge marked $e_j-e_i$  and by a    red edge edge if $\Gamma$ contains an edge marked $-e_j-e_i$. We mark $=$ the red edges.

 For each connected component of the encoding graph consider the subspace spanned by its edges. It is easily seen that  these subspaces form a direct sum. Hence the encoding graph of a minimal relation is connected. Moreover a circuit in the encoding graph corresponds to a relation  between the corresponding edges if and only if it contains an even number of red edges and we call it an {\em even circuit}.

This follows from the basic relations with which we can substitute  two consecutive edges with a single one:
$$ (e_i-e_j)+(e_j-e_k)+(e_k-e_i)=0,\quad \xymatrix{i\ar@{-}[rd]\ar@{-}[rr]&&k\\&j\ar@{-}[ru]}$$
$$ (e_i-e_j)-(-e_j-e_k)+(-e_k-e_i)=0,\quad \xymatrix{i\ar@{-}[rd]\ar@{=}[rr]&&k\\&j\ar@{=}[ru]}$$
$$ -2e_i=-(e_i-e_j)+(-e_i-e_j).$$
Thus    for each index $i$ of an odd circuit  a sum, with coefficients $\pm 1$, of its edges  equals to $-2e_i$. The edges of an even circuit have a linear  relation (unique up to sign) given by a sum with coefficients $\pm 1$ equal 0.
If we have a list of edges of $\Gamma$ which are linearly
dependent and minimal (with respect to this property) then we claim that the
corresponding elements in the encoding graph from a circuit,
with some provisos  due to the presence of red edges.
More precisely we may have a {\em simple circuit}  in which an even number of red edges appear or {\em two odd circuits} joined by a segment   (possibly reduced to a point).

\begin{example} An even and a doubly odd encoding graph:

\xymatrix{       & 1   \ar@{=}[r] &10\ar@{-}[r]    &9\ar@{-}[r]  &8\ar@{-}[r] &7    && \\   &2 \ar@{-}[u] \ar@{=}[r]   &3\ar@{ =}[r]    &4\ar@{-}[r]  &5\ar@{-}[r] &6\ar@{=}[u]      &&  }

\xymatrix{    &12\ar@{-}[r]&13\ar@{=}[r]&14\ar@{-}[r]&15\ar@{=}[r]&16&    && \\ & 11\ar@{-}[u] &   &19\ar@{-}[r]  \ar@{-}[u] &18\ar@{-}[r] &17\ar@{-}[u]     &&  \\   & 1 \ar@{-}[u]  \ar@{=}[r] &10\ar@{-}[r]    &9  \ar@{-}[r]  &8\ar@{=}[r] &7      && \\   &2 \ar@{-}[u] \ar@{=}[r]   &3\ar@{ =}[r]    &4\ar@{-}[r]  &5\ar@{-}[r] &6\ar@{=}[u]      &&  }
\end{example}
This can be easily justified.   Recall that the {\em valency} of a
vertex is the number of edges which admit it as vertex.  If the
given edges give a minimal relation their encoding graph must be
connected,   furthermore it cannot have any vertex of valency 1
since the corresponding edge is clearly linearly independent from
the others.   Finally it cannot have more than 2 simple circuits
otherwise we easily see that we have at least 2 relations.

 For a connected graph the number $c$ of independent circuits is the dimension of its first homology group and thus given,   using the Euler characteristic,   by $c=e-v+1$ where $e,  v$ are the number of edges and vertices respectively.   In our setting all vertices  have valency $\geq 2$ and we denote the valency of the vertex $i$ by $V_i=v_i+2$   (with $v_i\geq 0$).  We have $2e=\sum_i V_i=\sum_iv_i+2v$  so that we have  $\sum_iv_i=2c-2$.  If $c=1$ the encoding graph is a simple circuit. If $c=2$ we deduce that  $\sum_iv_i= 2$ hence we have either only one vertex of valency 4 and the others of valency 2 or  two vertices of valency 3 and the others of valency 2. The first case  gives two loops  joined in one vertex the second   gives  either two loops joined by a segment or two vertices joined by 3 segments.  This last case is not possible since two of these segments will have the same  parity and generate an even loop contradicting minimality.

\subsection{Minimal relations}We want to study a minimal degenerate resonant graph $\Gamma$.  Observe that for such a graph any proper subgraph is non-degenerate. In particular we have one and only one relation among the edges of a given maximal tree  $T$ in the graph and a corresponding relation for the vertices.

A minimal degenerate graph has a special type of relation which
comes from the fact that in a maximal tree we have a minimum
number of dependent edges.   Such a situation arises when these
edges,   call their set $\mathcal E$,   form  in the encoding
graph,   a {\em even circuit}    (where we allow the possibility
that we have two odd  circuits  matching) as in the previous
paragraph.   Call $|\mathcal E|$  the subgraph of $T$  formed by
the edges $\mathcal E$,   of course it need not be a priori
connected  but only a {\em forest} inside $T$.

In an even circuit the relation is a sum of edges
$\sum_j\delta_j\ell_j=0$,   with signs $\delta_j=\pm 1$  in two
odd matching circuits  we may have some $\delta_j=\pm 2$ corresponding to the
edges appearing in the segment connecting the two odd loops.   In
any case we list the edges appearing as $\ell_i$.  Each $\ell_i$
black is $\ell_i=a_i-b_i$  with $a_i,  b_i$,   its vertices of the
same color while a red is  $\ell_i=a_i+b_i$  with $a_i$ red and
$b_i$ black  its vertices.

The relation is thus
\begin{equation}
\label{leqed}\sum_{i\ \text{black}}\delta_i  (a_i-b_i)+\sum_{j\ \text{red}}\delta_j  (a_j+b_j)=0 .
\end{equation}

Notice that,   by minimality,   all the end points of $T$ must be
in $|\mathcal E|$. We may think of \eqref{leqed} as a formal
relation on the vertices   (instead of on the edges),   note that
a vertex in $\mathcal E$  need not appear in  \eqref{leqed}
however all end-points in $\mathcal E$ must appear and,   if a
vertex $v$ has coefficient $k$ in the relation,    it must be the
vertex of at least $k$ of the given edges   (in the case
$\delta_i=\pm 1$).
\subsubsection{Basic formulas \label{BaF}} We  work with $G_{-2}$  identified with elements in $\Z^m$  either with $\eta(a)=0$,   {\em black} or $\eta(a)=-2$   {\em red}.  We have set  $C(a)=\frac{1}{2}  (  a^2+  a^{  (2)})$ for $a$ black and  $C(a)=-\frac{1}{2}  (  a^2+  a^{  (2)})$ for $a$ red.

In our computations we use always the rules: \begin{itemize}
\item for $u,  v$ black, we have $u+v$ black and

$$ C  (u+v)=\frac{1}{2}  (  (u+v)^2+  (u+v)^{  (2)})=C  (u)+C  (v)+uv$$
\item for $u$ black $v$ red, we have $u+v$ red and

$$C  (u+v)=-\frac{1}{2}  (  (u+v)^2+  (u+v)^{  (2)})= -C  (u)+C  (v)-uv$$

\item for $u,  v$ red, we have $u-v$ black and
$$C  (u-v)=\frac{1}{2}  (  (u-v)^2+  (u-v)^{  (2)})=\frac{1}{2}  (  (u^2+v ^2-2uv+  (u-v)^{  (2)})$$$$=\frac{1}{2}  (  (u^2+v ^2-2uv+  (u-v)^{  (2)})=-C  (u)+C
(v)+v^2-uv$$
\item for $u$ black, we have $-u$ black and
$$C  (-u)=C  (u)-u^{  (2)}. $$\end{itemize}

 \section{The resonance \label{TR}}
 \subsection{The resonance relation}

This chapter is devoted to the proof of Theorem \ref{MM}. In order to prove it we take a minimal  degenerate resonant graph $\Gamma$ and inside it a maximal tree $T$ and then we start studying it. In fact it would be possible to classify these trees, we arrive a little short of this  since we need only to show \ref{MM}.
\subsubsection{Relations}
   Associated to $T$  we have its encoding graph and the encoding graph of the edges $\mathcal E$ involved in the relation. We index the edges in the relation and set
   $\ell_i=\vartheta_i e_i-e_{i+1}$ where $\vartheta_i=\pm 1$ (depending on the color of the  edge). As we explain in course of the proofs we will need to identify some vertices $e_i$.

 We  distinguish two cases, if the encoding graph of the relation is 1) an even or  2) a doubly odd loop. The simplest case to treat is case 1) which then suggests how to deal with the other cases. \smallskip

Case 1.  Up to changing notations we may assume that the loop is formed by the edges
 $\ell_i=\vartheta_i e_i-e_{i+1},  \ i=1,  \ldots,    k-1,  \ \ell_{ k}=\vartheta_{ k} e_{ k}-e_1,$ (here we identify $e_1=e_{k+1}$).  Set  $$\delta_i:=\prod_{j\leq i}\vartheta_i=\vartheta_i\delta_{i-1},$$ we assume we have an even number of $\vartheta_i=-1$, by assumption $\delta_k=1$.

 We call $\delta_i$ the {\em parity} of $i$.
 \begin{lemma}
We have the relation:
$$A=\sum_{i=1}^{ k}  \delta_i \ell_i=0.$$
\begin{proof}
Consider an index $ i>1$, the coefficient of $e_i$ in $A$ is  $-\delta_{i-1}+\delta_i\vartheta_i$.  Since $\delta_i=\delta_{i-1} \vartheta_i$  for this $e_i$ the coefficient is 0.  For $e_1$   the coefficient  comes from $\delta_1\ell_1+\delta_k\ell_k$, we have $\delta_1=\vartheta_1,\delta_k=1$  so we also get coefficient 0.
\end{proof}
\end{lemma} Set $\zeta:\Z^m\to\Z,  \ \zeta  (e_i)=\delta_{i-1}$  (by convention $\delta_0=1$)  so  that, by linearity, $\zeta(\ell_i)=\vartheta_i\delta_{i-1}-\delta_{i}=0.$
\begin{lemma}\label{zita}
The $\ell_i$  span the codimension 1 subspace of the space $e_1,
\ldots,  e_{ k}$  formed by the vectors $a$ such that \begin{equation}
\label{par}a=\sum_i\alpha_ie_i \ |\,   \zeta  (a)=\sum_i
 \delta_{i-1}\alpha_i=0.
\end{equation}
\end{lemma}\begin{proof}
$\zeta(\ell_i) =0$, so the $\ell_i$ are in this subspace, but they span a subspace of codimension 1 hence the claim.
\end{proof}

Case 2.  For a double  loop with $k$ edges, we have  either one or two  vertices in the encoding graph of valency $>2$ separating the two odd loops, we call these vertices {\em critical}. We start from a odd loop and a critical vertex  which we may assume to be 1. We call $A=\{1,\ldots,h\}$ the indices in the first loop. We then list the edges $\ell_1,\ldots,\ell_h$  in a circular way and 
\begin{lemma}  We may choose the signs  $ \delta_i=\pm 1$ so that  for any index $j\leq h$ we have:
\begin{equation} 
\label{sopar}\sum_{i=1}^j\delta_i\ell_i=-\delta_j e_{j+1}-e_1,\quad \sum_{i=1}^h\delta_i\ell_i=-2e_1,\quad \sum_{i=j+1}^{h}\delta_i\ell_i=  - e_1+\delta_j e_{j+1}.
\end{equation}
\end{lemma}
\begin{proof}  From the first Formula the others follow.  We define $\delta_i=\vartheta_i\delta_{i-1}$  if $i>1$ and set $\delta_1=-\vartheta_1$. Then 
if $j=1$, $\delta_1\ell_1=\delta_1(\vartheta_1e_1-e_2)=-e_1-\delta_1e_2$  and  this follows from the definitions. By induction
$$\sum_{i=1}^{j+1}\delta_i\ell_i=-\delta_j e_{j+1}-e_1+\delta_{j+1}\ell_{j+1}$$$$-\delta_j e_{j+1}-e_1+\delta_{j+1} \vartheta_{j+1}e_{j+1}-\delta_{j+1}e_{j+2} = -e_1        -\delta_{j+1}e_{j+2} .$$
\end{proof} 

For notational convenience we identify $e_{h+1}=e_1$. If we have two critical vertices, call $b\geq h+2$ the other, we have then a segment joining them formed by a string of elements $\ell_{h+1}=\vartheta_{h+1} e_{h+1}-e_{h+2},\ldots,\ell_{b-1}=\vartheta_{b-1} e_{b-1}-e_b$. We call $B$ this set of indices and assign to these edges{ \em signs} $\delta=\pm 2$    so that  $\sum_{i=h+1}^{b-1}  \delta_i\ell_i=\sum_{i\in B}  \delta_i\ell_i=2[e_1+\vartheta e_{b}]  $  where $\vartheta=1$  if and only if this segment is odd. 

We finish with the other odd loop, call $C$ the corresponding set of indices and assign, as before,  signs $\pm 1$ so that  $\sum_{i=b}^k\delta_i\ell_i=\sum_{i\in C}\delta_i\ell_i= -2\vartheta e_{b}$. With these choices the relation is
\begin{equation}
\label{LAR}R:=\sum_{i=1}^k\delta_i\ell_i =-2e_1+2[e_1+\vartheta e_{b+1}] -2\vartheta e_{b+1}=0.
\end{equation} We have chosen the indices so that   we order the edges   as they occur in one way of {\em walking} the cycle, starting from the critical vertex 1. We say that an index is critical if the corresponding vertex is critical. Here $1,h+1,b$ are critical.

\begin{remark}\label{crii}
The  non critical indices are divided in 2 or 3  sets (depending if we have only one critical vertex or two).
If $ u$  is not critical  we have $\delta_u=\vartheta_u\delta_{u-1}$.
\end{remark}

\medskip

\begin{lemma}\label{sublatti}
The $\ell_i$  span the sublattice of the lattice spanned by  $e_1,
\ldots,  e_{ k}$  formed by those vectors \begin{equation}
\label{par1}a=\sum_i\alpha_ie_i \ |\,   \piff  (a)=\sum_i
  \alpha_i\cong 0,\ \text{modulo}\ 2.
\end{equation}
\end{lemma}\begin{proof}
$\piff(\ell_i) \cong 0$ modulo 2, so the $\ell_i$ are in this sub--lattice,  the fact that they span  is easily seen by induction.
\end{proof}
\subsubsection{Signs}
We choose a root  $r$  in $T$  and then each vertex $x$ acquires a color  $\sigma_x=\pm1$. The color of $x$ is   red and $\sigma_x=-1$ if the path  from the root to $x$ has an odd number of red edges, the color is black and $\sigma_x=1$ if  the path  is even.

An edge $\ell_i$ is connected to the root $r$ by a unique path $p_i$ ending with $\ell_i$ we denote its final vertex $x_i$ and we set  $\sigma_i:=\sigma_{x_i}$.  If $\ell_i$ is black we set  $\lambda_i=1$ if  the edge is equioriented with the path, that is it points outwards, $\lambda_i=-1$ if it points inwards. Finally we set $\lambda_i=1$ if  the edge is red.
\begin{equation}
   \xymatrix{   r    \ar@{.}[r]  & \ldots&  \ar@{->}[r]^{\ell_i}  & x    }\quad  \lambda_i=1,\qquad  \xymatrix{   r    \ar@{.}[r]  & \ldots&  \ar@{<-}[r] ^{\ell_i}  & x    }\quad  \lambda_i=-1.
\end{equation}
\begin{definition}
Once we choose a root in $T$, each red edge $\ell_i$  (that is $\vartheta_i=-1$) appears as edge with one
end denoted by $a_i$ red and the other denoted by $b_i$ black,   we have $\ell_i=a_i+b_i$. For a black edge  $\vartheta_i=1$  we define $a_i,b_i$ so that  instead $a_i=b_i+\ell_i$, and $a_i,b_i$ have the same color. We thus write $\ell_i=a_i-\vartheta_ib_i$.
\end{definition}

In particular for the resonant trees:
\begin{proposition}
 \begin{equation}\mathcal R:=
\label{retr} \sum_{i\,|\,\vartheta_i=-1} \delta_i   ( -a_i^{  (2)}- \ell_i a_i +e_ie_{i+1} )+\sum_{i\,|\,\vartheta_i= 1}  \delta_i \sigma_i( -e_{i+1}^2+ e_ie_{i+1}+\ell_i a_i )=0.
\end{equation}
$$\sum_{i\,|\,\vartheta_i=-1} \delta_i   ( b_i^{  (2)} +\ell_ib_i  -e_ie_{i+1} )+\sum_{i\,|\,\vartheta_i= 1}  \delta_i \sigma_i(e_i ^2-e_ie_{i+1}+\ell_i  b_i )=0 $$
\end{proposition}
\begin{proof}
We start from the relation $\sum_i\delta_i\ell_i=0$ and substitute the previous formulas, we deduce
\begin{equation}
\label{retr0} R:=\sum_{i\,|\,\vartheta_i=-1} \delta_i   ( a_i+b_i )+\sum_{i\,|\,\vartheta_i= 1}  \delta_i  (a_i-b_i)=0.
\end{equation}
We next  have by the resonance hypothesis $$\sum_{i\,|\,\vartheta_i=-1} \delta_i   ( C(a_i)+C(b_i) )+\sum_{i\,|\,\vartheta_i= 1}  \delta_i  (C(a_i)-C(b_i))=0.$$
We next apply the formulas \ref{BaF}.

For $a_i,    \ell_i=-e_i-e_{i+1}$ red, we have $b_i+a_i=  \ell_i$  and $b_i$ is black:

$$C(a_i)+C(b_i)=-1/2(a_i^2+a_i^{(2)})+1/2(b_i^2+b_i^{(2)})$$$$=-1/2(a_i^2+a_i^{(2)})+1/2((\ell_i-a_i)^2 +\ell_i^{(2)}-a_i^{(2)})$$
$$=-1/2(a_i^2+a_i^{(2)})+1/2( \ell_i^2-2\ell_ia_i+a_i^2  +\ell_i^{(2)}-a_i^{(2)})$$
$$=-a_i^{(2)}  -\ell_ia_i+e_ie_{i+1}=-a_i^{(2)}  -\ell_ia_i+e_ie_{i+1}.$$
 For $a_i=b_i+\ell_i$ and $\ell_i=e_i-e_{i+1}$ black we have:
 $$C(a_i)-C(b_i)=\sigma_i[1/2(a_i^2+a_i^{(2)})-1/2(b_i^2+b_i^{(2)})]$$$$=\sigma_i[1/2(a_i^2+a_i^{(2)})-1/2((a_i-\ell_i)^2 -\ell_i^{(2)}+a_i^{(2)})$$
$$=\sigma_i[ -1/2( \ell_i ^2 -2\ell_ia_i-\ell_i^{(2)} )]=\sigma_i[ -e_{i+1}^2+ e_ie_{i+1} +\ell_ia_i   ].$$ The second identity follows from the first by substituting. \end{proof}

 \subsubsection{Some reductions}
Denote by $b_i=\sum_{h=1}^mb_{i,  h}e_h$ and expand the second Formula
\eqref{retr}. Observe that  the coefficients  of the mixed terms  $e_ie_j,\ i\neq j$   come all  from the sum

$$B:=\sum_{i\,|\,\vartheta_i=-1}   \delta_i   (   \ell_i b_i -e_ie_{i+1} )+\sum_{i\,|\,\vartheta_i= 1} \delta_i  \sigma_i( -e_ie_{i+1}+\ell_i b_i ).
$$ 

If $h\notin[1,  \ldots,   k],   $  the
coefficient of $e_h$ in $B$  (which must be equal to 0)  is
$$ \sum_{i\,|\,\vartheta_i=-1}   \delta_i       \ell_i b_{i,  h} +\sum_{i\,|\,\vartheta_i= 1} \delta_i  \sigma_i  \ell_i b_{i,  h} =0.
$$ 
By the uniqueness of the relation it follows that this relation is a multiple of \eqref{LAR} hence the numbers $b_{i,  h},\ \vartheta_i=-1$ and $ \sigma_i    b_{i,  h},\ \vartheta_i= 1$ are all equal.  Since  now we can choose as root one of the elements $b_i$  we deduce that all these coefficients $b_{i,  h}$ equal to 0. Thus, with this choice of root, $b_i,a_i$ have support in the vertices of the encoding graph.

 \smallskip

As a consequence we claim that:
\begin{lemma}
In case 2) the edges of the tree coincide with the edges $\ell_i$ of the relation.

In case 1) either the edges of the tree coincide with the edges $\ell_i$ of the relation or we can reduce to the case   in which
the tree $T$  consists only of the   edges involved in the
relation,   plus a single special extra edge $E$ with $\zeta(E)=2$ (see Lemma \ref{zita} for the definition of $\zeta$).

$E$ is    either a red edge of the form $-e_i-e_j$ with
$i,  j$ of the same value of $\zeta$ or  a black edge of the form $-e_i+e_j$
with $i,  j$ of the opposite  value of $\zeta$.

\end{lemma}
\begin{proof}
Let $T'$ be the forest  support of the edges $\ell_i$,   if this is a tree it must coincide with $T$ by minimality and we are done,   if $T'$ is not a tree there is at least one segment $S$ in $T$  joining two end points in $T'$.   All the edges in $S$ by definition are not in the relation.   Their  sum with suitable signs is supported in $[1,  \ldots,   k]$ and in fact it is either the sum or the difference of two of the elements  $a_i,  b_j$,   in particular it has the form $E=\sum_{i=1}^k\alpha_ie_i$.

If we are in case 2) then, by Lemma \ref{sublatti}, $2E$ is a linear combination of the $\ell_i$ with integer coefficients.  This is a new relation containing edges not supported in $[1,  \ldots,   k]$  contradicting the hypotheses.

If we are in case 1)  we must have $\zeta  (E)\neq 0$ otherwise $E$ is in the span of the edges $\ell_i$ and we have another relation among the edges of $T$ contradicting minimality.   By the  same reason  we cannot have two such segments,   since the $\ell_i$ span a subspace of codimension 1 and we still would have a new relation.

Finally we claim that $E$ is an edge.  
\smallskip

We look at the encoding graph  $U$ of the  edges in $S$,    we
want to show that they form a path  joining two points in $[1,  2,
\ldots,  k]$ so that the  loop they generate in this way is odd.

First remark that every end vertex  of  $U$  appears with non zero coefficient  $\pm 1$  in the vector $E$ hence all end points of $U$ lie in  $[1,  2,
\ldots,  k]$.

Next if $U$ contains   two different paths  joining points in
$[1,  2,  \ldots,  k]$ each such path gives rise by summing with suitable signs to a non--zero linear combination of elements in $[1,  2,  \ldots,  k]$.  Since the span of  the edges $\ell_i$ has codimension 1 in the span of the elements $e_i$,  if we have two more paths  we deduce a new relation.   We deduce that $U$ is either a  single path joining two vertices $u,v\in [1,  2,
\ldots,  k]$ and not meeting any other point of $[1,  2,
\ldots,  k]$  or it may also be a single loop  originating from a vertex  $u$  in $[1,  2,
\ldots,  k]$.
In this case the loop must be odd otherwise we have another relation, then we see that if we choose as root one of the two vertices of $T$ joined by $S$ the other vertex is $-2e_u$ and we are finished, since we have proved that the graph is not allowable i.e. we found the desired pair of Proposition \ref{ilpunto0}.

Otherwise  $E$ is an element of mass  either 0 or $-2$     has support in two elements of $[1,  \ldots,  k]$  with coefficients $\pm 1$ hence it is an edge,   since we are
assuming that it does not appear  in the relation the only
possibility is that  it must be of the form $e_u-e_v,   -e_u-e_v.
\ u,  v\in [1,  2,  \ldots,  k]$ and linearly independent from
the edges $\ell_h$,   this means,   by Formula \eqref{par1}, that
$u,  v$  must have opposite parity in the first case and the same
parity in the second. If $S$ is not equal to the edge $E$  we claim that  the complete graph $\Gamma$ we started from was not minimal. 
Indeed we construct a tree $T'$  by replacing the  path $S$  by the single  edge $E$. This is  a proper subgraph of $\Gamma$ by completeness.  The complete graph associated to $T'$ is  resonant-degenerate (it contains  all the vertices appearing in the relation). This is a contradiction.

$$ I)\quad\xymatrix{       & 1    \ar@{=}[r] &10\ar@{=}[r]    &9  \ar@{-}[r]  &8\ar@{=}[r] &7      && \\   &2 \ar@{=}[u] \ar@{=}[r]   &3\ar@{ =}[r] \ar@{ ->}[u]    &4\ar@{-}[r]  &5\ar@{=}[r] &6\ar@{=}[u]      &&  }$$
\begin{equation}
\label{dude}II)\quad \xymatrix{       & 1    \ar@{=}[r] &10\ar@{=}[r]    &9  \ar@{=}[r]  &8\ar@{-}[r] &7      && \\   &2 \ar@{=}[u] \ar@{=}[r]   &3\ar@{ -}[r] \ar@{ =}[ru]    &4\ar@{=}[r]  &5\ar@{=}[r] &6\ar@{=}[u]      &&  }\end{equation}

\end{proof}
\begin{remark}\label{crii2}
In the case 1) with an extra edge joining the indices $i,j$ we shall say that $i,j$ are critical and divide accordingly the remaining indices in two sets and all edges in two sets $A,B$ accordingly. Note that in this case for all indices one has $\delta_u= \vartheta_u \delta_{u-1}$.
\end{remark} 
\begin{remark}\label{divido}
In case 2) we divide the edges in three sets $A,B,C$ where $A$ are the edges of the first loop, $B$ (possibly empty) the edges of the segment and $C$ the ones of the second loop.
In case 1) with an extra edge we divide the edges in two sets $A$,$B$ separated by the extra edge $E$.

As for a non critical index $u$ we shall say that $u\in A$  resp. $u\in B,C$ if the two edges $\ell_{u-1},\ell_u$ are in $A$ (resp. $B,C$).
\end{remark}
\subsubsection{Some geometry of trees}  Let us collect some generalities which will be used in the course of the proof.   In all this section $T$ will be a tree,   for the moment with no further structure and later related to the Cayley graph.

Given a set $A$ of edges in $T$ let us denote  by $\langle
A\rangle$ the minimal tree contained in $T$ and containing $A$,
we call it the {\em tree generated by $A$}.  The simplest trees are the {\em segments} in which no vertex has valency $>2$. In fact in a segment we have exactly two end points of valency 1 and the {\em interior points} of valency 2.

\begin{lemma}
\label{SEP} 1)\quad If $A$ consists of 2 edges then  $\langle
A\rangle$  is a segment,   more generally if  $A$ consists of 2
segments $S_1,  S_2$ with the interior vertices of valency 2 then
again $\langle A\rangle$  is a segment,  if moreover $S_1\cap S_2$
contains an edge,   then  $S_1\cup S_2 =\langle  S_1,
S_2\rangle$ and all its interior vertices have valency 2.
\smallskip

 \quad If we only assume that $ S_2$ has interior vertices of valency 2   but we also assume that $S_1\cap S_2$ contains at least one edge then

2)    $\langle S_1,  S_2\rangle=S_1\cup S_2$ and it  is a segment.   
\end{lemma}\begin{proof}
1)\quad  Consider $S_1\cap S_2$,   if this is empty,   there is a
unique segment joining two  points in $S_1,  S_2$ and disjoint from
them,   then this must join two end points by the hypothesis on
the valency and the statement is clear.
\smallskip

2)\quad  Let $A$  be a segment connected component of  $S_1\cap
S_2$.   Unless $S_2\subset S_1$  one of the end points $a$ of $A$
is an internal vertex of $S_1$,   since this has valency 2 this is
possible only if  $a$  is an end point of $S_1$,   if also the
other end point of $A$ is an internal vertex of $S_1$ the same
argument shows that $S_1\subset S_2$.  The final case is that the
other end of $A$  is  also an end point of $S_2$ and then the
statement is clear. \end{proof}
 
\section{The contribution of an index $u$}
  \subsubsection{The strategy}
We want to exploit Formula \eqref{retr} in order to understand  the graph. We proceed as follows.  \begin{definition}
Given a quadratic expression $Q$ in the elements $e_i$ and any index $u$  we set $e_uC_u(Q)$  to be the  sum of all terms in $Q$ which contain $e_u$ but not $e_u^2$.
\end{definition}
Notice that $C_u$  is a linear map from quadratic expressions to linear expressions in the $e_i,\ i\neq u$. By Formula \eqref{retr} we    have  $C_u(\mathcal R)=0$.  We observe that only the terms $\ell_ia_i$ or $-e_ie_{i+1}$  may contribute to  $C_u(\mathcal R)$ hence:
$$C_u(\mathcal R)=\sum_{i\,|\,\vartheta_i=-1} \delta_i   (-C_u( \ell_i a_i )+C_u(e_ie_{i+1}) )+\sum_{i\,|\,\vartheta_i= 1}  \delta_i \sigma_i( C_u(e_ie_{i+1})+C_u(\ell_i a_i ))=0.$$

 We choose an index $u$  which appears only  in $\ell_{u-1}=\vartheta_{u-1}e_{u-1}-e_u$ and  in $\ell_{u }=\vartheta_{u }e_{u }-e_{u+1}$. This is any index in case 1) with no extra edge while it excludes the  {\em critical indices } in the other cases (see Remarks \ref{crii} and \ref{crii2}). 

We separately compute the contributions of 
$$-\mathcal R':=\sum_{i\,|\,\vartheta_i=-1} \delta_i    e_ie_{i+1}  +\sum_{i\,|\,\vartheta_i= 1}  \delta_i \sigma_i e_ie_{i+1}   ,\quad \mathcal R'':=-\sum_{i\,|\,\vartheta_i=-1} \delta_i     \ell_i a_i + \sum_{i\,|\,\vartheta_i= 1}  \delta_i \sigma_i  \ell_i a_i\,,$$
since
$C_u(\mathcal R)= C_u(\mathcal R'')- C_u(\mathcal R')$.

We need the following formula for the elements $a_j$, easily proved by induction, where the black edges $\ell$ are oriented outwards from the root and $\sigma_\ell$ denotes the color of the endpoint of the segment ending with $\ell$:
 
%\begin{equation} 
%\label{La}a_j=\begin{cases}\begin{matrix}
%-\sum_{\ell\preceq \ell_j}\sigma_\ell\ell  ,&\sigma_j=-1, &\ell_j\ \text{red}\quad \\
% - \sum_{\ell\prec  \ell_j}\sigma_\ell\ell  ,&\sigma_j= 1,\ \ &\ell_j\ \text{red}\quad \\
%\sigma_j\sum_{\ell\preceq \ell_j}\sigma_\ell\ell , &\lambda_j= 1,\ \ &\ell_j\ \text{black}\\
%\sigma_j\sum_{\ell\prec  \ell_j}\sigma_\ell\ell , &\lambda_j= -1, &\ell_j\ \text{black} 
%\end{matrix}
%\end{cases}
%,\quad a'_j:=\begin{cases}
% - \sum_{\ell\prec  \ell_j}\sigma_\ell\ell  ,\ \ \ell_j\ \text{red}\\
%\sigma_j\sum_{\ell\prec  \ell_j}\sigma_\ell\ell , \ \ \ell_j\ \text{black}
%\end{cases}
%\end{equation}  LA CAMBIO 

\begin{equation} 
\label{La}a_j=\begin{cases}\begin{matrix}
-\sum_{\ell\preceq \ell_j}\sigma_\ell\ell  ,&\sigma_j=-1, &\ell_j\ \text{red}\quad \\
 - \sum_{\ell\prec  \ell_j}\sigma_\ell\ell  ,&\sigma_j= 1,\ \ &\ell_j\ \text{red}\quad \\
\sigma_j\sum_{\ell\preceq \ell_j}\sigma_\ell\ell , &\lambda_j= 1,\ \ &\ell_j\ \text{black}\\
\sigma_j\sum_{\ell\prec  \ell_j}\sigma_\ell\ell , &\lambda_j= -1, &\ell_j\ \text{black} 
\end{matrix}
\end{cases}
\end{equation}

If $i\neq u-1,u$ set $\mu_u(i)$  to be the coefficient of $e_u$ in $a_i$ then
\begin{lemma}
If $i\neq u-1,u$ we have $C_u(\ell_i a_i )=\mu_u(i)\ell_i.$

The contribution $C_u(\mathcal R')$ depends on the two colors of $\ell_{u-1},\ell_u$ according to the following table (see Remarks \ref{crii}, \ref{crii2}) :
\begin{equation}\label{Rp}
\begin{matrix}
colors&&& contribution\ of\ \mathcal R'\\
rr  &\delta_{u-1}=-\delta_{u}&&-\delta_{u-1}e_{u-1}-\delta_{u }e_{u +1}= \delta_{u }[e_{u-1}- e_{u+1 }]\\
rb    &\delta_{u-1}=\delta_{u}\,\quad &&-\delta_{u-1}e_{u-1}-\sigma_u\delta_{u }e_{u +1}\ \ =-\delta_{u }[e_{u-1}+\sigma_ue_{u +1}]\\
br  &\delta_{u-1}=-\delta_{u}&& -\delta_{u-1}\sigma_{u-1}e_{u-1}- \delta_{u }e_{u +1}= \delta_{u}[\sigma_{u-1}e_{u-1}-  e_{u +1}]\\
bb   &\delta_{u-1}=\delta_{u}\, \quad && -\delta_{u-1}\sigma_{u-1}e_{u-1}- \sigma_u\delta_{u }e_{u +1}\ \ =-\delta_{u }[\sigma_{u-1}e_{u-1}+\sigma_{u }e_{u+1 }]
\end{matrix}
\end{equation}
\end{lemma}
\begin{proof}
The first statement is clear since the edge $\ell_i$ does not contain the term $e_u$.  For the second  we see  that the contribution  to $C_u(\mathcal R')$ comes from the two terms $e_{u-1}e_u,\ e_ue_{u+1}$.  The term  $e_{u-1}e_u$ if  $\theta_{u-1}=-1$, i.e. $\ell_{u-1}$ is red,   appears from $C_u(-\delta_{u-1} e_{u-1}e_u)=-\delta_{u-1} e_{u-1}$. If $\theta_{u-1}= 1$, i.e. $\ell_{u-1}$ is black,     appears from $C_u(-\sigma_{u-1}\delta_{u-1} e_{u-1}e_u)=-\sigma_{u-1}\delta_{u-1} e_{u-1}$.

The term  $e_ue_{u+1}$, if  $\theta_{u }=-1$, i.e. $\ell_{u }$ is red,   gives rise to $C_u(-\delta_{u } e_ue_{u+1})=-\delta_{u } e_{u+1}$ if   $\theta_{u }= 1$, i.e. $\ell_{u }$ is black,   gives rise to $C_u(-\sigma_{u }\delta_{u } e_ue_{u+1})=-\sigma_{u }\delta_{u } e_{u+1}$.

We then use the fact that $\delta_u=\delta_{u-1}$ if $\delta_u$ is black, while $\delta_u=-\delta_{u-1}$ if $\delta_u$ is red.

\end{proof}
We thus write
$$0=-C_u(\mathcal R)= \sum_{i\,|\,\vartheta_i=-1,\ i\neq u-1,u} \delta_i \mu_u(i)\ell_i  -\sum_{i\,|\,\vartheta_i= 1,\ i\neq u-1,u}   \delta_i \sigma_i  \mu_u(i)\ell_i+L_u$$ where $L_u$  is the contribution from $C_u(\mathcal R')  $ and    from the terms associated to $a_{u-1}\ell_{u-1}, a_u\ell_u  $.

We now choose the root so that the segment $S_u$, generated by the two edges $\ell_{u-1}, \ell_u $,  appears as follows:
\begin{equation}
\label{esu}S_u:=\quad  \xymatrix{   r    \ar@{-}[r]^{\ell_u} & \ldots&  \ar@{-}[r]^{\ell_{u-1}}  & x_{u-1}   } .
\end{equation}

   The value of $L_u$ depends upon 3 facts, 1) the two colors of  $\ell_{u-1},\ell_u $. 2) The orientation $\lambda$  of the edges  $\ell_{u-1},\ell_u $ which are black. 3)  The color $\sigma_{u-1}$ of  $x_{u-1}$.  We thus obtain 18 different cases described in \S \ref{18}.

\subsubsection{The contribution of  $ a_u\ell_u  $ }
   If $\ell_u=-e_u-e_{u+1}$ is red we have $a_u=\ell_u$ and $C_u(\delta_u\ell_ua_u)=2\delta_ue_{u+1}$.  If  $\ell_u=e_u-e_{u+1}$ is black we have $\sigma_u=1$, if $\lambda_u=1$ we have $a_u=\ell_u$ and $C_u(-\delta_u\sigma_u\ell_ua_u)=2\delta_u e_{u+1}$. If $\lambda_u=-1$ we have $a_u=0$ and $C_u(-\delta_u\sigma_u\ell_ua_u)=0.$  Summarizing:\smallskip

   \begin{equation}
\label{ellu} \begin{matrix}
\quad\quad  C_u(\delta_u\ell_ua_u)=2\delta_ue_{u+1},\quad &\ell_u\quad \text{is red}\quad&\\
C_u(-\delta_u\sigma_u\ell_ua_u)=2\delta_u e_{u+1},\quad &\ell_u\quad \text{is black}& \lambda_u=1\quad\\
C_u(-\delta_u\sigma_u\ell_ua_u)=0,\quad\quad \quad \quad &\ell_u\quad \text{is black}& \lambda_u=-1
\end{matrix}\end{equation}
\subsubsection{The contribution of  $ a_{u-1}\ell_{u-1} $ }
   In $a_{u-1}$ consider the part $\bar a_{u-1}$ of the sum formed by the edges $\ell_i,\ \ell_u\prec\ell_i\prec\ell_{u-1}$.

 We have $a_{u-1}=\bar a_{u-1}+\tilde a_{u-1}$ where

\begin{equation}
\label{La1}\tilde a_{u-1}=\begin{cases}\begin{matrix}
-\sigma_{u}\lambda_u\ell_u  +\ell_{u-1}  ,\ &\text{if}\ \sigma_{u-1}=-1,\quad  &\ell_{u-1}\ \quad\text{red}\quad\\
-\sigma_{u}\lambda_u\ell_u  ,\quad\ &\text{if}\ \sigma_{u-1}= 1,\ \quad&\ell_{u-1}\  \quad\text{red}\quad\\
\sigma_{u-1}\sigma_{u}\lambda_u\ell_u  +\ell_{u-1}  ,\quad\ &\text{if}\  \lambda_{u-1}= 1,\ \quad&\ell_{u-1}\  \quad\text{black}\\
\sigma_{u-1}\sigma_{u}\lambda_u\ell_u  ,\quad\ &\text{if}\  \lambda_{u-1}= -1,\ &\ell_{u-1}\   \quad\text{black}\\
\end{matrix}
\end{cases}
\end{equation}
 We then have
 $$ C_u(\ell_{u-1}a_{u-1})=-\bar a_{u-1}+C_u(\ell_{u-1}\tilde a_{u-1})$$
   Finally
   $$C_u(\ell_{u-1}\ell_u)=\vartheta_{u-1}\vartheta_u e_{u-1}+e_{u+1},\ C_u(\ell_{u-1}^2)=-\vartheta_{u-1}2e_{u-1}.$$
  $$C_u(\ell_{u-1}\tilde a_{u-1})=\begin{cases}\begin{matrix}
-\sigma_{u}\lambda_uC_u(\ell_{u-1}\ell_u) +C_u(\ell_{u-1}^2)  ,\ &\ \sigma_{u-1}=-1,\ &\ell_{u-1}\ \text{red}\quad \\
-\sigma_{u}\lambda_uC_u(\ell_{u-1}\ell_u)  ,\quad\ &\ \sigma_{u-1}= 1,\   &\ell_{u-1}\ \text{red}\quad \\
\sigma_{u-1}\sigma_{u}\lambda_uC_u(\ell_{u-1}\ell_u)  +C_u(\ell_{u-1}^2)  , \ &\  \lambda_{u-1}= 1,\ \quad &\ell_{u-1}\ \text{black}\\
\sigma_{u-1}\sigma_{u}\lambda_uC_u(\ell_{u-1}\ell_u)  , \ &\  \lambda_{u-1}= -1,\ &\ell_{u-1}\ \text{black}
\end{matrix}
\end{cases} $$
 $$C_u(\ell_{u-1}\tilde a_{u-1})=\begin{cases}\begin{matrix}
-\sigma_{u}\lambda_u(-\vartheta_u e_{u-1}+e_{u+1}) +2e_{u-1}  ,\ & \sigma_{u-1}=-1,\   &\ell_{u-1}\  \text{red}\quad \\
-\sigma_{u}\lambda_u(-\vartheta_u e_{u-1}+e_{u+1})  ,\quad\ & \sigma_{u-1}= 1,\ \ \ &\ell_{u-1}\  \text{red}\quad \\
\sigma_{u-1}\sigma_{u}\lambda_u( \vartheta_u e_{u-1}+e_{u+1})  -2e_{u-1}  ,\  & \lambda_{u-1}= 1,\ \quad &\ell_{u-1}\  \text{black}\\
\sigma_{u-1}\sigma_{u}\lambda_u( \vartheta_u e_{u-1}+e_{u+1})  , \ & \lambda_{u-1}= -1,\ &\ell_{u-1}\  \text{black} 
\end{matrix}
\end{cases} $$
   If $\ell_{u-1}$ is red  we then compute the contribution of
   $\delta_{u-1}\ell_{u-1}a_{u-1}$  getting (recall that $\sigma_u$ is $-1$ if $\ell_u$ is red, one otherwise)  \begin{equation}
\label{ru-1}-\delta_{u-1}\bar a_{u-1}+\delta_{u-1}\begin{cases}\begin{matrix}
 e_{u+1}+3e_{u-1},\ &\  \sigma_{u-1}=-1,   \ &\ell_{u} &\text{red}\quad \\
 e_{u+1}+e_{u-1} ,\quad\ &\  \sigma_{u-1}= 1,\  \quad &\ell_{u} &\text{red}\quad \\
- \lambda_u[e_{u+1}-e_{u-1}] +2e_{u-1} ,\ &\  \sigma_{u-1}=-1,  \ &\ell_{u} &\text{black}\\
- \lambda_u [e_{u+1}-e_{u-1}] ,\quad\ &\  \sigma_{u-1}= 1,\  \quad  &\ell_{u} &\text{black}
\end{matrix}
 \end{cases}
\end{equation}
   If $\ell_{u-1}$ is black  we then compute the contribution of
   $-\sigma_{u-1}\delta_{u-1}\ell_{u-1}a_{u-1}$  getting  
 \begin{equation}
\label{bu-1}\sigma_{u-1}\delta_{u-1}\bar a_{u-1}-\sigma_{u-1}\delta_{u-1}\begin{cases}\begin{matrix}
-\sigma_{u-1}  [e_{u+1}-e_{u-1} ]  -2e_{u-1} , & \lambda_{u-1}= 1,\ \  &\ell_{u}\ \text{red}\ \ \\
- \sigma_{u-1}   [e_{u+1}-e_{u-1} ]  , & \lambda_{u-1}= -1, &\ell_{u}\  \text{red}\ \ \\
\sigma_{u-1}\lambda_u[e_{u-1}+e_{u+1} ]  -2e_{u-1}  , & \lambda_{u-1}= 1,\ \ &\ell_{u}\ \text{black}\\
 \sigma_{u-1}\lambda_u[e_{u-1}+e_{u+1} ], &  \lambda_{u-1}= -1,  &\ell_{u}\ \text{black}
\end{matrix}
\end{cases} 
\end{equation}
   We thus write if $\ell_{u-1}$ is red
\begin{equation}
\label{relr}0=-C_u(\mathcal R)= \sum_{i\,|\,\vartheta_i=-1,\ i\neq u-1,u} \delta_i \mu_u(i)\ell_i  -\sum_{i\,|\,\vartheta_i= 1,\ i\neq u-1,u}   \delta_i \sigma_i  \mu_u(i)\ell_i-\delta_{u-1}\bar a_{u-1}+L
\end{equation}If $\ell_{u-1}$ is black
\begin{equation}
\label{relb}0=-C_u(\mathcal R)= \sum_{i\,|\,\vartheta_i=-1,\ i\neq u-1,u} \delta_i \mu_u(i)\ell_i  -\sum_{i\,|\,\vartheta_i= 1,\ i\neq u-1,u}   \delta_i \sigma_i  \mu_u(i)\ell_i+\sigma_{u-1}\delta_{u-1}\bar a_{u-1}+L.
\end{equation}
In both cases by $L$  we denote the contribution  from the  Formulas \eqref{Rp},\eqref{ellu}, and \eqref{ru-1} or \eqref{bu-1}.

\subsubsection{The 18  cases\label{18}}So now we expand  $L$

\noindent 1)\  $\ell_{u-1},\ell_u$ both red $\sigma_{u-1}=1.$
$$\delta_{u }[e_{u-1}- e_{u+1 }]+2\delta_ue_{u+1}-\delta_u(e_{u+1}+e_{u-1})\ =0.$$
2)\  $\ell_{u-1},\ell_u$ both red $\sigma_{u-1}=-1.$
$$\delta_{u }[e_{u-1}- e_{u+1 }]+2\delta_ue_{u+1}-\delta_u[e_{u+1}+3e_{u-1}]=-2\delta_u e_{u-1}.$$
3)\  $\ell_{u-1}$   red, $\ell_u$ black $\sigma_{u-1}=1,\lambda_u=1.$
$$-\delta_{u }[e_{u-1}+e_{u +1}]+2\delta_u e_{u+1}-\delta_u   [e_{u+1}-e_{u-1}]=0$$
4)\  $\ell_{u-1}$   red, $\ell_u$ black $\sigma_{u-1}=-1,\lambda_u=1.$
$$-\delta_{u }[e_{u-1}+e_{u +1}]+2\delta_u e_{u+1}-\delta_u [e_{u+1}-e_{u-1}] -\delta_u2e_{u-1}=-2\delta_u e_{u-1}$$
5)\  $\ell_{u-1}$   red, $\ell_u$ black $\sigma_{u-1}=1,\lambda_u=-1.$
$$-\delta_{u }[e_{u-1}+e_{u +1}]+\delta_u  [e_{u+1}-e_{u-1}]=-2\delta_u e_{u-1}$$
6)\  $\ell_{u-1}$   red, $\ell_u$ black $\sigma_{u-1}=-1,\lambda_u=-1.$
$$-\delta_{u }[e_{u-1}+e_{u +1}]+\delta_u  [e_{u+1}-e_{u-1}] +\delta_u2e_{u-1}=0$$
7)\  $\ell_{u-1}$   black, $\ell_u$ red $\sigma_{u-1}=1,\lambda_{u-1}=1.$
$$\delta_{u}[ e_{u-1}-  e_{u +1}] +2\delta_ue_{u+1}-\delta_{u}  [e_{u+1}-e_{u-1} ] -2\delta_{u} e_{u-1}=0$$
8)\  $\ell_{u-1}$   black, $\ell_u$ red $\sigma_{u-1}=-1,\lambda_{u-1}=1.$
$$\delta_{u}[-e_{u-1}-  e_{u +1}]+2\delta_ue_{u+1}+\delta_{u}  [e_{u+1}-e_{u-1} ] +2\delta_{u} e_{u-1}=2\delta_ue_{u+1}.$$
9)\  $\ell_{u-1}$   black, $\ell_u$ red $\sigma_{u-1}=1,\lambda_{u-1}=-1.$
$$\delta_{u}[ e_{u-1}-  e_{u +1}]+2\delta_ue_{u+1}+\delta_{u}  [e_{u+1}-e_{u-1} ]= 2\delta_ue_{u+1} $$
10)\  $\ell_{u-1}$   black, $\ell_u$ red $\sigma_{u-1}=-1,\lambda_{u-1}=-1.$
$$\delta_{u}[-e_{u-1}-  e_{u +1}]+2\delta_ue_{u+1}-\delta_{u}  [e_{u+1}-e_{u-1} ] =0. $$
11)\  $\ell_{u-1}$,  $\ell_u$  both black,  $\sigma_{u-1}=1,\lambda_{u-1}=1,\ \lambda_u=1.$
$$-\delta_{u }[ e_{u-1}+ e_{u+1 }]+2\delta_ue_{u+1}-\delta_u [e_{u-1}+e_{u+1} ]  +2\delta_ue_{u-1} =0$$
12)\ $\ell_{u-1}$,  $\ell_u$  both black $\sigma_{u-1}=-1,\lambda_{u-1}=1,\ \lambda_u=1.$
$$-\delta_{u }[-e_{u-1}+ e_{u+1 }]+2\delta_ue_{u+1}-\delta_u[e_{u-1}+e_{u+1} ]  -2\delta_ue_{u-1} =-2\delta_ue_{u-1}$$
13)\ $\ell_{u-1}$,  $\ell_u$  both black $\sigma_{u-1}=1,\lambda_{u-1}=-1,\ \lambda_u=1.$
$$-\delta_{u }[ e_{u-1}+ e_{u+1 }] +2\delta_ue_{u+1}-\delta_u [e_{u-1}+e_{u+1} ]=-2\delta_u e_{u-1}$$
14)\ $\ell_{u-1}$,  $\ell_u$  both black $\sigma_{u-1}=-1,\lambda_{u-1}=-1,\ \lambda_u=1.$
$$-\delta_{u }[-e_{u-1}+ e_{u+1 }] +2\delta_ue_{u+1}-\delta_u[e_{u-1}+e_{u+1} ]=0$$
15)\  $\ell_{u-1}$,  $\ell_u$  both black,  $\sigma_{u-1}=1,\lambda_{u-1}=1,\ \lambda_u=-1.$
$$-\delta_{u }[ e_{u-1}+ e_{u+1 }]+\delta_u[e_{u-1}+e_{u+1} ]  +2\delta_ue_{u-1}=2\delta_ue_{u-1} $$
16)\ $\ell_{u-1}$,  $\ell_u$  both black $\sigma_{u-1}=-1,\lambda_{u-1}=1,\ \lambda_u=-1.$
$$-\delta_{u }[-e_{u-1}+ e_{u+1 }]+\delta_u[e_{u-1}+e_{u+1} ]  -2\delta_ue_{u-1} =0$$
17)\ $\ell_{u-1}$,  $\ell_u$  both black $\sigma_{u-1}=1,\lambda_{u-1}=-1,\ \lambda_u=-1.$
$$-\delta_{u }[ e_{u-1}+ e_{u+1 }]+\delta_u[e_{u-1}+e_{u+1} ]=0 $$
18)\ $\ell_{u-1}$,  $\ell_u$  both black $\sigma_{u-1}=-1,\lambda_{u-1}=-1,\ \lambda_u=-1.$
$$-\delta_{u }[-e_{u-1}+ e_{u+1 }]+\delta_u[e_{u-1}+e_{u+1} ]=2\delta_ue_{u-1}$$
By inspection we see that we have proved the following remarkable:
\begin{corollary}\label{leduei}
The contribution of $L$ equals to 0 if and only if $\sigma_{u-1}= \lambda_{u-1}\lambda_u $. In this case the coefficient of $e_u$ in the end point $x_{u-1}$ of  the segment $S_u$ is $0$.

If $\sigma_{u-1}= -\lambda_{u-1}\lambda_u $ the contribution of $L$ equals to $\pm 2e_{u\pm 1}$. In this case the coefficient of $e_u$ in the end point $x_{u-1}$ of  the segment $S_u$ is $\pm 2$.
\end{corollary}
 \begin{proof}  The first is by inspection, as for the second we check a few cases.
This coefficient comes from the two contributions of  $\ell_{u-1},\ell_u$. They appear by $\sigma_{u-1}[\sigma_u\lambda_u\ell_u+\sigma_{u-1}\lambda_{u-1}\ell_{u-1}]$.  Now $\sigma_u\lambda_u\ell_u=-\ell_u=e_u+e_{u+1}$  if $\ell_u$ is red and similarly $\sigma_{u-1}\lambda_{u-1}\ell_{u-1}=e_u+e_{u-1}$  if $\ell_{u-1}$ is red and $\sigma_{u-1}=-1$.  This is case 2).    If $\ell_{u-1}$ is black  then the coefficient of $e_u$ in $\sigma_{u-1}\lambda_{u-1}\ell_{u-1}$ is 1 if and only if $\sigma_{u-1}\lambda_{u-1}=-1$  and in this case this is equivalent to $\sigma_{u-1}=-\lambda_{u-1} \lambda_{u }.$  These are cases 8,9.

Similar argument when $\ell_u$ is black.
\end{proof} 
\begin{corollary}\label{leduei1}
If $\ell_{u-1}\prec\ell_j$ we have $\mu_u(j)=0$ if the contribution of $L$ is 0, otherwise $\mu_u(j)=\pm 2$.
\end{corollary}

\subsubsection{Contribution  of $L$ equals to 0\label{coL0}} We say that $u$ is of type  I.  We deduce that the other edges $\ell_i$  satisfy a relation, i.e. either \eqref{relr} or \eqref{relb}. This is impossible unless this is the trivial relation with all coefficients 0.  Let us draw the implications of this.  Recall that $S_u$ is the minimal segment containing the edges $\ell_u,\ell_{u-1}$ (cf. Formula \eqref{esu}).

 Notice that any edge $\ell_j$ comparable with $\ell_u$ and not with $\ell_{u-1}$  appears in the relation, only from the  term $\mu_u(j)\ell_j$ (indeed in this case $\bar a_{u-1}$ does not depend on $\ell_j$).  Since then $\mu_u(j)=\pm 1$   this is a contradiction. Thus  no  edge  is comparable  with $\ell_u$ and not with $\ell_{u-1}$. This means that  all internal vertices of $S_u$ have valency 2, moreover  all  edges  $\ell_j$  with $\ell_u\prec\ell_j\prec\ell_{u-1}$ appear with coefficient $\pm\delta_{u-1}\pm \delta_j$, coming from $\bar a_{u-1}\delta_{u-1}$ and from $\pm \delta_j\ell_j\mu_u(j)$ (see formulas \eqref{relr}-\eqref{relb}) we thus must have that this sum equals zero. 
 
  Now, in case 2)  if we start from $u\in A\cup C$ (see Remark \ref{divido}) this implies that  it is not possible that $j\in B$  since  the sum of these two coefficients  is odd  and so it is not zero, so the segment $S_u$ is all formed by elements in $A\cup C$. If we start from $u\in B$   it is not possible that $j\in A\cup C$  since again  $\pm\delta_{u-1}\pm \delta_j$ is odd, so the segment $S_u$ is all formed by elements in $B$.

 Finally in case 1)  with an extra edge $E$  it is not possible that $E$  is in between  $\ell_{u-1},\ell_u$ otherwise  $E$ would appear and only in $\bar a_{u-1}$. Hence the value of $\zeta$ of the  relation would be $\pm2.$
 \subsubsection{Contribution  of $L$ equals to $\pm 2\delta_ue_{u\pm 1}$}  We say that $u$ is of type II.    We  thus have, from \eqref{relr} or  \eqref{relb}, a relation expressing $\pm 2\delta_ue_{u\pm 1}$ as linear combination of the edges $\ell_j\neq  \ell_{u-1},\ \ell_u$.  
 Now these edges are linearly independent so such an expression if it exists it is unique.  Let us assume for instance that the relation expresses $2e_{u-1}$, the other case is identical.
 
 In order to understand which elements appear in $C_u$, 
 first remark that
 the only edges that may contribute to the expression of $C_u$ are those for which  $\ell_u\prec\ell_j$. If $\ell_j$ is not comparable with $ \ell_{u-1}$ they contribute by  $\pm \delta_j$.
If $\ell_u\prec\ell_j \prec\ell_{u-1}$ they contribute by  $\pm \delta_j\pm\delta_{u-1}.$  Finally if  $ \ell_{u-1}\prec\ell_j$  they contribute by $0,\pm2\delta_j$ by   Corollary \ref{leduei1}.
 
{\bf Case 1A} (single loop) no extra edge: \quad such a relation does not exist. For instance if $2e_{u-1}$ is a linear combination $\sum_jc_j\ell_j$ of the edges $\ell_j\neq  \ell_{u-1},\ \ell_u$ since $e_{u-1}$ only appears in $\ell_{u-2}$ with sign $-1$ we must have that $c_{u-2}=-2$ and then $2e_{u-2}$ is a linear combination $\sum_jc_j\ell_j$ of the edges $\ell_j\neq  \ell_{u-2}, \ell_{u-1},\ \ell_u$, continuing by induction we reach a contradiction.\smallskip

{\bf Case 1B}  (single loop)  an extra edge: \quad we may assume that the extra edge $E=\vartheta e_1-e_h$, this edge divides the loop into two parts $A,B$. The edges in $A:=\{\ell_1,\ldots,\ell_{h-1}\}$ and $E$ form an odd loop as well as the edges in $B$ and $E$.   We may  assume for instance that $h<u $ is an index in $B$. We know that, for an odd loop, we can write $2e_{1}$ uniquely as the sum  of the edges of the odd loop $A,E$ and then we write $2e_{u-1}= \pm \sum_{k=1}^{u-2}2\delta_k\ell_k\pm 2e_1$, let us call $\mathcal R'$ this relation.  The edges appearing in the relation are all the edges of $A,E$ with coefficient $\pm 1$  and all the edges $\ell_k,\ h\leq k\leq u-2$ with coefficients $\pm 2$.
 This relation must be proportional to either \eqref{relr}  or \eqref{relb}. Notice that $E$ appears in this relation with coefficient $\pm 1$.
 
   This is  possible if and only if $E\prec \ell_{u-1}$. Moreover we know that  all the edges in $A$ appear with coefficient $\pm 1$ hence by  Corollary \ref{leduei1} it follows that they must be comparable with $\ell_u$ but not with $\ell_{u-1}$.  Finally for the edges in $B$  we have that the $\ell_k$ with $h\leq k\leq u-2$ are comparable with $\ell_u$ and, since they appear with coefficient $\pm 2$ in $\mathcal R'$, we must have either $ \ell_k\prec\ell_{u-1}$ or $\ell_{u-1}\prec\ell_k     $. All the others are not comparable with $\ell_u$.  
   
   Denote by  $T_A$ and $T_B$  the  two minimal trees generated by $A,B$ respectively. We have:
   \begin{corollary}\label{ePo} \quad 1)\  If the indices of $A$ and $B$ are all of type I  then either $T_A$ and $T_B$ form two disjoint segments separated by $E$, or      the edges in $A\cup B$ form a segment, the extra edge is outside this segment so the graph is not minimal degenerate.
   
\quad 2)\ If there is an index in $B$  (resp. in $A$) of type II, the two minimal trees  $T_A$ and $T_B$  generated by $A,B$ respectively are segments and can intersect only in a vertex  or in the edge $E$.  If they intersect in a vertex then all $v\in A$ (resp. all   $v\in B$) have type I and the vertex is an end point of $E$.
\end{corollary}
\begin{proof} \quad 1)\ In this case we know that all the segments $S_u$ for $u$ non critical are segments which do not contain $E$ and with the interior vertices of valency 2. By a simple induction we have that $\cup_{u\in A}S_u$ and $\cup_{v\in B}S_u$ are segments which do not contain $E$ and with the interior vertices of valency 2 (cf. Lemma \ref{SEP}).  If these two segments have an edge in common  then, by the same Lemma, their union is a segment not containing $E$ and thus  this segment gives a minimal degenerate graph and the one we started from is not minimal.  The same happens if they meet in an end point of both. The only remaining case is that  $T_A$ and $T_B$ form two disjoint segments separated by $E$.\smallskip

\quad 2)\ We have just seen that all the edges in $A$ lie in branches originating from vertices of the segment $S_u$  different from the last vertex of $\ell_{u-1}$. On the other hand the edges in $B$ are in $S_u$ and possibly in the other branches  originating from the end points of $S_u$. This implies that the two trees $T_A$ and $T_B$ can only have an intersection inside $S_u$.

Take any non critical index $v\in A$, if  $v$ is of type I  the segment $S_v$  is either disjoint from $S_u$ or it may intersect  $S_u$ in a vertex, since $S_v$ has the interior vertices of valency 2 and it cannot overlap with $S_u$ otherwise one or both of its ending edges, both in $A$ would also be in $S_u$ which on the contrary is all formed by edges in $B$. If  $v$ is of type II  we can apply the same analysis to $v$ and deduce that   the segment $S_v$  intersects  $S_u$ in the edge $E$.

If  all indices in $A$ are of type I  by the previous analysis the tree they generate can  meet  $S_u$ (and also $T_B$) only in one vertex so  they lie in a single branch. Applying Lemma \ref{SEP} it follows that the tree $T_A$  is a segment and it intersects $S_u$  in a vertex.

Now suppose that this vertex $v$  is not an end point of  $E$. Call $S$ the segment from $v$ to $E$.  If $\ell_j\in S$   we must have that if $j$ is not a critical index $1,h$ it must be of type I (otherwise we could not have that the edges in $A$ follow $\ell_j$) and thus $\ell_{j-1}\in S$.   Also $\ell_{j+1}\in S$  otherwise  it should be of type II  but then we have again that the vertex $v$ is outside the segment $S_{j+1}$, by induction we arrive at a contradiction $\ell_u\in S$.
\smallskip

As for $T_B$  we have now proved that it is formed only by the edges in $B$ and by $E$.  By induction we see that $T_B=\cup_{a\in B} S_a$ and in fact it is a segment. In fact let $T_B^i:=\cup_{h<j\leq i\leq k} S_i$, assume $T_B^i$ is a segment and consider $T_B^{i+1}= T_B^i\cup S_{i+1}.$  By induction and construction these two segments intersect at least in the edge $\ell_i$.  If at least one of the two is only formed from indices of type I we see again by induction that its interior vertices have valency 2 and by Lemma \ref{SEP} we have that their union is a segment.  If   $i+1$ is of type II  as well as one of the indices $j$ with $h<j\leq i$  we have that $T_B^i$ contains $E$. By the previous analysis it follows that inside the segment  $T_B^i$ and $S_{i+1}$  all interior vertices have valency 2 hence again Lemma \ref{SEP} applies.

\end{proof}

 {\bf Case 2}    A doubly odd loop is divided in 3 (or 2) parts: the two odd loops $A,C$ and the segment $B$ (possibly empty) joining them. We divide this into two subcases:
 
 \quad {\bf  Assume first $u\in A$  } (the case $u\in C$ is similar)\quad we have  $\pm 2\delta_ue_{u\pm 1} $ a linear combination  of the edges  in $B,C$ with coefficient $\delta_i$ (or all $-\delta_i$) equal to $2e_1$ plus, (cf. Formula \eqref{sopar}), $2\sum_{i=1}^{u-2}\delta_i\ell_i=-2\delta_{u-2} e_{u-1}-2e_1$      from which we have the required expression for $-2e_{u-1}$, similarly for $-2e_{u+1}$. This is the unique expression $\mathcal R'$ as linear combination of the linearly independent edges $\ell_j\neq  \ell_{u-1},\ \ell_u$.

As before  this relation must be proportional to either \eqref{relr}  or \eqref{relb}. Inspecting these relations we first observe that, if $j\in B,C$     the edge $\ell_j$ must have coefficient   $\pm\delta_j$.  By corollary \ref{leduei1}  if $\ell_{u-1}\prec \ell_j$ we have that $\mu_u(j)=\pm 2$  hence by inspection we deduce that   $\ell_{u-1}\not\prec\ell_j $. 

If $\ell_u\prec\ell_j\prec \ell_{u-1}$ the coefficient of $\ell_j$  in the two possible relations comes from two terms,   a term $\pm \delta_j$ coming from the first two summands (since in this case $\mu_u(j)=\pm 1$), and a term $\pm \delta_{u-1}$ from $\bar a_{u-1}$, hence no index in $B$ or $C$ can appear in  $\bar a_{u-1}$ by parity.   Since these edges appear in the relation $\mathcal R'$ we deduce that all $\ell_j, j\in B\cup C$ are in branches which originate from internal vertices of $S_u$. Inside the segment $S_u$  there are only   edges of  $A $. If we are in the case  $L=\pm 2\delta_ue_{u- 1}$ all edges $\ell_j$ with $\ell_{u-1}\prec\ell_j$  appear with coefficient $\pm 2\delta_j$ hence they are in the set  $i\in A, i\leq u-2$.  The remaining edges $\ell_i$ in $A$ with $i>u$ do not appear hence they  either satisfy   $\ell_{u }\prec\ell_i\prec\ell_{u-1}$ or are not comparable with $\ell_u$. Similar discussion for $L=\pm 2\delta_ue_{u+ 1}$.
 A similar consideration holds if $u\in C$.
 
 \quad {\bf  Assume   $u\in B$  } \quad If $u\in B$   the contribution of $L$ is $\pm 4e_{u\pm 1}$.  The two cases are similar.

   i) If the contribution is  $\pm 4e_{u-1}$, this comes from a sum   $2\sum_{i\in A}\delta_i\ell_i=\pm 4e_{ 1}$  plus  $2\sum_{j\in B,\ j\leq u-2}\delta_j\ell_j=\pm 4[e_{u-1}\pm e_{ 1}]$.
   
    ii) The contribution    $\pm 4e_{u+1}$,  comes from  the sum   $2\sum_{i\in C}\delta_i\ell_i=\pm 4e_{b}$  plus a sum of  $2\sum_{j\in B,\ j\geq u+1}\delta_j\ell_j=\pm 4[e_{u+1}\pm e_{ b}]$.

This formula for $L$  must coincide  with that given by \eqref{relr}  or \eqref{relb}.

We claim that there is no edge $\ell_j$  with $\ell_u\prec\ell_j$ and $\ell_j$ is not comparable with $\ell_{u-1}$.  Indeed this edge would have  $\mu_u(j)=\pm 1$ and would not appear in $\bar a_{u-1}$. This is incompatible with the fact that the coefficient must be $\pm 2\delta_j$. Thus  we deduce that all internal vertices of the segment $S_u$ have valency 2.

Finally if $\ell_u\prec\ell_j\prec\ell_{u-1}$  we have that the coefficient of  $\ell_j$ in the relation associated to Formulas \eqref{relr}  or \eqref{relb} is $\pm \delta_j\pm \delta_{u-1}$. Note that $u\in B$ is not critical and hence $\ell_u,\ell_{u-1}\in B$ so $\delta_{u-1}=\pm 2$.  If $j\in A\cup C$  we have that this  number is odd  so it cannot be one of the coefficients  appearing in the relation i) or ii). In case i) finally we deduce that if $j\in A$ we have $\ell_{u-1}\prec \ell_j$ while all the  $j\in C$  lie in the branches of the tree from the root different from the one  containing $\ell_u$.
\begin{corollary}\ 1)\ 
The edges in $B$ always form a segment, its internal vertices have valency 2.

\  2)\  If there is  an index of type II in $B$ all edges in $A$ and all edges in $C$ are separated and lie in the two segments originating from the two end points of  $S_u$.

\ 3)\  If there is  an index of type II in $A$ (or $C$) all edges in $A$ and all edges in $C$ are separated and lie in   two segments which can be disjoint or meet in one vertex.

\ 4)\ If all indices are of type I  then either all edges in $A$ and all edges in $C$ are separated and lie in the two segments originating from the two end points of  $S_u$.
 or the edges of $A\cup C$ form a segment.
\end{corollary} 
\begin{proof} \ 1)\ The proof is similar to that of Corollary \ref{ePo}.  We already know that, if $j\in B$  is of type I 
inside the segment $S_u$ there are only edges $\ell_j$ with $j\in B$ and its internal vertices have valency 2, we have proved this now also for type II.   The claim follows from  Lemma \ref{SEP}.

\ 2)\  Assume there is an index $u\in B$  of type II with contribution $\pm 4e_{u-1}$. Analyzing the corresponding relation we have then that all edges $\ell_i$ with $i\leq u-2$   and all edges in $C$ precede $\ell_u$, all edges in $A$ follow $\ell_{u-1}$.\smallskip

Finally if $\ell_{u-1}\prec\ell_j$ then $\ell_j$ appears in the relation  so since in  the relation appear either all the edges in $C$ and none of the edges in $A$ or conversely we must have that these two blocks lie in the two branches originating from the two end points of  $S_u$.\smallskip

\ 3)\  Assume there is an index of type II in $A$, we then have seen that $T_C$ is formed by branches originating from interior points of $S_u$.  Now  if  $u\in C$  is of type I  the segment  $S_u$ cannot contain edges in $A$ otherwise it would contain interior vertices of valency $>2$.   If  $u\in C$  is of type II  the segment  $S_u$ does not contain edges in $A$ by the previous argument.

\ 4)\  If all indices are of type I we have seen that all segments $S_u$ for $u\in A\cup C$ are formed by edges in $A\cup C$ and their interior vertices have valency 2. Finally the statement that we have segments follows from Lemma \ref{SEP} as in Corollary \ref{ePo}.
\end{proof} 
 
\subsubsection{ All indices  are of type I,\ $L=0$}  We have already seen (Case 1) that the case of   the single loop and all indices  are of type I is not possible. Let us thus treat the special case when we are in the doubly odd loop  and still  all indices of $A\cup C$ are of type I   or  when just the indices of $A$ are of type I but we know that they form a segment.

If neither $S_A,S_B,S_C$ contains a critical vertex we   have seen that the graph spanned by  $A\cup C$  is a segment as well as $S_B$ and we have.
\begin{equation}
\label{unseg}a') \xymatrix{    0\ar@{-}[rr]^{S_{A\cup C}  }& &v  \ar@{-}[rr]^{S_B} && w    }. 
\end{equation}

 In this segment we take as root one on its end points and denote by $\bar\sigma_i,\bar\lambda_i$ the corresponding values  of color and orientation (with respect to this root).  Recall that the notation $\sigma_i,\lambda_i$ is relative to the segment   $S_u$ as in the previous discussion (see formula \eqref{esu}).   
 In the next Lemma we analyze  the 9 cases in which $L=0$.

 \begin{lemma}\label{icons}
 We claim that every edge $\ell_j,\ j\in A$  (resp. $j\in C$)   has the property that $\delta_j= \delta  \bar\sigma_j$ if red and $\delta_j=  \delta \bar\lambda_j\bar\sigma_j$ if black for $\delta=\delta_1\bar\sigma_1$ (resp.  $\delta=\delta_h\bar\sigma_h$ where $h$ is the minimal element in $C$).
\end{lemma}
\begin{proof} By induction $\delta_{u-1}= \delta  \bar\sigma_{u-1}$ if red and $ \delta_{u-1}=  \delta \bar\lambda_{u-1}\bar\sigma_{u-1}$ if black.

Look at  $S_u$. If $\ell_{u-1}, \ell_u$ are both red $\sigma_{u-1}=1$, (Case 1))  $$\delta_u=-\delta_{u-1}=  -\delta  \bar\sigma_{u-1}= \delta  \bar\sigma_{u }\sigma_{u-1}   =   \delta  \bar\sigma_{u }$$

If $\ell_{u-1}$ is red and $ \ell_u$ is black  we are in Cases 3), 6) and we have $\sigma_{u-1}=\lambda_u,\ \delta_u=\delta_{u-1}=  \delta  \bar\sigma_{u-1}$.   We also have $\sigma_{u-1}=-\bar \sigma_{u-1}\bar \sigma_{u}$ if $\ell_{u-1}\prec\ell_u$ and $\sigma_{u-1}= \bar \sigma_{u-1}\bar \sigma_{u}$ if $\ell_u\prec\ell_{u-1}$.
$$ \delta_u= \begin{cases}
-\delta  \bar\sigma_{u }\lambda_u= \delta  \bar\sigma_{u }\bar\lambda_u\quad \ell_{u-1}\prec\ell_u\\
\delta  \bar\sigma_{u }\lambda_u= \delta  \bar\sigma_{u }\bar\lambda_u\quad \ell_u\prec\ell_{u-1}\end{cases} .$$
If  $\ell_{u-1}$ is black and $ \ell_u$ is  red  we are in Cases 7), 10) and we have  $\sigma_{u-1}=\lambda_{u-1}$.   If $\ell_{u-1}\prec\ell_u$ we have $\lambda_{u-1}\bar\lambda_{u-1}=-1, \bar\sigma_{u-1}=  \bar\sigma_u\sigma_{u-1}$$$\delta_u= -\delta_{u-1}=   -\delta  \bar\sigma_{u-1}\bar\lambda_{u-1} = - \delta  \bar\sigma_{u }\sigma_{u-1} \bar\lambda_{u-1}  =  - \delta  \bar\sigma_{u }\lambda_{u-1}\bar\lambda_{u-1}=\delta  \bar\sigma_{u }.$$
If $\ell_u\prec\ell_{u-1}$ we have $\lambda_{u-1}\bar\lambda_{u-1}= 1, \bar\sigma_{u-1}= - \bar\sigma_u\sigma_{u-1}$$$\delta_u= -\delta_{u-1}=   -\delta  \bar\sigma_{u-1}\bar\lambda_{u-1} =   \delta  \bar\sigma_{u }\sigma_{u-1} \bar\lambda_{u-1}  =   \delta  \bar\sigma_{u }\lambda_{u-1}\bar\lambda_{u-1}=\delta  \bar\sigma_{u }.$$

If $\ell_{u-1}, \ell_u$ are both black  we are in Cases 11), 14), 16), 17)  and we have  $\sigma_{u-1}=\lambda_u\lambda_{u-1}$ by Corollary \ref{leduei}.
 If $\ell_{u-1}\prec\ell_u$ (in the order of the total segment) we have $\lambda_{u-1}\bar\lambda_{u-1}=-1, \bar\sigma_{u-1}=  \bar\sigma_u\sigma_{u-1}$$$\delta_u=  \delta_{u-1}=   \delta  \bar\sigma_{u-1}\bar\lambda_{u-1} =  \delta  \bar\sigma_{u }\sigma_{u-1} \bar\lambda_{u-1}  =    \delta  \bar\sigma_{u }\lambda_{u-1}\bar\lambda_{u-1}=\delta  \bar\sigma_{u }.$$
 $$\delta_u= -\delta_{u-1}=   -\delta  \bar\sigma_{u-1}\bar\lambda_{u-1} =   \delta  \bar\sigma_{u }\sigma_{u-1} \bar\lambda_{u-1}  =   \delta  \bar\sigma_{u }\lambda_{u-1}\lambda_{u }\bar\lambda_{u-1} $$
  Now clearly  $\lambda_{u-1}\lambda_{u }\bar\lambda_{u-1}=\bar\lambda_{u }$.
 \end{proof}
 
Now  we take the left vertex of $S_{A\cup C}$ as in \eqref{unseg}  as root, that is we consider it as the 0 vertex and want to compute first the value of the other end vertex $v$  of $S_{A\cup C}$ and then the end vertex $w$ of the total segment appearing in \eqref{unseg}. Recall that  we have an even number of red edges so that the end vertex is  black, let us say that this vertex belongs to the last edge $\ell_j$.  We can compute it by using the various options of formula \eqref{La}.  If $\ell_j$ is red  or if it is black and $\lambda_j=-1$  we have that the last vertex is $v=b_j$ and not $a_j$, in the remaining case $v=a_j$. In all cases a simple   analysis shows that $v=\pm\sum_j\bar \lambda_j \bar\sigma_j\ell_j$. By Lemma \ref{icons} we have  $\bar \lambda_j \bar\sigma_j=\delta\delta_j$ hence   $\sum_{j\in A}\bar \lambda_j \bar\sigma_j\ell_j=\pm 2e_1$ and similarly $\pm\sum_{j\in C}\bar \lambda_j \bar\sigma_j\ell_j=\pm 2e_b$.  We thus  have that $v=\pm 2(e_1-e_b)$  or  $v=\pm 2(e_1+e_b)$ but this  is impossible for a black vertex which has mass 0.

Now a similar argument on the segment $S_B$ gives as value of $S_B$ either $\pm(e_1-e_b)$ or  $-e_1-e_b$.

In the first case  we take as root the point $v$. Now the left and right hand vertices are $a=\pm (e_1-e_b),b=\pm 2(e_1-e_b)$. The relation is $b=\pm 2a$ so the resonance must be $C(b)=\pm 2 C(a)$  which we see immediately is not valid.

It remains the possibility $a=-e_1-e_b,\, b=\pm 2(e_1-e_b)$, in this case fixing one end vertex to be 0 the other is $a+b= -e_1-e_b  \pm 2(e_1-e_b)$ which also gives a non allowable graph from Definition \ref{ilpunto1} and Proposition \ref{ilpunto0}.

 If the edges in $A$ form a segment and are of type I  the same argument shows that fixing the root at one end the other end vertex is $-2e_i$ for some $i$. We deduce
\begin{corollary}\label{PrC}
The case of all indices of type I  does not occur or it produces a not--allowable graph   \ref{ilpunto1}.
\end{corollary}

 \subsubsection{Indices of type II}  If there is at least one index of type II  the case analysis that we have performed shows that
  between two edges in $A$ there are only edges in $A$ and the edges in $A$ form a segment, the same happens for $B,C$. Denoting $S_A,S_B,S_C$ these segments  their union is a tree, the internal vertices of $S_B$ have valency 2,  so their relative position a priori can be  only one of the following.
$$ a) \xymatrix{    \\\ar@{-}[rr]^{S_C}  & &  \ar@{-}[rr]^{S_B} &&  \ar@{-}[rr]^{S_A}  &&   } \qquad b) \xymatrix{  &  \ar@{-}[ddd]^{S_C}  && \ar@{-}[dd]^{S_B}  &\\  &&&&   \\  \ar@{-}[rrrr]^{S_A} &&&&\\  &&&&   } $$
$$ c) \xymatrix{  &  \ar@{-}[ddd]^{S_A}  && \ar@{-}[dd]^{S_B}  &\\  &&&&   \\  \ar@{-}[rrrr]^{S_C} &&&&\\  &&&&   } \qquad  d) \xymatrix{     \ar@{-}[rrr]^{S_A}  && \ar@{-}[dd]^{S_B}  &\\  &&&&   \\  \ar@{-}[rrrr]^{S_C\qquad} &&&&\\  &&&&   } $$ where if only one of $S_A,S_C$ contains a critical vertex we have the special cases 

$$ b') \xymatrix{  &  \ar@{-}[dd]^{S_C}  && \ar@{-}[dd]^{S_B}  &\\  &&&&   \\  \ar@{-}[rrrr]^{S_A} &&&&\\  &&&&   } \qquad  c') \xymatrix{  &  \ar@{-}[dd]^{S_A}  && \ar@{-}[dd]^{S_B}  &\\  &&&&   \\  \ar@{-}[rrrr]^{S_C} &&&&\\  &&&&   } $$
In all these cases it is possible that the two critical vertices coincide as in 
 $$ b'') \xymatrix{  &  \ar@{-}[ddd]^{S_C}  && \ar@{-}[ddll]^{S_B}  &\\  &&&&   \\  \ar@{-}[rrrr]^{S_A} &&&&\\  &&&&   } $$

 In all these cases we  may also have that $B$ is empty  so $S_B$ does not appear.

2) If  $A$   contains no index of type II)  we apply to it Lemma \ref{icons}  and deduce that the segment equals $\delta \sum_{i\in A}\delta_i\ell_i=-2\delta e_1$. Since the mass of a segment can only be $0,-2$ we deduce that if one extreme is set to be 0 the other is $-2e_1$.

3) is similar to 2).

Notice that at this point we have proved for the doubly odd loop Theorem \ref{MM} in all cases except b), c), d), b'). Of course b) and c) are equivalent and in fact b') is a special case of b).

4) Let us treat the case in which $u\in A$ gives a contribution to $L$ equal $\pm 2e_{u-1}$ (the other is similar), from our analysis in our setting  all edges $\ell_j,\ j\leq u-2$ must be comparable with $\ell_u $.
\smallskip

In all cases we have that   $S_A$  and $S_C$ have a unique critical vertex which divides the segment.

 So $S_A$ is divided into two segments, one  $X$ ending with a red vertex $x$ the other $Y$ with a black vertex $y$ since in   $S_A$ there is an odd number of red edges  which are distributed into the two segments.  \smallskip

We choose as root the critical vertex. With this choice we denote by $\bar\sigma,\bar \lambda$ the corresponding values  on the edges  (in order to distinguish from the ones $ \sigma,  \lambda$ we have used where the root is at the beginning of $S_u$). 
 \smallskip

 \begin{lemma}\label{chiave} i)\quad   The edges in $Y, X$  have the property that, $ \delta_j\bar\sigma_j\bar\lambda_j=\delta$ is constant.
 \smallskip

  ii)\quad  $$y=\sum_{j\in Y} \bar\sigma_j \bar\lambda_j\ell_j= \delta\sum_{j\in Y}\delta_j\ell_j;\quad x=-\sum_{j\in X} \bar\sigma_j \bar\lambda_j\ell_j=- \delta\sum_{j\in X}\delta_j\ell_j$$
$$\delta=-1,\quad x-y=-2e_1 $$
 
\end{lemma}
\begin{proof}  i)    We want to prove that on   $X$ and $Y$ the value  $ \delta_j\bar\sigma_j\bar\lambda_j$ is constant. For this by induction it is enough to see that  the value does not change  for $\ell_u,\ell_{u-1}$. When they are not separated  we can use Lemma  \ref{icons}. When separated
  we first compare the values that we call $\bar\sigma_j$ when we place the root at the critical vertex with the values $\sigma_j$  when we place the root at the beginning of $\ell_u$ and we easily see that
 $ \bar\sigma_u\bar\sigma_{u-1}=\sigma_{u-1} $. 
In order to prove  that $\delta_j \bar \sigma_j\bar\lambda_j$ is constant  we need to show that when $\ell_u,\ell_{u-1}$ are separated
$$ 1=   \delta_{u-1}\bar\sigma_{u-1}\bar\lambda_{u-1}    \delta_u \bar\sigma_u \bar\lambda_u=  \delta_{u-1} \sigma_{u-1}\bar\lambda_{u-1}    \delta_u   \bar\lambda_u . $$
  We have $\bar\lambda_{u-1}= \lambda_{u-1}$  while $\bar\lambda_u=- \vartheta_u\lambda_u$. In other words  we need
  $$  -\delta_{u-1}\vartheta_u \sigma_{u-1} \lambda_{u-1}    \delta_u    \lambda_u=1. $$
Since by definition $\delta_{u-1}\vartheta_u =    \delta_u $
we have to verify that
   $$  -\delta_{u-1}\vartheta_u \sigma_{u-1} \lambda_{u-1}    \delta_u    \lambda_u=-  \sigma_{u-1} \lambda_{u-1}      \lambda_u=1.  $$  This is in our case the content of the second part of Corollary \ref{leduei}.
\smallskip

    ii) By definition
   $$y=\sum_{j\in Y} \bar\sigma_j \bar\lambda_j\ell_j= \delta\sum_{j\in Y}\delta_j\ell_j;\quad x=-\sum_{j\in X} \bar\sigma_j \bar\lambda_j\ell_j=- \delta\sum_{j\in X}\delta_j\ell_j$$ hence
   $x-y=- \delta\sum_{j\in A}\delta_j\ell_j=\delta2e_1$.  But $\eta(x)=-2,\eta(y)=0$  implies $\delta=-1$.
\end{proof}

If we take as root the vertex $x$  the other vertex of $S_A$ is $x+y$.   \begin{proposition}
If the graph is resonant $x+y=-2e_j$ for some $j$.
\end{proposition}
\begin{proof} We choose as root the critical vertex of $S_A$.
We have  $x-y=-2e_1 =\sum_{j\notin A}\delta_j \ell_j$. This  is a linear combination of the edges outside the segment $S_A$ therefore the resonance relation has the form:
$$C(x)-C(y)=\sum \alpha_iC(v_i) $$  where the vertices $v_i$ are linear combination of the edges not in $A$. Therefore  these vertices have support which intersects the support of the vertices in $S_A$ only in $e_1$, hence we must have $C(x)-C(y)=\alpha e_1^2$ for some $\alpha$. Applying the mass $\eta$ we see that $\eta(C(y))=0,\ \eta(C(x))=-1$ hence $\alpha=-1.$

We now apply the rules of the operator $C$ to $x $ red, $y$ black  $$ -e_1^2=C(-2e_1)=-C(-y )+C(x )+xy=-C( y)+y^{(2)}+C(x)+xy$$and get that $C(x)-C(y)=-e_1^2-y^{(2)}- x y$.  Thus if  the  graph is resonant we must have    $  y^{(2)}+ x y=0.$  One easily verifies that $  y^{(2)}$ is an irreducible polynomial unless $y$ is of the form
   $y =\beta(e_i-e_j)$. In this case from  the factorization  $  y^{(2)}=- x y$ and the fact that $\eta(x)=-2$ we deduce that $x =-e_i-e_j$.  Since $x-y=-2e_1$  we must have that $\beta=\pm 1$  and if $\beta=1$ we have  $e_i=e_1,\ x +y = -2e_j$.  If $\beta=-1$ we have  $e_j=e_1,\ x +y = -2e_1$.
\end{proof}
We have thus verified  that the graph is not--allowable by Definition \ref{ilpunto1} for the two extremes of the segment $S_A$, a similar analysis would apply to $S_C$.

\subsection{The extra edge}

We treat now case 1) with an extra edge $E=\vartheta e_1-e_h,\ \vartheta=\pm1$.  We have the function $\zeta$  such that $\zeta(e_1)=1,\ \zeta(\ell_i)=0,\ \forall i$ and $\zeta(E)=2\vartheta $.  In this case the  even loop is divided into two  odd paths. We divide the indices different from the two critical indices $1,h$  in  two blocks $A=(2,\ldots,h-1),\ B=(h+1,\ldots,k-1)$ and argue as in the previous section.

From Corollary \ref{ePo} it follows that, either the extra edge  is outside  the segment spanned by the $\ell_i$, this may happen if we are in a situation as (up to symmetry between $A,B$)

$$ a) \xymatrix{   &&\\ \ar@{-}[r]^{E}  &   \ar@{-}[rr]^{S_B} &&  \ar@{-}[rr]^{S_A}  &&   } \quad b) \xymatrix{  && \ar@{-}[d]^{E}  & \\  \ar@{-}[rr]^{S_B} &&  \ar@{-}[rr]^{S_A}  &&   } $$

In these cases the edge $E$ can be removed and the graph is not minimal. Otherwise it could separate the two segments spanned by the two blocks $A,B$ or it could appear in one or both of these segments according to the following pictures:

$$ c) \xymatrix{  &&\ar@{-}[dd]^{S_A}  && &  &\\ &&&&&\\ \ar@{-}[rr]^{S_B} &&\ar@{-}[r]^{E}  &\ar@{-}[dd]^{S_A}\ar@{-}[rr]^{S_B} &&\\  &&&&   \\  &&&& \\&&&&  } $$
$$ d) \xymatrix{  &&\ar@{-}[dd]^{S_A}  && &  &\\ &&&&&\\ \ar@{-}[rr]^{S_B} &&\ar@{-}[r]^{E}  & \ar@{-}[rr]^{S_B} &&\\  &&&&   \\  &&&& \\&&&&  } \quad  e) \xymatrix{  &&\ar@{-}[dd]^{S_B}  && &  &\\ &&&&&\\ \ar@{-}[rr]^{S_A} &&\ar@{-}[r]^{E}  & \ar@{-}[rr]^{S_A} &&\\  &&&&   \\  &&&& \\&&&&  } $$

Cases d), e)  are   special cases of  c), and in fact follow from previous results, so we treat   case c).

  \subsubsection{ $E=e_1-e_h$ is black} We look at the picture c).$$ c) \xymatrix{  &&a\ar@{-}[dd]^{S_A^0}  && &  &\\ &&&&&\\ y\ar@{-}[rr]^{S_B^0} &&r\ar@{-}[r]^{E}  &z\ar@{-}[dd]^{S_A^1}\ar@{-}[rr]^{S_B^1} &&x\\  && &&   \\  &&&b& \\&&&&  } $$

   We can fix the signs $\delta_i$ so that $$\sum_{i=1}^{h-1}\delta_i\ell_i=-e_1-e_h,\quad   \sum_{i=h }^{k}\delta_i\ell_i= e_1+e_h.$$
  Of the  two vertices $y,x$ one is black   the other is red.
  The same for $a,b$.

{\bf  Case 1:\quad  $a,y$  black $b,x$ red}\quad  gives for the various paths:
  $$  S_B^1=z+x,\ S_B^0=y,\ S_A^0=a,\ S_A^1=z+b$$$$y=\sum_{j\in S_B^0}\sigma_j\lambda_j\ell_j= \delta\sum_{j\in S_B^0}\delta_j\ell_j,\quad x=-E-\sum_{j\in S_B^1}\sigma_j\lambda_j\ell_j=-E-\delta\sum_{j\in S_B^1}\delta_j\ell_j$$
$$a=\sum_{j\in S_A^0}\sigma_j\lambda_j\ell_j= \delta'\sum_{j\in S_A^0}\delta_j\ell_j,\quad b=-E-\sum_{j\in S_A^1}\sigma_j\lambda_j\ell_j=-E-\delta'\sum_{j\in S_A^1}\delta_j\ell_j$$
  $$x-y= -\delta \sum_{i\in B}\delta_i\ell_i-E= -\delta (e_1+e_h)-e_1+e_h ,$$$$\ b-a= -\delta' \sum_{i\in B}\delta_i\ell_i-E=  \delta'(e_1+e_h)-e_1+e_h$$ for two signs $\delta,\delta'$. Applying the mass $\eta$  we see that $\delta= 1,\delta'=-1$ hence
   $  x-y=  b-a = -2e_1 $     is the relation among the vertices of the graph.   By resonance
  $$  x -y =  b -a  ,\implies C(x)-C(y)= C(b)-C(a) .$$  We now apply the rules of the operator $C$ to $x $ red, $y$ black  $$ -e_1^2=C(-2e_1)=-C(-y )+C(x )+xy=-C( y)+y^{(2)}+C(x)+xy$$and get that $C(x)-C(y)=-e_1^2+y^{(2)}+ x y=-e_1^2+y^{(2)}+ (y-2e_1) y $.  On the other hand this element is a quadratic polynomial in the elements $e_i$ appearing in the edges of $B$ which must be equal by the resonance relation to a quadratic polynomial in the elements $e_i$ appearing in the edges of $A$. Now the edges of $A$ have in common with the edges of $B$ only the elements $e_1,e_h$, so $-e_1^2+y^{(2)}+ (y-2e_1) y$  must contain only these indices, it easily follows that  if an element $e_i,\ i\neq 1,h$ appears in $y$ with coefficient $\alpha$ we must have $\alpha=-1$, moreover if $e_i$ appears in $y$ no $e_j,\ j\neq 1$ can appear in $y$ otherwise we have a mixed term in $y^2$  of type $2e_ie_j$ which does not cancel.
  Next  we can only have $y=e_1-e_i$  in order to cancel the mixed term  from $-2e_1  y$.

  In this case the segment from $y$ to $x$ has value $x-(-y)=x-y+2y= -2e_1+2(e_1-e_i)=-2e_i$ and the result is proved.\smallskip

  The other possibility is that $y=\alpha(e_1-e_h)$ for some $\alpha$, since $y$ is in any case a sum of edges in  $B$    this is actually not  possible by computing the value of $\zeta$.
 \smallskip

 {\bf   $a,y$  red $b,x$ black}\quad  is symmetric to the previous case.\smallskip

 {\bf  Case 2:\quad   $a,x$  black $b,y$ red}\quad  gives, as in the previous case, the value $b-a=-2e_1$. Then:

  $$  S_B^1+z=x,\ S_B^0=y,\ S_A^0=a,\ S_A^1-z=b$$$$y=-\sum_{j\in S_B^0}\sigma_j\lambda_j\ell_j= -\delta\sum{j\in S_B^0}\delta_j\ell_j,$$$$ x= E+\sum_{j\in S_B^1}\sigma_j\lambda_j\ell_j=E+\delta\sum{j\in S_B^1}\delta_j\ell_j$$
  $$x-y=  \delta \sum_{i\in B}\delta_i\ell_i+E=  \delta(e_1+e_h)+e_1-e_h , $$ by mass  $\delta=1$ and $y-x=-2e_1$, we argue as before.
 \subsubsection{ $E=-e_1-e_h$ is red}

    In this case the  even loop is divided into two  even paths. We can fix the signs $\delta_i$ so that $$\sum_{i=1}^{h-1}\delta_i\ell_i= e_1-e_h,\quad   \sum_{i=h }^{k}\delta_i\ell_i= -e_1+e_h.$$

    We still have a situation as in the previous analysis  with some changes.

  {\bf   Case 1:\quad  $a,y$  black $b,x$ red}\quad  gives for the various paths:
  $$y=\sum_{j\in S_B^0}\sigma_j\lambda_j\ell_j= \delta\sum_{j\in S_B^0}\delta_j\ell_j,\quad x= E-\sum_{j\in S_B^1}\sigma_j\lambda_j\ell_j= E-\delta\sum_{j\in S_B^1}\delta_j\ell_j$$
$$a=\sum_{j\in S_A^0}\sigma_j\lambda_j\ell_j= \delta'\sum_{j\in S_A^0}\delta_j\ell_j,\quad b= E-\sum_{j\in S_A^1}\sigma_j\lambda_j\ell_j= E-\delta'\sum_{j\in S_A^1}\delta_j\ell_j$$
  $$x-y= -\delta \sum_{i\in B}\delta_i\ell_i+E= -\delta (-e_1+e_h)-e_1-e_h ,$$$$ b-a= -\delta' \sum_{i\in B}\delta_i\ell_i+E=  \delta'(e_1-e_h)-e_1-e_h$$ for two signs $\delta,\delta'$.  Thus $x-y,b-a$ can take the values $-2e_1,-2e_h$. If they take the same value we have   $  x-y=  b-a  $  and we argue as in the previous section.  Otherwise up to symmetry we may assume that $  x-y= -2e_1,\  b-a= -2e_h $ and $x-y=b-a+2z$   is the relation among the vertices of the graph.  By resonance
  $$  x -y =  b -a +2z ,\implies C(x)-C(y)= C(b)-C(a) +2C(E)=C(b)-C(a) -2e_1e_h .$$  We now apply the rules of the operator $C$ to $x $ red, $y$ black  $$ -e_1^2=C(-2e_1)=-C(-y )+C(x )+xy=-C( y)+y^{(2)}+C(x)+xy$$and get that $C(x)-C(y)=-e_1^2+y^{(2)}+ x y=-e_1^2+y^{(2)}+ (y-2e_1) y $.  On the other hand this element is a quadratic polynomial in the elements $e_i$ appearing in the edges of $B$ which must be equal by the resonance relation to a quadratic polynomial in the elements $e_i$ appearing in the edges of $A$. Now the edges of $A$ have in common with the edges of $B$ only the elements $e_1,e_h$, so $-e_1^2+y^{(2)}+ (y-2e_1) y$  must contain only these indices, it easily follows that  if an element $e_i,\ i\neq 1,h$ appears in $y$ with coefficient $\alpha$ we must have $\alpha=-1$, moreover two distinct elements of this type cannot appear otherwise we have a mixed term in $y^2$  of type $2e_ie_j$ which does not cancel.
  Next  we can only have $y=e_1-e_i$  in order to cancel the mixed term  from $-2e_1  y$.

  In this case the segment from $y$ to $x$ has value $x-(-y)=x-y+2y= -2e_1+2(e_1-e_i)=-2e_i$ and the result is proved.\smallskip

  The other possibility is that $y=\alpha(e_1-e_h)$ for some $\alpha$, this is possible only if $\alpha=\pm 1$ and $y=\sum_{j\in S_B }\sigma_j\lambda_j\ell_j $ all edges are involved, and $x=z$.  Then the segment from $y$ to $x=z=E$  has values $E-y=-e_1-e_h\pm (e_1-e_h)=-2e_1,\ -2e_h$.
 \smallskip

 {\bf   $a,y$  red $b,x$ black}\quad  is symmetric to the previous case.\smallskip

 {\bf  Case 2:\quad   $a,x$  black $b,y$ red}\quad  gives as in the previous case the value $b-a=-2e_1,\ -2e_h$. Then:

   $$y=-\sum_{j\in S_B^0}\sigma_j\lambda_j\ell_j= -\delta\sum{j\in S_B^0}\delta_j\ell_j,$$$$ x= -E+\sum_{j\in S_B^1}\sigma_j\lambda_j\ell_j=-E+\delta\sum{j\in S_B^1}\delta_j\ell_j$$
  $$y-x=  -\delta \sum_{i\in B}\delta_i\ell_i+E=  -\delta(-e_1+e_h)-e_1-e_h \in\{-2e_1,-2e_h\}. $$
  We argue again as before.
\vfill \eject

 \part{The irreducibility Theorem}
 \section{The matrices}  The operator $ad(N)=2\ii  Q$  under study acts on the space spanned by the frequency basis and here it decomposes into blocks corresponding to the connected components of the Cayley graph $G_X$ restricted by Defnition \ref{gliegg}(Theorem \ref{iblo}).

 For each such component $A$  we have seen that $Q$ acts as a scalar $K(a)$  plus a matrix $C_A$  homogeneous of degree 1 in the variables $\xi_i$. According to Formulas \eqref{maent}, \eqref{maent1}, \eqref{maent2}
the entries of $C_A=(c_{a,b})$ are the following.   If $a\in A,\ a=\sum_ia_ie_i\in \Z^m$ the diagonal entry $c_{a,a}=-a(\xi)=-\sum_ia_i\xi_i$.
 If $a\in A,\ a=(\sum_ia_ie_i)\tau\in \Z^m\tau$ the diagonal entry $c_{a,a}= a(\xi)= \sum_ia_i\xi_i$.

 If $a,b\in A$ are not connected by an edge $c_{a,b}=0$.  If $a,b\in \Z^m$ are connected by a black edge $e_i-e_j$  then $c_{a,b}=2\sqrt{\xi_i\xi_j}$, if $a,b\in \Z^m\tau$ are connected by a black edge $e_i-e_j$  then $c_{a,b}=-2\sqrt{\xi_i\xi_j}$, finally if $a,b$ are connected by a red edge $-e_i-e_j$  then one of them is in $\Z^m$ the other in $\Z^m\tau$ and we have $c_{a,b}=-2\sqrt{\xi_i\xi_j}$ if $a\in\Z^m,\ b\in \Z^m\tau$ and   $c_{a,b}=2\sqrt{\xi_i\xi_j}$ in the other case.  If red edges are not present the matrix is symmetric.

 Notice then some rules, if  $b\in\Z^m$  we have  $C_{Ab}=C_A-b(\xi)Id$, finally $C_{A\tau}=-C_A$.

By Lemma \ref{spr}, when we expand the characteristic polynomial of such a matrix the square roots disappear and we get a polynomial, denoted $\chi_A(t)$ (or sometimes just $\chi_A$) monic in $t$ and with coefficients polynomials in the variables $\xi_i$ with integer coefficients.  Our goal is to prove that \begin{theorem}[irreducibility theorem]
\label{irter}If $A$ is a non--degenerate allowable graph in $G_X$  the  polynomial  $\chi_A(t)$ is irreducible  as  polynomial in $\Z[t,\xi]$. \end{theorem} We prove furthermore that the graph $A$ is determined by $\chi_A(t)$, this we call the {\em separation lemma} \ref{seplem}.

 In fact in this form the statement is not true, we need to use the fact that mass is conserved.  This is enough for the dynamical consequences. In algebraic terms the  conservation of mass consists in restricting to the coset of $G_2$ (one of the connected components of the Cayley graph) of elements $a, a\tau\in G,\ a\in \Z^m,\ \eta(a)=-1$.    We also need to use systematically Theorem \ref{MM}  which tells us that we can restrict to those graphs in which the vertices are affinely independent.
 \begin{remark}
The hypothesis that the graph is non--degenerate is necessary.  In the simple example of
$$ \xymatrix{&0\ar@{->}^{1,  2}[r]  &e_2-e_1\ar@{->}^{1,  2}[r]   &2e_2-2e_1}$$one easily verifies that the characteristic polynomial is not irreducible.
\end{remark}
On the other hand it is likely that the condition to be allowable is not necessary in order to prove irreducibility and separation. To avoid it complicates the proofs and, since we do not need the stronger result, we have not tried to discuss it.
 \section{Irreducibility and separation\label{irresp}} \subsection{Preliminaries} Observe first that,   given $g\in G,  A\subset G$ we have that  $\chi_A  (t)$ is irreducible if and only if $\chi_{Ag}  (t)$ is irreducible.

  Consider a projection $\pi_i:\Z^m\rtimes \Z/(2)\to \Z^{m-1}\rtimes \Z/(2)$  where  we remove the $i^{th}$ coordinate $\pi_i  [(a_1,  \ldots,  a_m),\delta]\mapsto   [(a_1,  \ldots,  \check a_i,  \ldots a_m),\delta]$.   Take now  a set $A\subset \Z^m\rtimes \Z/(2)$ of vertices  and consider the graph  obtained from $\Gamma_A$ by removing all the edges which contain $i$ in its marking,   call this new graph $\Gamma_A^i$.  Even if $A$ is connected this new graph $\Gamma_A^i$ may well not be connected.   We now claim
\begin{proposition}
If $A$ is connected the map $\pi_i$,   restricted to $\Gamma_A^i$,
is injective and a graph isomorphism with $\Gamma_{\pi_i  (A)}$,
a graph in $\Z^{m-1}\rtimes \Z/(2)$.

If $A$ is non degenerate each connected component of  $\Gamma_{\pi_i  (A)}$ is non degenerate.
\end{proposition}
\begin{proof}
We  know that the mass $\ell=\eta  (a)$  depends only on
the color of $a$ so that we have $a_i=\eta  (a)-\eta  (\pi_i (a))$
and thus if $a,  b$ are black vertices  (or red vertices), $\pi_i
(a)=\pi_i  (b)$:$\eta  (a)=\eta  (b)$ hence $a_i=b_i\implies a=b$.
Otherwise,   if $a$ is black,   $b$ is red then it is clearly
$\pi_i (a)\neq\pi_i  (b)$ because $\pi_i  (a)$ is black,   $\pi_i
(b)$ is red. If we decompose $X=X_m$ into the elements containing
the index $i$ and the complement $X_m^i$ we see that $\pi_i$
establishes a 1--1 correspondence between $X_m^i$ and $X_{m-1}$
from which the second claim since $\pi_i$ is a group homomorphism.  The third claim follows easily from
the definitions.
\end{proof}
A simple corollary of this proposition is that.
\begin{corollary}\label{laffa}
If we set $\xi_i=0$ in the matrix $C_A$ we have the matrix
$C_{\pi_i  (A)}$,   hence
$$\chi_A  (t)|_{\xi_i=0}=\chi_{\pi_i  (A)}  (t) $$
 Let $B_1,  \ldots,  B_k$ be the connected components of  $\pi_i  (A)$.  We have  $$\prod_{j=1}^k\chi_{B_j}  (t)=\chi_{\pi_i  (A)}  (t) =\chi_A  (t)|_{\xi_i=0}. $$
\end{corollary}
As a consequence,   we have the following inductive step.
\begin{corollary}
\label{iste} Assume that $A$ is non degenerate and that we have already proved the irreducibility theorem  for $m-1$ or for  $n<|A|$.  We deduce that the factors $\chi_{B_j}  (t)$ of $\chi_{\pi_i  (A)}  (t) $ are the irreducible monic factors of $\chi_A  (t)|_{\xi_i=0}$.
\end{corollary}

We want to prove Theorem \ref{Lase} by induction as follows.
We assume irreducibility and separation in dimension $n-1$ and prove first the  separation in dimension $n $ and  finally irreducibility   in dimension $n. $

 Take a connected $A$ and let $\ell$ be the mass of a
black vertex of $A$,   then the mass of a red vertex is
$-2-\ell$.
\begin{lemma}[Parity test]\label{parita}
\begin{enumerate}
\item If we compute  $t$ at a number $g\not\cong\ell\ \mod  (2)$,    we have  $\chi_A  (g)\neq 0.$

\item If a linear form $t+\sum_ia_i\xi_i,  \ a_i\in\Z$ divides $\chi_A  (t)$ we must have $\sum_ia_i \cong\ell\ \mod  (2)$.
\end{enumerate}
\end{lemma}
\begin{proof}
i)   The matrix $C_A$ modulo 2 is diagonal and $\chi_A
(t) \cong  \prod_i (t+a_i(\xi)) \ \mod  (2)$. If we compute modulo 2 and  set all $\xi_i=1$,   we get $\chi_A
(t) \cong   (t+\ell)^m\ \mod  (2)$,   hence  $\chi_A  (g) \cong
  (g+\ell)^m\cong g+\ell\ \mod  (2)$.

ii) A linear form $t+\sum_ia_i\xi_i,  \ a_i\in\Z$ divides $\chi_A
(t)$ if and only if  we have $\chi_A  (-\sum_ia_i\xi_i)=0$,   then
set $\xi_i=1$ and use the first part.
\end{proof}

We shall use the parity test  as follows. \begin{lemma}\label{supertest}
  Suppose we have a connected set $A$ in $\Z^m$,   in which we find a vertex  $a $ and an index,   say 1,    so that  the graph $\Gamma_A$  has the following properties:
 $$ \xymatrix{&&c&\\ \ldots &d&a\ar@{. }^{1,  h}[r]  \ar@{. }^{1,  h}[r] \ar@{. }^{1,  k}[l]\ar@{. }^{1,  j}[d]\ar@{. }^{1,  i}[u]&b \ldots   &\ldots \\&&e&\  }$$
  we have:
\begin{itemize}
\item $1 $ appears in all and only the edges  having $a$ as vertex.
\item When we remove $a$   (and the edges meeting $a$) we have    a connected graph ${\A}$ with at least 2 vertices.
\item When we remove the edges associated to any index,   the factors described in Corollary \ref{laffa}  are irreducible.
\end{itemize}
  Then the polynomial   $\chi_A  (t)$ is irreducible.
\end{lemma}
\begin{proof}
We take $a$ as root,   and translate the set $A$ so that $a=0$.
Setting $\xi_1=0$ we have by Corollary \ref{laffa} and the
hypotheses,   that  $\chi_A  (t)=t\,    P  (t)$  with $P=\chi_{\A}
(t)$ irreducible of degree $>1$.    Thus,   if the polynomial
$\chi_A  (t)$ factors,   then it must factor into a linear $t-L
(\xi)$ times an  irreducible polynomial of degree $>1$.

Moreover modulo $\xi_1=0$ we have that $0$ and $\ell$ coincide,
thus $L  (\xi)$ is a multiple of $\xi_1$.

Take another index $i\neq 1,  h$ if $a$ is an end and the only
edge from $a$ is marked $  (1,  h)$ otherwise just different from
$1$ and   set $\xi_i=0$.  Now    the polynomial $\chi_A  (t)$
specializes to the product $\prod_j\chi_{A_j}  (t)$ where the
$A_j$ are the connected components of the graph obtained from $A$
by removing all edges in which $i$ appears as marking.  By
hypothesis $\{a\}$ is not one of the $A_j$.

If no factor is linear we are done.  Otherwise there is an isolated vertex $d\neq a$   so that $\{d\}$ is one of the connected components $A_j$.  The linear factor associated is $t+d  (\xi)|_{\xi_i=0}$.   Clearly    we have that the coefficient  of  $\xi_1 $ in  $d  (\xi)$ is $\pm 1$   (since the marking 1 appears only once).     This implies that $L  (\xi)=\pm\xi_1$  and this is not possible by the parity test. \end{proof}

\section{The separation lemma\label{SEPA}}
Given a connected  graph $G\subset G_X$ consider  $\tau
G=\{ (-a, -\delta)| (a, \delta)\in G\}$.
\begin{remark}\label{rem3}$\tau G$ is a connected graph,  if and only if $G$ contains only black edges. \end{remark}
\begin{proof} The connected components of the Cayley graph are the cosets $G_2u,\ u\in G$. If there exists a red edge $(-e_i-e_j,\tau)$ connecting    two elements $a$, $b\in G$ then $ba^{-1}  =(-e_i-e_j,\tau)\Rightarrow \tau b(\tau a )^{-1}=(e_i+e_j,\tau)\notin G_2$. $\tau b,\tau a $  are not in the same connected component of the Cayley graph.
Instead $b a^{-1}=e_i-e_j\Rightarrow   \tau b(\tau a )^{-1}=e_j-e_i$,  $\tau a, \tau b$ are
connected by a black edge marked $j, i$ in $\tau G$.
\end{proof}
\begin{lemma}\label{seplem} (Separation lemma)\quad Given two connected non--degenerate allowable graphs $G_1, G_2\subset G_X$ if $\chi_{G_1}=\chi_{G_2}$,  then $G_1=G_2$ or
$G_1=\tau G_2$. \end{lemma}
If we take $G\subset G^1$, then $G$ is of mass  $-1$ we have that $\tau G$ is of mass 1, we deduce that a connected color marked
graph  $G$ of mass -1   can be recovered    from its characteristic
polynomial.
\begin{proof}We will prove this lemma by induction.  When $n=0:
\chi_G (t)=t+a$,  it is easy to see that $G=\{ (a, +)\}$ or $G=\{
(-a, -)\}$.  \newline Induction process: $n>1$. Suppose that we
have the separation and the irreducibility for graphs of
dimensions $k\leq n-1$.  Take a connected colored marked graph $G=\{
(v_1, \delta_1 ), \ldots,  (v_{n+1}, \delta_m)\}$, $ (v_i,
\delta_i)\in \mathbb Z^m \rtimes  \mathbb{Z}/ (2)$,  the associated
matrix $C_G$ and its characteristic polynomial $\chi_G$.   

Associate to $G$ the list $L$ of vectors $w_i:=\delta_iv_i$, we see that these vectors are affinely independent. If the $w_i$ have all the same mass then the graph $G$  has only black edges  and then it is either the graph with vertices $w_i$ or with vertices $\tau w_i$ as seen before, if they have different masses  then the masses are of type $k$ for black vertices and $k+2$ for red and the graph $G$ is thus reconstructed  from $L$

Therefore we need to show that, from the characteristic polynomial, we can recover the list $L:=\{w_1,\ldots,w_n\}$.
Before starting the proof let us make a useful remark, the characteristic polynomial  gives as information the trace of the matrix $C_G$  and thus in particular the sum $\sum_{i=1}^n w_i(\xi)$ and the mass $s:=\sum_{i=1}^n \eta(w_i ).$   If we have $a$  elements in the list of mass $k$ and $(n-a)$ of mass $k+2$  we have that $s=nk+2b=n(k+2)-2(n-b)$.  Thus if we know that a certain number $h$ is the mass of a vertex we can deduce
\begin{lemma}\label{lams}
If   $s=nh$  then all  vertices in $G$ have the same color.  If $nh<s$ then $h$ is the mass of the black vertices and there are $b$  red vertices where $s=nh+2b$. Similarly if   $nh>s$ then $h$ is the mass of the red vertices and there are $b$  red vertices where $s=nh-2(n-b)$. 
\end{lemma}

  We set
one of the variables $\xi_i=0$ for instance $\xi_1=0$.  We know
that the matrix $C_G$ specializes to the direct sum of the
matrices $C_{G_i}$ where the $G_i$ correspond to the various
connected components of the graph $G$ which are obtained by
removing all edges in which 1 appears as marking  and dropping in
each component the first coordinate of the various vertices.  We
have that specializing $\xi_1=0$ we specialize the polynomial
$\chi_G$ to $\prod_i\chi_{G_i}$.  Since we are assuming
irreducibility in dimensions less than $n-1$  the factors
$\chi_{G_i}$ are all irreducible and thus can be determined by the
unique factorization of polynomials.  Therefore all the vectors of
$\pi_1 (L)$,  that is the $w_i$ with the first coordinate removed
can be recovered uniquely (up to the sign) by induction and we obtain a list of $n$ vectors  $L^1:\{  
(*, b_i, c_{3, i}, . . . , c_{m, i} )\}$.
\newline Now we set another variable,  say $\xi_2=0$.  By similar
arguments as above all the $w_i$ with the second coordinate
removed can be recovered   by induction giving a list $L^2:\{ 
(a_i, *, c_{3, i}, . . . , c_{m, i} )\}$.\smallskip

Now our problem is
this: if we know the vectors obtained from $L$
after removing the first or the second coordinate can we recover
the given vectors? We shall need to perform a case analysis.\smallskip

1) Recovering the list $L$: \newline 

We thus 
  consider the vectors $L^{1,2}$ obtained from $L$  by dropping the first two coordinates $  (*,
*, c_3, . . . , c_m )$ and collect the ones   where $c_3, \ldots, c_m $ are
fixed.  The first remark is that, if in this list a given vector  $  (*,
*, c_3, . . . , c_m )$ appears only once then we know exactly from which vector it comes from the two lists $L^1,L^2$    and so we can reconstruct  the vector $v$ in $L$ from which it arises.  Then by       Lemma \ref{lams} we can determine if in the graph all vertices have the same color or, if this is not the case, which is the mass of the black end red vertices and how many there are.

Next 
since the vectors in the graph, by assumption, are affinely
independent,  we have at most 3 vectors in $L$, giving the same vector   $ (*, *,
c_3, . . . , c_m)$ in $L^{1,2}$  since 4 of such vectors lie in a 2--dimensional plane so they are not affinely independent. 

a)\quad Assume we have 3 vectors $v_1,v_2,v_3\in L$ giving the same vector $\underline c= (*, *,
c_3, . . . , c_m)$ in $L^{1,2}$ and let $c=\eta(\underline c)$. We claim that $v_1,v_2,v_3 $ cannot have the same color, in fact this would imply that they have the same mass and then they lie in a line and cannot be affinely independent.  Let then $a_1,a_2,a_3$ resp. $b_1,b_2,b_3$ be the first, resp. second coordinates of these vectors (deduced from the two lists  $L^1,L^2$) we need to be able to reconstruct the 3 vectors $v_1,v_2,v_3\in L$ by matching the $a_i$ with the $b_j$.  First observe that we know the total mass  $m$ of $v_1,v_2,v_3$. This is $m=3k+2$ or $m=3k+4$ depending if we have two or 1 black vertices among $v_1,v_2,v_3 $. Since $3k+2$ is congruent to 2 modulo 3 while  $3k+4$ is congruent to 1 modulo 3, we can deduce both $k$ and the number of black vertices from $m$.

Call $l:=k-c$, now consider one of the vectors in $L^1$, start from $(a_1, *,
\underline c)$, if there is no $b_i$ with $a_1+b_i=l$ then there must necessarily be one, say $b_1$ with $a_1+b_1=l+2$ and then 
$(a_1, *,
\underline c)$ comes from the red vector $(a_1, b_1,
\underline c)$.  Similarly  if there is no $b_i$ with $a_1+b_i=l+2$ then there must necessarily be one, say $b_1$ with $a_1+b_1=l $ and then 
$(a_1, *,
\underline c)$ comes from the black vector $(a_1, b_1,
\underline c)$.   In this case we can easily see how to match the other two vectors, in case the other two vectors have the same color  we must match them so that $a_2+b_i=l',a_3+b_j=l'$ where $l'=l$ if the color is black and $l+2$ if red. We claim that  only one match is possible, in fact if  we had  $a_2+b_3=a_3+b_2=a_2+b_2=a_3+b_3$ we would have that the two vectors $v_2,v_3$ coincide.

Suppose now we know that the two colors are distinct, then as before,  if there is no $b_j, j=2,3$ such that $a_2+b_j=l$ we know that there is one, say $b_2$ for which $a_2+b_2=l+2$ and we have reconstructed the two vectors $(a_2,b_2,
\underline c),  (a_3, b_3,
\underline c)$. Finally it is possible that $b_3=b_2+2$ and $a_2+b_2=l$ then we have $a_3+b_3=l+2$ which implies $a_3=a_2=a$ and again we reconstruct the two vectors (actually by Definition \ref{ilpunto1} this is not allowed).

It remains to  analyze the case in which none of the $a_i$  satisfies the condition that it cannot be paired uniquely.

So let us assume that, up to reordering $b_1$ is maximum. There is one $a_i$ which must be paired with $b_1$  and we are assuming that it can also be paired with another $b_i$ giving a different color.  We must necessarily have that the value of this $a_i$, which we may assume reordering to be $a_1$ is $a_1=l+2-b_1$, we have recovered a red vector   $(a_1,b_1,\underline c)$. The rest of the analysis follows as before.
\medskip

b) There are in $L^{1,2}$ only 2 vectors of the form $  (*, *, c_3,
..., c_m )$ with $c_3, ...,c_m $  fixed. For
simplicity we denote $\underline c:= (c_3, . . ,c_m)$ and their sum
by $c$.  We know  then two vectors in $L^{1,2}$ of the form $(a_1,*,\underline c),(a_2,*,\underline c)$ and  two vectors in  $L^{ 2}$ of the form $(*,b_1, \underline c),(*,b_2, \underline c)$ which specialize in $L^{1,2}$ to the given vectors.

A priori in $L$ we can either have $(a_1,b_1,\underline c),(a_2,b_2,\underline c)$ or  $(a_1,b_2,\underline c),(a_2,b_1,\underline c)$. The first pair gives two vertices of the same color if and only if  $a_1+b_1=a_2+b_2$, similarly for the second.  If we have $a_1+b_1=a_2+b_2,  a_1+b_2=a_2+b_1$ we deduce that $a_1=a_2, b_1=b_2$  and this is impossible since it implies that in $L$ we have two equal vectors, therefore in at least one of the  two pairs  we have different colors.  We may thus assume (changing the indices if necessary)  that $ a_1+b_2=a_2+b_1+2 $, this implies $ a_1-a_2=b_1-b_2+2$. Write $a_1+b_1=a_2+b_2+x,\ x\in(-2,0,2) $ and thus $2(b_1-b_2)=x-2$. 
If $x=-2$  we have $b_1-b_2=-2,a_1=a_2$  and we argue as before, this case is impossible.

If $x=2$  we have $b_1=b_2=b,  a_1=a_2+2=a+2$  we have in the possible list  of  vectors $(a+2,b ,\underline c),(a ,b ,\underline c)$.  We know that this list is not allowed by Definition \ref{ilpunto1}.  Assume that $x=0$ thus $b=b_1, b_2=b+1, a=a_2,a_1=a+1$ we have  the two possibilities 1) $(a+1,b ,\underline c),(a ,b+1,\underline c)$ or  2) $(a+1,b+1,\underline c),(a ,b ,\underline c)$.  In this case both cases are a priori possible, in fact if  the graph were just a single edge marked $ e_1-e_2 $ or $-e_1-e_2$ the two  cases cannot be recovered by the two specializations but only from the full characteristic polynomial. 
$$  
 G_1= \xymatrix{  (e_1,+)\ar@{->}[r]^{e_2-e_1 }&  (e_2 ,+)  &  G_2=\ (0,+)  \ar@{=}[r]^{-e_2-e_1 } &  (-e_1-e_2,-)}, 
$$  
  \begin{equation}\label{22}C_{G_1 }= \begin{vmatrix}
- \xi_1& 2\sqrt{\xi_1\xi_2}\\&&&\\ 2\sqrt{\xi_1\xi_2}&- \xi_2  
\end{vmatrix},\quad C_{G_2 }= \begin{vmatrix}
0&  -2\sqrt{\xi_1\xi_2}\\&&&\\   2\sqrt{\xi_1\xi_2}& -\xi_1-\xi_2  
\end{vmatrix}  \end{equation}
The characteristic polynomials are distinct:
$$t^2+(\xi_1+\xi_2)t-3\xi_1\xi_2,\quad  t^2+(\xi_1+\xi_2)t+4\xi_1\xi_2$$ but the two specializations coincide.

%In fact in the first case we have the list $(a+1,b-1,\underline c), (a ,b  ,\underline c)$ which gives the two lists $[(0,b-1,\underline c), (0 ,b ,\underline c)]; [(a+1,0,\underline c), (a ,0,\underline c)] $ in the second 
% we have the list $[(a-1,b-1,\underline c), (a ,b ,\underline c)]$ (or the symmetric) which gives the two lists $[(0,b-1,\underline c), (0 ,b ,\underline c)]; [(a-1,0,\underline c), (a=(a-1)+1 ,0,\underline c)] $  on the other hand one can see how to reconstruct the edge from the characteristic polynomial, for instance in the case of a black edge this polynomial has always real roots for all real values of the $\xi$ while for a red edge there are regions with complex values.\smallskip
 
 So we need a deeper analysis.
 First let us assume that we know if all the vectors have the same mass or we know the mass of black and red vertices.
 
 If we know that all vertices have the same mass then case 2) is excluded. Suppose then that we know the mass $k$ of a black vertex.
 
 If case 1) holds we must have that $a+b+c $ is either $k-1$ or $k+1$, if case 2) holds we must have that $a+b+c=k$. Thus we can determine  in which case we are.
 
 The other possibility is that we do not have the previous information but by the previous analysis this means that in the list $L^{1,2}$ each vector appears twice.  If the list consists of just two vectors we can conclude by the explicit formulas of the characteristic polynomial.
 
Assume we have at least two pairs one $u_1,u_2$ giving $(*,*, \underline c)$ the other $v_1,v_2$ giving $(*,*, \underline d)$. In each case we know that the two vertices are connected either by the edge $e_1 -e_2$ or by $-e_1-e_2$. We deduce that the only possibility at this point is that there are only two such lists  so $L$ has 4 elements and we must have both edges  $e_1 -e_2$ and $-e_1-e_2$.

The two edges  involve two disjoint pairs of vertices so that the graph must be  of the form 
$$ \xymatrix{ &a\ar@{->}[r]^{\pm(e_1 -e_2)} &b  \ar@{- }[r]^{\ell} &c\ar@{=}[r]^{-e_1-e_2} &d&   }$$ if $\ell$ does not contain any of the indices $1,2$ or possibly of  the form

$$ \xymatrix{  a\ar@{->}[d]_{\pm(e_1 -e_2)}  \ar@{- }[rd]&&\\ b  \ar@{- }[r]^{\ell} &c\ar@{=}[r]^{-e_1-e_2} &d&   } \xymatrix{   & &   c\ar@{- }[ld]\ar@{=}[d]^{-e_1-e_2}\\ a\ar@{->}[r]_{\pm(e_1 -e_2)}&b \ar@{- }[r]^{\ell}  &d&   }\xymatrix{   &    a \ar@{->}[d]_{\pm(e_1 -e_2)}\ar@{-}[dr] &c\ar@{- }[ld]\ar@{=}[d]^{-e_1-e_2}\\  &b \ar@{- }[r]^{\ell}  &d&   }$$  if $\ell$   contains one of the indices $1,2$. The edge $l$ can have either color (which determines the color of the further edge).

In particular the graph  has either 3 black and one red vertex or 3 red and one black vertex so either $s= 4k+6=4(k+1)+2$  or $s=4k+2$.  

This gives two possible values for the mass of black vertices, $k$ or $k+1$. Finally specializing  to $\xi_i=0$ where $i\neq 1,2$ appears in $\ell$ and to $\xi_1=0$ (or $\xi_2=0$) if 1 resp. 2 does not appear in $\ell$  we see that of the 4 vectors in $L^{1,2}$ at least one appears only once and we are back in the previous case which we have treated.

    \end{proof}\section{Irreducibility theorem}
    %\subsection{Multiplicity of indices}
We prove Theorem \ref{irter} by induction. Assume the separation and
irreducibility in all dimensions less than $n$,  we will prove the
irreducibility in dimension $n$. Since this property is invariant
under translation we often choose a vertex as the root and assume
that it corresponds to 0. We thus always deal with combinatorial graphs and we may identify  the black vertices as elements $a$ in $\Z^m$
with $\eta(a)=0$ and  the red vertices as elements $a$ in $\Z^m$
with $\eta(a)=-2 $ (Remark \ref{G2}). \smallskip

 Therefore from now on we assume that $G$ is a combinatorial graph with $n+1$ vertices and $T$ a maximal tree in $G$  with $n$  linearly
independent  edges.
 \begin{lemma}\label{le4p}
We  have one of the following possibilities:
\begin{enumerate}
\item We have $n$ indices all with multiplicity 2.
\item We have at least two indices  with multiplicity 1  in distinct edges.\item We have  two  indices  with multiplicity 1  in the same edge the remaining  with multiplicity 2.
\item We have one index with multiplicity 1  one with multiplicity 3 and the remaining  with multiplicity 2.

\end{enumerate}
\end{lemma}
\begin{proof}
 We must have at least $n$ distinct indices
appearing in the edges,
 otherwise these edges span a subspace of dimension less than
$n$.  In total on the $n$ edges of $T$ appear $2n$ indices counted with multiplicity.  If every index appears with multiplicity $\geq 2$  we must have   $n$ indices all with multiplicity 2.

 If we have at least 3 indices of multiplicity 1 we are in case ii), if we have only two indices of multiplicity 1 in the same edge, the remaining indices satisfy property i) for the remaining $n-1$ edges.
Assume finally that only one index appears with multiplicity 1. Of the remaining
$k\geq n-1$ indices appearing assume $a$ have multiplicity $\geq 3$ and $b$ multiplicity $2$  hence $$a+b\geq n-1, 3a+2b\leq 2n-1\implies   b\geq n-2$$  we deduce that $a=1$  and the multiplicity is 3, we are in the last case.

 \end{proof}

We thus have to treat 4 cases.
\begin{remark}
\begin{itemize}
\item Dash lines mean that they may be black or red.
\item Black edges are denoted by single lines,   red edges-by double
lines.
\item $\bar{A}$   denotes the completed graph obtained from
the graph $A$.
\end{itemize}
\end{remark}
Sometimes given a  combinatorial graph $G$  by a {\em block} $A$ of $G$ we mean a connected complete subgraph $A$ of $G$.  If $A$ is a block in a maximal tree $T$ of $G$  the completion $\bar A$ is a block in $G$. By abuse of notation we denote by $\chi_A(t):=\chi_{\bar A}(t)$ to be the characteristic polynomial of the matrix associated to $\bar A$. We now fix a maximal tree in $G$.
 \begin{lemma}\label{lem0}
If in $T$ there are two blocks $A,  B$ and two indices $i,  j$
such that:\begin{enumerate}\item $i,  j$ do not appear in the
edges of the blocks $A,  B$.
\item
\begin{equation}\label{05}\chi_{\bar{A}}\cong \chi_{\bar{B}}\ \text{modulo}\  {\xi_i=\xi_j=0},\end{equation}
\end{enumerate}  then $|B|=|A|=1, \ A=\{(a,\sigma_1)\}, B=\{(b,\sigma_2)\},\ a,b\in\Z^m$ and
$b=\ell+\sigma_2\sigma_1a$.  Where $\ell=n_ie_i+n_je_j, n_i+n_j= -1+\sigma_2\sigma_1$. 

Assume that $i,j$ appear at most twice in the tree then  if $\sigma_2\sigma_1=1$ we may have $\ell=\pm(e_i-e_j),\pm2(e_i-e_j)$. If $\sigma_2\sigma_1=-1$ we may have $\ell=-e_i-e_j,-2e_i,-2e_j$. \end{lemma}
\begin{proof} Since the degree of the characteristic polynomial is the number of vertices by assumption  $|B|=|A|$.  Choose the root in $A$. This gives to each vertex $v$ a sign $\sigma_v$. Let $A=\{  (a_1,  \sigma_1),  . . . , (a_r,
\sigma_r)\}; B=\{  (b_1,  \delta_1),  . . . ,    (b_r,
\delta_r)\}$, then to these graphs we associate  as in \S \ref{SEPA} the list $L$ of vectors $v_h=\sigma_ha_h$ and $w_h= \delta_h b_h$.
Since $i,  j$ do not appear in $A$   (resp. $B$),   the vectors $v_h$  have the same $i$-th and
$j$-th coordinates and we can write $v_h=\bar v_h+a$,   similarly for $B$  the vectors $w_h=\bar w_h+b$  where $a,b$ are linear combinations of $e_i,e_j$ and $\bar v_h,\bar w_h$  are linear combinations of the $e_s, s\neq i,j$.  

The list of vectors $\bar v_h$ is the one associated to the graph $\bar A$ once we set equal to 0 the elements $e_i, e_j$ hence it is the list of vectors associated to the polynomial $ \chi_{\bar{A}}|_{\xi_i=\xi_j=0}$ similarly $\bar w_h$ is the one associated to  $\chi_{\bar{B}}|_{\xi_i=\xi_j=0}$. Hence by the separation lemma  up to reordering we may assume that $\bar v_h=\bar w_h$ hence $v_h=w_h+c,\ c=a-b=n_ie_i+n_je_j$.

Clearly if  $r>1$ we have that $w_r=w_1-v_1 +v_r$ so that the vectors  $(v_h, w_k)$ are not affinely independent contrary to the hypotheses.

We have thus proved that $|B|=|A|=1 $ hence $ A=\{(a,\sigma_1)\}, B=\{(b,\sigma_2)\}$ and finally $b=n_ie_i+n_je_j+\sigma_2\sigma_1a$. Of course  $n_ie_i+n_je_j$ is the value up to sign of the path joining $a,b$. If $\sigma_2\sigma_1=1$  we have $\eta(a)=\eta(b)$ hence $\ell=n(e_i-e_j)$  if  both indices $i,j$   cannot  appear more than twice in the path we have $|n|\leq 2$. If $\sigma_2\sigma_1=-1$  we have $\eta(a+b)=-2$ hence $\ell=n e_i-(n+2)e_j$.  A similar case analysis gives the possibilities $\ell=-e_i-e_j,-2\e_i,-2e_j $ if  both indices $i,j$   cannot  appear more than twice in the path.   \end{proof}
\begin{corollary}\label{corlem0}
Under the assumptions of Lemma \ref{lem0}  the number of edges in the path from $a$ to $b$  in which appears any marking $h\neq i,j$ must be even. The parity of the number of edges in which appears $i$ equals the  parity of the number of edges in which appears $j$.
\end{corollary}
In  a maximal tree $T$    in a graph $\Gamma$  consider an edge $\ell$  
  containing the indices $i,  j$. Denote by
 $A,  B$ the two connected components   obtained by removing $\ell$ from  $T$.
\begin{lemma} Assume that the two connected components $A,  B$ do not have the index $i$ in any edge.
  Then any other edge in $\Gamma$ connecting $A,  B$  must contain the index $i$.
\end{lemma}
\begin{proof}
In a  path which is a circuit you cannot have that an index  appears only once   (or even an odd number of times).
\end{proof}
We now consider two edges $\ell_1,  \ell_2$  containing the
indices $i,  h$ and $i,  k$ respectively.  When we remove these
edges in $T$ we have 3 connected components in $T$ $$A\stackrel{i,
h}\ldots B\stackrel{i,  k}\ldots C$$ in the complete graph $\bar
T$ once we remove all the edges  containing $i$  the graph $\bar
B$ is a connected component.   Then we may either have other 2
components $\bar A,  \bar C$  or a connected component
$\overline{A\cup C}$. We shall use this fact systematically as follows. By induction in the first case we have $\chi_G(t)|_{\xi_i=0}=\chi_{\bar{A}}(t)  \chi_{\bar{B}}(t)  \chi_{\bar{C}}(t)\ \text{modulo}\ {\xi_i=0}$ is a factorization into irreducible factors, in the second case a factorization into irreducible factors  is $\chi_G(t)\cong \chi_{\overline{A\cup C}}(t) \chi_{\bar{B}}(t)\ \text{modulo}\  {\xi_i=0} $. 

Hence if $G$ is not irreducible in the second case it can only factor into two irreducible factors $\chi_G(t)=UV$  with $U\cong \chi_{\bar{B}}(t) ,V\cong\chi_{\bar{A\cup C}}(t)\ \text{modulo}\  {\xi_i=0}$, in the first case we may have either a factorization into 3 irreducible factors $\chi_G(t)=UVW$  with $U\cong\chi_{\bar{A}}(t) ,V\cong\chi_{\bar{B}}(t) , W\cong\chi_{\bar{C}}(t)\ \text{modulo}\  {\xi_i=0}  $ or 3 possible factorizations into 2 irreducible factors.
\subsection{Indices appearing once}
\begin{lemma}\label{angle} If there exists a pair of indices,   say $  (1,  i)$,   such
that 1 appears only once in the maximal tree $T$ and $T$ has the
form:
\begin{figure}[H]
\begin{center}$$\xymatrix{A\ar@{--}[r]^{1,  h}&B}$$\caption{}\label{pic_angle}\end{center}
\end{figure}\noindent where $i\neq h$,   and $i$ appears only in the block
$B$. Then $\chi_G$ is irreducible.
\end{lemma}
\begin{proof}Let the root be in $A$.  Since 1 appears only once in $T$,   every edge in $G$ that connects $A$ and $B$ must have 1 in the indexing. We have:\begin{equation}\chi_G\cong \chi_{\bar{A}}\chi_{\bar{B}}\  \text{modulo}\   {\xi_1=0}. \end{equation} By the previous discussion   if $\chi_G$ is not irreducible,   it must factor into two irreducible polynomials: $\chi_G=UV  $ such that $U\cong \chi_{\bar{A}} \  \text{modulo}\   {\xi_1=0}. $

Let $B_1,  . . . ,  B_s$ be the connected components obtained from
$B$ by deleting all the edges which have $i$ in the indexing,
$B_1$ be the component that is connected with $A$.  We
have:\begin{equation}\chi_G\cong\chi_{\overline{A\cup
B_1}}\chi_{\bar{B_2}} . . . \chi_{\bar{B_s}}\  \text{modulo}\   {\xi_i=0}    .
\end{equation} Remark that $deg  (U)=|A|<deg
(\chi_{\overline{A\cup B_1}})=|A|+|B_1|$.
$U\cong \chi_{\bar{A}}$ is irreducible    modulo ${\xi_1=\xi_i=0}$,   then
$U $ must be irreducible modulo $ \xi_i=0 $.  Hence
\begin{equation}\label{09}U\cong \chi_{\bar{B_j}}\  \text{modulo}\   {\xi_i=0}  \mbox{ for some }j\in\{2,  . . . ,  s\}\end{equation}
From  $U\cong \chi_{\bar{A}} \  \text{modulo}\   {\xi_1=0} $ and \eqref{09} we deduce
$\chi_{\bar{A}}\cong \chi_{\bar{B_j}}  \  \text{modulo}\   {\xi_1=\xi_i=0}$.  So,   by
lemma \ref{lem0},   $|A|=|B_j|=1$.  Let $A=\{a\}$.  Then by lemma
\ref{supertest},   for the vertex $a$ and the index 1,   $\chi_G$
is irreducible.
\end{proof}
\begin{corollary}\label{duesoli} If there are two indices which appear only once and not in the same edge in the maximal tree then $\chi_G$ is irreducible.
\end{corollary}
We have thus treated one of the 4 cases of Lemma \ref{le4p}.

\begin{lemma}\label{angle1} If there exists a pair of indices,   say $  (1,  i)$,   such
that 1 appears only once in the maximal tree $T$ while $i$ appears twice and $T$ has the
form:
\begin{figure}[H]
\begin{center}$$\xymatrix{A\ar@{--}[r]^{i,  l}&B\ar@{--}[r]^{1,  h}&C\ar@{--}[r]^{i,  k}&D}$$\caption{}\label{pic_angle1}\end{center}
\end{figure}\noindent  then either $\chi_G$ is irreducible or $|A|=|C|=1$ or $|B|=|D|=1$.
\end{lemma}\begin{proof}
We have $\chi_G\cong\chi_{\overline {A\cup B}} \chi_{\overline {C\cup D}}$ modulo ${\xi_1=0}$ so if  $\chi_G$ is not irreducible it has a factor  $U\cong\chi_{\overline {A\cup B}}$ modulo ${\xi_1=0}$. This implies  $  U\cong\chi_{\overline {A }} \chi_{\overline { B}} $  modulo ${\xi_1=\xi_i=0}$.  Now $\chi_G\cong \chi_{\overline {A\cup D}} \chi_{\overline {B\cup C}} $ or $\chi_G\cong \chi_{\overline {A } }\chi_{\overline {  D}} \chi_{\overline {B\cup C}}$ modulo ${\xi_3=0}$ and inspecting the two factorizations the claim follows from Lemma \ref{lem0}.
\end{proof}
 \subsection{Two indices   appear only once and in the same edge}
Let these two indices be $1,  2$.  If there exists another index,
say 3,   which appears only once,   then we can replace 2 by 3 and
we are back in the case of Corollary \ref{duesoli}.  Otherwise by Lemma \ref{le4p}    we have exactly $n-1$ distinct indices
different from 1,  2 and they appear twice.   Take one of these
indices,   say 3.  If we cannot apply lemma \ref{angle} we must be
in the case,   in which the maximal tree $T$ has the form
\begin{figure}[H]
   \begin{center}
  $$ \xymatrix{ A \ar@{--}[r]^{3,  k} &B  \ar@{--}[r]^{1,  2}  & C  \ar@{--}[r]^{3,  h}  & D  }  $$ \caption{}
 \label{pic1p}
  \end{center}
\end{figure}
  \noindent where the indices $1$ and $3$ do not appear
elsewhere in the tree.
   By  inspection of figure \eqref{pic1p}
all edges in G which connect $A$ and $C$ contain $1,  3$ in the
indexing,   all edges in G which connect $B$ and $D$ contain $1,
3$ in the indexing. Then we have:

\begin{equation}\label{eq7} \chi_G\cong  \chi_{\overline{A
\cup B}}\, \chi_{\overline{C \cup D}}\ \text{modulo}\ {\xi_1=0}.
\end{equation}
\begin{equation}\label{eq8}
\chi_G\cong\chi_{\bar{A}}\, \chi_{\overline{B\cup
C}}\, \chi_{\bar{D}} \quad\text{or}\quad \chi_G\cong \chi_{\overline{A\cup
D}}\, \chi_{\overline{B\cup
C}}\ \text{modulo}\ {\xi_3=0}.
\end{equation} The second case holds when $A,  D$ are joined by some edge  which  does not contain 3.
From \eqref{eq7} we see that if $\chi_G$ is not irreducible, then
it has an irreducible factor 
$U\cong\chi_{\overline{A \cup B}}\ \text{mod.}\  {\xi_1=0} $ which implies $ U\cong\chi_{\overline{A}} \chi_{ \overline{B}}\ \text{modulo}\ {\xi_1=\xi_3=0}$.  Comparing \eqref{eq7}
and \eqref{eq8} taking into account the  degree and using the irreducibility of
$\chi_{\bar{A}},\chi_{ \overline{B}},  \chi_{\bar{D}}\ \text{modulo}\ {\xi_1=\xi_3=0}$ we get the following
possibilities\begin{equation}
\label{eqq8}U\cong \chi_{\bar{A}}\chi_{\bar{D}},\ 
\chi_{\overline {A\cup D}},\ 
\chi_{\overline {B\cup C}}\quad \text{modulo}\ {\xi_3=0}
. 
\end{equation} In   the first two cases of  \eqref{eqq8}  we have 
$$U\cong \chi_{\overline{A}} \chi_{ \overline{B}} \cong \chi_{\bar{A}}\chi_{\bar{D}}\ \text{modulo}\ {\xi_1=\xi_3=0}$$
which implies \begin{equation}\label{eq13} 
\chi_{\bar{B}}\cong \chi_{\bar{D}}\ \text{modulo}\ {\xi_1=\xi_3=0}\end{equation}
Hence by lemma \ref{lem0} we must have: $ B =\{b\},D=\{d\}$.  But the index 2 appears only once in the path from $b$ to $d$ contradicting Corollary \ref{corlem0}.

In   the   last case  of  \eqref{eqq8}  we have 
$$U\cong \chi_{\overline{A}} \chi_{ \overline{B}}\cong \chi_{\bar{B}}\chi_{\bar{C}}  \ \text{modulo}\ {\xi_1=\xi_3=0}$$
which implies \begin{equation}\label{eq14} 
\chi_{\bar{A}}\cong \chi_{\bar{C}}\ \text{modulo}\ {\xi_1=\xi_3=0}\end{equation}
We arrive at the same
conclusions.

\subsection{Only the index  1  appears once in the
tree}\label{sub2} From Lemma \ref{le4p} there is only one index,
say 3,  which appears three times. All other indices,   different
from 1,  3,  appear twice. We need to distinguish two subcases:
\subsubsection{When 1,  3 appear together in one edge}    If $T$ has the form as in figure \eqref{pic26} then,   by lemma \ref{angle},  $\chi_G$ is
irreducible.
\begin{figure}[H]
   \begin{center}
  $$\xymatrix{A\ar@{--}[r]^{1,  3}&B\ar@{--}[r]^{2,  k_1}&C\ar@{--}[r]^{2,  k_2}&D}$$ \caption{}
 \label{pic26}
  \end{center}
\end{figure}

   Therefore, assume that $T$ has the form as in figure \eqref{pic29}
\begin{figure}[H]
   \begin{center}
  $$\xymatrix{A\ar@{--}[r]^{2,  k_1}&B\ar@{--}[r]^{1,  3}&C\ar@{--}[r]^{2,  k_2}&D}$$ \caption{}
 \label{pic29}
  \end{center}
\end{figure}

We start the discussion as in the previous paragraph

\begin{equation}\label{beq7} \chi_G\cong  \chi_{\overline{A
\cup B}}\, \chi_{\overline{C \cup D}} \quad \text{modulo}\  {\xi_1=0}.
\end{equation}
\begin{equation}\label{beq8}
\chi_G\cong \chi_{\bar{A}}\, \chi_{\overline{B\cup
C}} \, \chi_{\bar{D}}  \quad\text{or}\quad \chi_G\cong \chi_{\overline{A\cup
D}}\, \chi_{\overline{B\cup
C}}\quad \text{modulo}\ {\xi_2=0}.
\end{equation} The second case holds when $A,  D$ are joined by some edge  which  does not contain 2.
From \eqref{beq7} we see that if $\chi_G$ is not irreducible, then
it must factor into two irreducible polynomials: $\chi_G=UV$,
$U\cong \chi_{\overline{A \cup B}}$ modulo $\xi_1=0$ implies $ U\cong \chi_{\overline{A}} \chi_{ \overline{B}} $  modulo $\xi_1=\xi_2=0$.  Comparing \eqref{beq7}
and \eqref{beq8} taking into account the  degree and using the irreducibility of
$\chi_{\bar{A}},\chi_{ \overline{B}},  \chi_{\bar{D}} $  modulo $\xi_1=\xi_2=0$ we get the following
possibilities\begin{equation}
\label{beqq8}U\cong \chi_{\bar{A}}\chi_{\bar{D}},\ 
\chi_{\overline {A\cup D}},\ 
\chi_{\overline {B\cup C}}   \quad\text{modulo}\ \xi_2=0 
. 
\end{equation} In   the first two cases of  \eqref{beqq8}  we have 
$$U\cong\chi_{\overline{A}} \chi_{ \overline{B}}\cong \chi_{\bar{A}}\chi_{\bar{D}}\quad \text{modulo}\ {\xi_1=\xi_2=0}$$
which implies \begin{equation}\label{beq13} 
\chi_{\bar{B}}\cong\chi_{\bar{D}}\quad \text{modulo}\ {\xi_1=\xi_2=0}\end{equation}
 
In   the   last case  of  \eqref{beqq8}  we have 
$$U\cong\chi_{\overline{A}} \chi_{ \overline{B}}\cong \chi_{\bar{B}}\chi_{\bar{C}}\quad \text{modulo}\ {\xi_1=\xi_2=0}$$
which implies \begin{equation}\label{beq14} 
\chi_{\bar{A}}\cong\chi_{\bar{C}}\quad \text{modulo}\ {\xi_1=\xi_2=0}\end{equation}                        
By symmetry we need to consider only case \eqref{beq14}. By lemma \ref{lem0} we get   $|A|=|C|=1,  A=\{0\},   C=\{c\},
c=\tau_{n_1e_1+n_2e_2}  (0)$.  By inspection of Figure
\eqref{pic29} $n_1,  n_2\in\{\pm 1\}$.
\begin{equation}\label{14}\eta(c)\in\{0,-2\}\implies
c=\pm   (e_1-e_2),-e_1-e_2  \end{equation}We have thus proved: \begin{lemma}\label{dari}
Either $|A|=|C|=1$ and there is an
edge marked $  (1,  2)$ that connects $A=0$ and $c=C$. Or the same statement for $B,D$. 
Moreover,  all indices,   different from 1,  2 must appear an even
number of times in every path from 0 to $c$ (resp. $b,d$).
\end{lemma}  Assume $A=0, C=c$, consider the index
$k_1$.

i) If $k_1\neq 3$,   then $k_1$ must appear once more in the block
$B$ like:$$\xymatrix{0\ar@{--}[d]_{2, k_1}\ar@{--}[rrd]^{2,  1}\\ B_1\ar@{--}[r]^{k_1,
s}&B_2\ar@{--}[r]_{1, 3}&c\ar@{--}[r]^{2,  k_2}&D}$$
\noindent Now
we can apply
 \ref{angle} to the pair $  (1,  k_1)$ and get the irreducibility of
$\chi_G$.

ii) So we can assume that $k_1=3$,  consider the index $k_2$.
$$\xymatrix{0\ar@{--}[d]_{2, 3}\ar@{--}[rd]^{2,  1}\\  B \ar@{--}[r]_{1, 3}&c\ar@{--}[r]^{2,  k_2}&D}$$ 

A) If $k_2\neq 3$, then either $k_2$ appears in the block $D$ as
in   figure \eqref{pic45}, and then  by lemma \ref{angle} for
the pair $  (1,  k_2)$,  $\chi_G$ is irreducible;
  or it appears in the block $B$ as in
figure \eqref{pic44}.

\begin{figure}[H]\begin{minipage}[t]{2.6in}\begin{center}$$ \xymatrix{ 0\ar@{--}[d]_{2, 3}\ar@{--}[rd]^{2,  1}\\  B_1 \ar@{--}[r]_{1, 3}&c\ar@{--}[r]^{2,  k_2}&D\\B_2 \ar@{--}[u]_{k_2, s} }$$\caption{}\label{pic44}\end{center}\end{minipage}\hfill\begin{minipage}[t]{2.6in}\begin{center}$$\xymatrix{ 0\ar@{--}[d]_{2, 3}\ar@{--}[rd]^{2,  1}&&&\\  B_2 \ar@{--}[r]_{1, 3}&c\ar@{--}[r]^{2,  k_2}&D_1 \ar@{--}[r]_{k_2, s} &D_2}$$ \caption{}\label{pic45}\end{center}\end{minipage}\end{figure}

 %
%
%
%\begin{figure}[H]\begin{center}$$ \xymatrix{ 0\ar@{--}[d]_{2, 3}\ar@{--}[rd]^{2,  1}\\  B_1 \ar@{--}[r]_{1, 3}&c\ar@{--}[r]^{2,  k_2}&D\\B_2 \ar@{--}[u]_{k_2, s} }$$\caption{}\label{pic44}\end{center}\end{figure}
%\begin{figure}[H]\begin{center}$$\xymatrix{ 0\ar@{--}[d]_{2, 3}\ar@{--}[rd]^{2,  1}&&&\\  B_2 \ar@{--}[r]_{1, 3}&c\ar@{--}[r]^{2,  k_2}&D_1 \ar@{--}[r]_{k_2, s} &D_2}$$\caption{}\label{pic45}\end{center}\end{figure}
%\begin{figure}[H]\begin{center}{\includegraphics[scale=0.7]{pictures/pic45.JPG}}\caption{}\label{pic45}\end{center}\end{figure}
%\begin{figure}[H]\begin{center}{\includegraphics[scale=0.7]{pictures/pic44.JPG}}\caption{}\label{pic44}\end{center}\end{figure}
%$$\begin{matrix}
%\xymatrix{ 0\ar@{--}[d]_{2, 3}\ar@{--}[rd]^{2,  1}\\  B_1 \ar@{--}[r]_{1, 3}&c\ar@{--}[r]^{2,  k_2}&D\\B_2 \ar@{--}[u]_{k_2, s} }&\xymatrix{ 0\ar@{--}[d]_{2, 3}\ar@{--}[rd]^{2,  1}&&&\\  B_2 \ar@{--}[r]_{1, 3}&c\ar@{--}[r]^{2,  k_2}&D_1 \ar@{--}[r]_{k_2, s} &D_2}\\
%\text{Figure}\ \caption{}\label{pic45}&\text{Figure}\ \caption{}\label{pic45}
%\end{matrix} $$ 
In the case of   figure
\eqref{pic44} we can apply Lemma \ref{dari}  for $1,k_2$. Since $|0\cup B_1|>1$  the only possibility is  that $B_2=b_2$ and  there exists an edge with the
marking $ (1,  k_2)$ that connects $c$ and $b_2$.

Now we claim that we must have $s=3$  in fact $s$  must appear an even number of times  in both paths from $0,c$ and from $b_2,c$, this is possible only for $s=3$.
$$ \xymatrix{ 0\ar@{--}[d]_{2, 3}\ar@{--}[rd]^{2,  1}\\  B_1 \ar@{--}[r]_{1, 3}&c\ar@{--}[r]^{2,  k_2}&D\\b_2 \ar@{--}[u]^{k_2, 3}\ar@{--}[ur]_{k_2, 1} }$$\noindent We now remove the two edges marked $1,3$ and $k_2,3$. In the resulting maximal tree 3 appears once and we can apply Lemma  \ref{angle} to the pair $  (3,  k_2)$, $\chi_G$ is
irreducible.  

% i) If  $B_1=\{b_1\}$,   then, by lemma \ref{supertest} for
%the vertex $b_1$ and the index 3,  $\chi_G$ is irreducible.
%
%\noindent ii) If   $|B_1|>1$,   let $i$ be an index that appears in
%the block $B_1$.  If $i$ appears twice in the block $B_1$  then by
%Lemma \ref{angle} for the pair $  (1,  i)$, $\chi_G$ is
%irreducible.  Hence,  since $i$ appears only twice, we need to
%consider the case,   when $i$ appears once in the block $B_1$ and
%once in the block $D$ as in figure \eqref{pic46}.  \begin{figure}[H]\begin{center}$$\xymatrix{b_2\ar@{--}[d]_{k_2, 3}&0\ar@{--}[dl]^{2, 3}\ar@{--}[d]^{2,  1}\\B'\ar@{--}[r]^{1, 3}\ar@{--}[d]^{i, h}& c\ar@{--}[r]^{2,  k_2}&D_1\ar@{--}[r]^{i,
%m}&D_2\\ B "&}$$\caption{}\label{pic46}\end{center}\end{figure}
%
%This case is easily excluded by Lemma \ref{dari} for the pair $1,i$. 

B) If $k_2=3$ and $|B|>1$.  Let $i$ be an index that appears in
$B$. If $i$ appears twice in $B$,   then, by lemma \ref{angle} we
get the irreducibility of $\chi_G$.  Otherwise,   $i$ appears in
this form:
\begin{figure}[H]\begin{center}\includegraphics[scale=0.7]{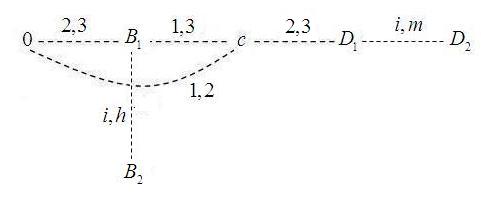}\caption{}\label{pic47}\end{center}\end{figure}
This case is   excluded by Lemma \ref{dari}  for the pair $1,i$. 
  The case $|D|>1$ is treated similarly.  So
now we have to consider only the case,   when $|B|=|D|=1$.

C) $k_2=3,   |B|=|D|=1$.  Up to symmetry,   we have 4 subcases,
displayed in figures \eqref{pic48}-\eqref{pic51}.
\begin{figure}[H]\begin{minipage}[t]{2.6in}\begin{center}\includegraphics[scale=0.7]{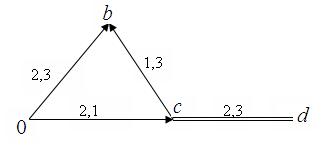}\caption{}\label{pic48}\end{center}\end{minipage}\hfill\begin{minipage}[t]{2.6in}\begin{center}\includegraphics[scale=0.7]{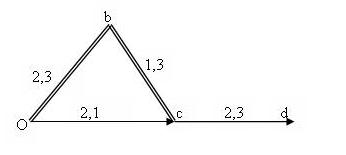}\caption{}\label{pic49}\end{center}\end{minipage}\end{figure}
\begin{figure}[H]\begin{minipage}[t]{2.6in}\begin{center}\includegraphics[scale=0.7]{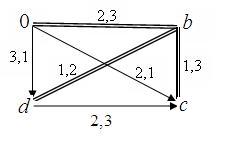}\caption{}\label{pic50}\end{center}\end{minipage}\hfill\begin{minipage}[t]{2.6in}\begin{center}\includegraphics[scale=0.7]{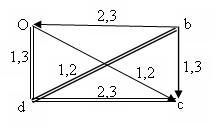}\caption{}\label{pic51}\end{center}\end{minipage}\end{figure}
By using the program Mathematica we have verified that the
characteristic polynomials of these graphs are irreducible.
\subsubsection{When 1,  3 do not appear together in any edge:} We have three possible
cases  (given in figures \eqref{pic30},  \eqref{pic31}, \eqref{pic54a}).

1)\quad When $T$ up to   symmetry has the form as in   figure
\eqref{pic30}:
\begin{figure}[H]
   \begin{center}
$$\xymatrix{A\ar@{--}[r]^{1,  2}&B}$$\caption{}
 \label{pic30}
  \end{center}
\end{figure}\noindent where $3$ appears only in the block $B$
 then, by lemma \ref{angle},  for the pair $  (1,  3)$,  $\chi_G$ is irreducible.
\newline 2)\quad When $T$ up to  symmetry has the form as in  figure 
\eqref{pic31}:
\begin{figure}[H]
   \begin{center}
  $$\xymatrix{A\ar@{--}[r]^{3,  k_1}&B\ar@{--}[r]^{1,  2}&C\ar@{--}[r]^{3,  k_2}&D\ar@{--}[r]^{3,  k_3}&E}$$ \caption{}
 \label{pic31}
  \end{center}
\end{figure}
We have
\begin{equation}\label{106}\text{modulo}\quad \xi_1=0,\quad 
\chi_G\cong\chi_{\overline{A\cup B}}\chi_{\overline{C\cup
D\cup E}}  \end{equation}
\begin{equation}\label{107}\text{modulo}\quad \xi_3=0,\quad\chi_G=\begin{cases}
\chi_{\bar{A}}\chi_{\overline{B\cup
C}}\chi_{\bar{D}}\chi_{\bar{E}}\\
\chi_{\overline{A\cup D}}\chi_{\overline{B\cup
C}}\chi_{\bar{E}}\\
\chi_{\overline{A\cup D}}\chi_{\overline{B\cup C\cup
E}}\\\chi_{\overline{A}}\chi_{\overline{D}}\chi_{\overline{B\cup C\cup
E}}\end{cases}
\end{equation}  Arguing as in previous cases,  if $\chi_G$ factors then we can factor it as $UV$ with $U_{\xi_1=0}=\chi_{\overline{A\cup B}}$.  Analyzing the possible values of $U_{\xi_3=0}$ we have,
comparing \eqref{106} and \eqref{107} and setting $\xi_1=\xi_3=0$,
the following possibilities:
\begin{equation}\begin{matrix}\\
U\\\text{mod} \ \xi_3=0
\end{matrix}
 \cong\begin{cases}\chi_{\overline{B\cup
C}}\qquad \qquad\quad\,\implies \,
\chi_{\bar{A}}\cong\chi_{\bar{C}}\ \text{mod.}\ {\xi_1=\xi_3=0}
 \\\chi_{\overline{A\cup
D}} \ \text{or}\ \chi_{\bar{A}}\chi_{\bar{D}}\quad  \implies
\chi_{\bar{B}}\cong \chi_{\bar{D}}\ \text{mod.}\ {\xi_1=\xi_3=0}
 \\ \chi_{\bar{A}}\chi_{\bar{E}}\quad\qquad\qquad\implies
\chi_{\bar{B}}\cong \chi_{\bar{E}}\ \ \text{mod.}\ {\xi_1=\xi_3=0} 
\\ \chi_{\bar{D}} \chi_{\bar{E}} \implies
\begin{cases}
\chi_{\bar{A}}\cong \chi_{\bar{D}} ,\chi_{\bar{B}} \cong\chi_{\bar{E}}\ \text{mod.}\ {\xi_1=\xi_3=0}
 \\\chi_{\bar{A}}\cong\chi_{\bar{E}}  ,\chi_{\bar{B}}\cong\chi_{\bar{D}}\ \text{mod.}\ {\xi_1=\xi_3=0}
\end{cases}\end{cases}\label{109}\end{equation} 
It is enough to exclude the first 3 cases of \eqref{109}.  

  {\bf Case 1}\quad  If
$\chi_{\bar{A}}=\chi_{\bar{C}}$ modulo ${\xi_1=\xi_3=0}$, by lemma \ref{lem0} and by inspection we deduce
that $A=\{0\}, C=\{c\}$
 and $c=\pm(e_1-e_3),-e_1-e_3$.  Hence there is an edge marked $1,3$  that connects 0 and $c$. We can then replace the maximal tree $T$ with the one in which we keep this edge and remove the one marked $1,2$ and we find ourselves in the case treated in the previous paragraph.

{\bf Case 2}\quad  If $\chi_{\bar{B}}\cong \chi_{\bar{D}}$ modulo ${\xi_1=\xi_3=0}$,   then, by lemma \ref{lem0}  $B=\{b\},
D=\{d\}$ and 2 should appear an even number of times  between them, again a contradiction (we are in the case $k_2=2$).

{\bf Case 3}\quad  If $\chi_{\bar{B}}\cong\chi_{\bar{E}}$ modulo ${\xi_1=\xi_3=0}$, then, by lemma \ref{lem0} and choosing the root at $B$ we have $B=\{0\},
E=\{e\}$ we have the same contradiction as in the previous case.\medskip

3)\quad  When $T$ has the
form:\begin{figure}[H]\begin{center}$$\xymatrix{&A\ar@{--}[r]^{3,
k_1}&B\ar@{--}[r]^{1,  2}&C\ar@{--}[r]^{3,  k_3}\ar@{--}[d]^{3,
k_2} &E\\&&&D&&}$$\caption{}\label{pic54a}\end{center}\end{figure}
\begin{equation}\label{19}\chi_G\cong \chi_{\overline{A\cup
B}}\chi_{\overline{C\cup D\cup
E}}\quad \text{modulo}\ {\xi_1=0}\end{equation}  From \eqref{19} we see that if $\chi_G$ is not irreducible,
then $\chi_G=UV$,  where $U,  V$ are irreducible,
$U\cong \chi_{\overline{A\cup B}}\  \text{modulo}\ {\xi_1=0}, \implies U\cong\chi_{\overline{A }}\chi_{\overline{  B}}\ \text{modulo}\  {\xi_1=\xi_3=0}$.\begin{equation}\label{20}\text{modulo}\ \xi_3=0,\ \chi_{G}\cong\begin{cases}\chi_{\bar{A}}\chi_{\overline{B\cup
C}}\chi_{\bar{D}}\chi_{\bar{E}}\\\chi_{\overline{A\cup
D}}\chi_{\overline{B\cup C}}\chi_{\bar{E}}\\
\chi_{\overline{A\cup E}}\chi_{\overline{B\cup
C}}\chi_{\bar{D}}
\\\chi_{\bar{A}}\chi_{\overline{B\cup
C}}\chi_{\overline{D\cup E}}
\\\chi_{\overline{A\cup D\cup
E}}\chi_{\overline{B\cup C}}\end{cases}\end{equation}
 As for  $U $  it may be congruent modulo $  \xi_3=0$ to
 $$\chi_{\overline{B\cup C}},\,\chi_{\overline{A\cup D }},\,\chi_{\overline{A\cup E }},\,$$$$\chi_{\overline{A  }} \chi_{\overline{ D }},\,\chi_{\overline{A  }} \chi_{\overline{ E }},\,\chi_{\overline{D  }} \chi_{\overline{ E }},\,\chi_{\overline{A\cup D\cup
E}} $$  giving the following subcases:
1) $\chi_{\bar{C}}\cong\chi_{\bar{A}}$, 2)  $\chi_{\bar{B}}\cong\chi_{\bar{D}}$, 3) $\chi_{\bar{B}}\cong\chi_{\bar{E}}$,  4) $\chi_{\bar{B}}\cong\chi_{\overline{D\cup E}}$ modulo ${\xi_1=\xi_3=0}. $  The fourth case can be excluded by cardinality. We treat the other 3 cases.\medskip

1)\quad $\chi_{\bar{C}}|_{\xi_1=\xi_3=0}=\chi_{\bar{A}}$,  by Lemma
\ref{lem0},  $ A=\{0\},   C=\{c\},$ and $c=\pm(e_1-e_3),-e_1-e_3$.  Hence there is an edge marked $1,3$  that connects 0 and $c$. We can then replace the maximal tree $T$ with the one in which we keep this edge and remove the one marked $1,2$ and we find ourselves in the case treated in the previous paragraph.

2) $ \chi_{\bar{B}}\cong\chi_{\bar{D}}$ modulo ${\xi_1=\xi_3=0}$ by
lemma \ref{lem0} $\implies |B|=|D|=1,  B=\{b\},   D=\{d\}$  
and $\sigma_dd+\sigma_bb=\pm(e_1-e_3),-e_1-e_3$.  Hence there is an edge marked $1,3$  that connects $b$ and $d$. We can then replace the maximal tree $T$ with the one in which we keep this edge and remove the one marked $1,2$ and we find ourselves in the case treated in the previous paragraph.

3) $ \chi_{\bar{B}}\cong\chi_{\bar{E}}$ modulo ${\xi_1=\xi_3=0}$  is  similar to case 2),
changing the role of $k_2$ and $k_3$.

\subsection{Every index appears twice in the tree}
\begin{lemma}\label{dumt}
If $\chi_G$ is not irreducible the graph is a tree.
\end{lemma}
\begin{proof}
Assume there is a an edge marked $i,j$  in the graph and not in the tree, then a segment in the tree together with this edge form a dependent circuit, thus  we can remove an edge marked $a,b$ in this segment and add the edge  $i,j$ in order to obtain another maximal tree. Clearly in a circuit there is at least an edge such that the indices $i,j$ are distinct fro the indices $a,b$. This means that in the new maximal tree  one of the indices $i,j$ appears with multiplicity 1 and we are back to a previous case.
\end{proof} From now on we thus assume that the graph is a tree $T$.
We start with
some special cases:\subsubsection{$n=2$}
 $$ T:\quad \xymatrix{ &-e_1-e_2 \ar@{=}[r]  &0\ar@{->}[r]&e_1-e_2 }$$\smallskip
%$$ C_T=\begin{pmatrix}
%-\xi_1-\xi_2&2\sqrt{\xi_1\xi_2}&0\\ \\
%-2\sqrt{\xi_1\xi_2}&0&2\sqrt{\xi_1\xi_2}\\\\0&2\sqrt{\xi_1\xi_2}&\xi_2-\xi_1
%\end{pmatrix}$$ 
is not allowable (but its characteristic polynomial is irreducible).
%  determinant
%$$  (-\xi_1-\xi_2)  (-4\xi_1\xi_2)+4\xi_1\xi_2  (\xi_2-\xi_1)=8\xi_1\xi_2^2$$
%$$ \chi_T  (t)=\det  (tI-C_T)=\det\left  ( 
%\begin{array}{ccc}
%  t+\xi_1+\xi_2 & -2\sqrt{\xi_1\xi_2} & 0 \\
%  2\sqrt{\xi_1\xi_2} & t & -2\sqrt{\xi_1\xi_2} \\
%  0 & -2\sqrt{\xi_1\xi_2} & t-\xi_2+\xi_1 \\
%\end{array} 
%\right)$$ if it is not irreducible it is divisible by a linear
%form,   set $\xi_1=0$  get $t  (t+\xi_2)  (t-\xi_2)$  set
%$\xi_2=0$ get $t  (t+\xi_1)^2$ so the possible linear factors can
%be $$ t,  t+\xi_1,  t\pm \xi_2.$$ On the other hand we easily verify that \begin{equation}\label{23} \chi_T  (t)=t^3+2\xi_1t^2+  (\xi_1^2-\xi_2^2)t-8\xi_1\xi_2^2   \end{equation}
%is not divisible by any of these linear factors.
\subsubsection{$n=3$} Up to symmetry of the indices $T$ has the form as in   figure \eqref{pic55} or as in   figure \eqref{pic56}:\begin{figure}[H]\begin{center}$$\xymatrix{0\ar@{--}[r]^{1,  2}&b\ar@{--}[r]^{2,  3}&c\ar@{--}[r]^{1,  3}&d}$$\caption{}\label{pic55}\end{center}\end{figure}
\begin{figure}[H]\begin{center}$$\xymatrix{0\ar@{--}[r]^{1,  2}&b\ar@{--}[r]^{2,  3}\ar@{--}[d]^{1,  3}&c\\&d&&}$$\caption{}\label{pic56}\end{center}\end{figure}
\begin{remark}\label{rem2}  If all edges in $T$ are
black,  or there are exactly two red edges then    the edges are
linearly dependent.   \end{remark}1) When the graph $T$ has
the form as in figure \eqref{pic55}
%\begin{remark}\label{rem1}\begin{itemize}\item If in $\bar{T}$ there is an edge marked $  (1,  3)$ that connects $0$ and $c$,   then, by lemma \ref{supertest} for the vertex $b$ and the index 2,  $\chi_T$ is irreducible. \item If in $\bar{T}$ there is an edge marked $  (1,  2)$ that connects $b$ and $d$,   then, by lemma \ref{supertest} for the vertex $c$ and the index 3,  $\chi_T$ is irreducible. \end{itemize}\end{remark}
\noindent a)\quad If all edges are red,   then $G=\bar{T}$ is not a tree:
\begin{figure}[H]\begin{center}\includegraphics[scale=0.8]{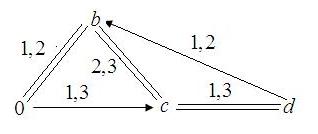}\caption{}\label{pic54}\end{center}\end{figure}
  We need to consider the cases,
when in $T$ there is one red   and two black edges. Up to symmetry we may assume the red edge is the first or the second.\medskip

\noindent b)\quad When the red
edge connects $0$ and $b$:\smallskip

 b1) When $T$ has the form:
$$\xymatrix{0\ar@{=}[r]^{1,  2}&b\ar@{->}[r]^{2,  3}&c\ar@{-}[r]^{1,  3}&d}$$ We have $$b=-e_1-e_2,  c-b=e_2-e_3\implies c=-e_1-e_2. $$ Hence  $G=\bar{T}$ is not a tree.
\smallskip

 b2)\quad  If $T$ has the form:
$$\xymatrix{0\ar@{=}[r]^{1,  2}&b\ar@{<-}[r]^{2,  3}&c\ar@{->}[r]^{1,  3}&d}$$
We have $b-c=e_1-e_3,  d-c=e_1-e_3\implies d-b=e_1-e_2$,   i. e.
in $G$ there is a black edge marked $  (1,  2)$ that
connects $b$ and $d$. Hence  $G=\bar{T}$ is not a tree. \smallskip

 b3) If $T$ has the
form:$$\xymatrix{0\ar@{=}[r]^{1, 2}&b\ar@{<-}[r]^{2,
3}&c\ar@{<-}[r]^{1,  3}&d}$$
$$\label{24}\chi_T\!=\det\!\left  (\!\!\!\!
\begin{array}{cccc}
  t & 2\sqrt{\xi_1\xi_2} & 0 & 0 \\
  -2\sqrt{\xi_1\xi_2} & t+\xi_1+\xi_2 & 2\sqrt{\xi_2\xi_3} & 0 \\
  0 & 2\sqrt{\xi_2\xi_3} & t+\xi_1+2\xi_2-\xi_3 & 2\sqrt{\xi_1\xi_3} \\
  0 & 0 & 2\sqrt{\xi_1\xi_3} & t+2\xi_1+2\xi_2-2\xi_3 \\
\end{array}%
\!\right)$$By using the program Mathematica we computed
$\chi_T$ and verified that it is irreducible.  \medskip

\noindent c) When
the red edge connects $b$ and $c$:\smallskip

 c1)\quad If $T$ has the
form:
$$\xymatrix{0\ar@{-}[r]^{1,  2}&b\ar@{=}[r]^{2,  3}&c\ar@{<-}[r]^{1,  3}&d}$$
we have $b+c=-e_2-e_3,  c-d=e_1-e_3\implies b+d=-e_1-e_2$,   i. e.
there is a red edge marked $  (1,  2)$ that connects $b$
and $d$. Hence  $G=\bar{T}$ is not a tree.
\smallskip

 c2)\quad If $T$ has the form $$\xymatrix{0\ar@{->}[r]^{1,
2}&b\ar@{=}[r]^{2, 3}&c\ar@{-}[r]^{1,  3}&d}$$ we have $b=e_1-e_2,
b+c=-e_2-e_3\implies c=e_1-e_3$,   i. e.  there is a black edge
marked $  (1,  3)$ that connects $0$ and $c$. Hence  $G=\bar{T}$ is not a tree. \smallskip

 c3)\quad If $T$ has
the form:$$\xymatrix{0\ar@{<-}[r]^{1,  2}&b\ar@{=}[r]^{2,
3}&c\ar@{->}[r]^{1,  3}&d}$$
we have $$\chi_T=\det\left  (%
\begin{array}{cccc}
  t & -2\sqrt{\xi_1\xi_2} & 0 & 0 \\
  -2\sqrt{\xi_1\xi_2} & t-\xi_1+\xi_2 & 2\sqrt{\xi_2\xi_3} & 0 \\
  0 & -2\sqrt{\xi_2\xi_3} & t-\xi_1+2\xi_2+\xi_3 & 2\sqrt{\xi_1\xi_3} \\
  0 & 0 & 2\sqrt{\xi_1\xi_3} & t-2\xi_1+2\xi_2+2\xi_3 \\
\end{array}%
\right)$$We used the program Mathematica to compute
$\chi_T$ and to verify that it is irreducible.
\medskip

2) When $T$ has the form as in figure \eqref{pic56}:\smallskip

\noindent a)\quad
When in $T$ there are 3 red edges, then $G=\bar{T}$ has the
form:\begin{figure}[H]\begin{center}\includegraphics{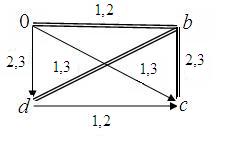}\caption{}\label{pic57}\end{center}\end{figure}
This figure can be obtained from figure \eqref{pic50} by
exchanging the role of indices  (i. e.  the role of variables
$\xi_1,  \xi_2,  \xi_3$). Hence $\chi_T$ is irreducible.
\smallskip

\noindent  b)\quad  When in $T$ there is only one red edge,   by the symmetry
property of $T$ we may suppose that this red edge connects $0$ and
$b$. \smallskip

 b1)\quad If $T$ has the form:$$\xymatrix{0\ar@{=}[r]^{1,
2}&b\ar@{->}[r]^{2, 3}\ar@{-}[d]^{1,  3}&c\\&d&&}$$  in
$G$ there is a red edge marked $  (1,  3)$ that connects
0 and $c$.  Hence $G=\bar T$ is not a tree. \smallskip

 b2)\quad If $T$ has the
form:$$\xymatrix{0\ar@{=}[r]^{1,  2}&b\ar@{-}[r]^{2,
3}\ar@{->}[d]^{1,  3}&c\\&d&&}$$ we have $b=-e_1-e_2,
d-b=e_1-e_3\implies d=-e_2-e_3$,   hence in $G$ there is a red
edge marked $  (2,  3)$ that connects 0 and $d$.  Hence $G=\bar T$ is not a tree. \smallskip

 b3)\quad  If $T$ has
the form:
$$\xymatrix{0\ar@{=}[r]^{1,  2}&b\ar@{<-}[r]^{2,  3}\ar@{<-}[d]^{1,  3}&c\\&d&&}$$
we have $b-c=e_2-e_3,  b-d=e_1-e_3\implies d-c=e_2-e_1$,   hence
there is a black edge marked $  (2,  1)$ that connects
$c$ and $d$.  Hence $G=\bar T$ is not a tree.

\subsection{\ $n\geq 4$}  At this point we are assuming that we
have $n\geq 4$ edges in a maximal tree $T$ and $n$ indices,  each
appearing twice. Thus given an index,  say 1,  it appears in two
edges paired with at most two other indices,  thus we can find
another index,  say 2 which is not in these two edges.    Up to
symmetry we may have  six cases displayed in figures  \eqref{pic37}-\ -\eqref{pic32}:
\begin{figure}[H]\begin{center}$$ \xymatrix{&&D&&\\A \ar@{--}[r]^{1,  h}  &B\ar@{--}[r]^{2,  k}  &C\ar@{--}[u]^{1,  i}  \ar@{--}[r]^{2,  j}  & E \\&&&&}$$\caption{}\label{pic37}\end{center}\end{figure}
 \begin{figure}[H]\begin{center}$$ \xymatrix{A \ar@{--}[r]^{1,  h}  &B\ar@{--}[r]^{2,  k}  &C\ar@{--}[r]^{2,  i}  &D\ar@{--}[r]^{1,  j}  & E
 }$$\caption{}\label{pic34}\end{center}\end{figure}
 \begin{figure}[H]\begin{center}$$ \xymatrix{&&D&&\\A \ar@{--}[r]^{1,  h}  &B\ar@{--}[r]^{1,  k}  &C\ar@{--}[u]^{2,  i}  \ar@{--}[r]^{2,  j}  & E
 \\&&&&}$$\caption{}\label{pic35}\end{center}\end{figure}
 \begin{figure}[H]\begin{center}$$ \xymatrix{&C&&\\A \ar@{--}[r]^{1,  h}  &B\ar@{--}[r]^{1,  k}
\ar@{--}[u]^{2,  i}\ar@{--}[d]^{2,  j}  & E
\\&D&&}$$\caption{}\label{pic36}\end{center}\end{figure}
 \begin{figure}[H]\begin{center}$$ \xymatrix{A \ar@{--}[r]^{1,  h}  &B\ar@{--}[r]^{2,  k}  &C\ar@{--}[r]^{1,  i}  &D\ar@{--}[r]^{2,  j}  & E
 }$$\caption{}\label{pic33}\end{center}\end{figure}
 \begin{figure}[H]\begin{center} $$\xymatrix{A \ar@{--}[r]^{1,  h}  &B\ar@{--}[r]^{1,  k}  &C\ar@{--}[r]^{2,  i}  &D\ar@{--}[r]^{2,  j}  & E
 }$$\caption{}\label{pic32}\end{center}\end{figure}

 When we put $\xi_1=0$  or $\xi_2=0$  we   have   3  connected components in the graph,   so by induction we deduce that,   if the characteristic polynomial is not irreducible it can factor in at most 3  factors.  We will perform a case analysis in order to produce  two pairs of  disjoint blocks  which give under specialization $\xi_1=\xi_2=0$ the same characteristic
polynomials and we   apply Lemma \ref{lem0}.   In this way we will prove the irreducibility of $\chi_T$ in each case,
displayed in figures
 \eqref{pic37}-\eqref{pic32}.
 \subsubsection {Figure \eqref{pic37}}$$ \xymatrix{&&D&&\\A \ar@{--}[r]^{1,  h}  &B\ar@{--}[r]^{2,  k}  &C\ar@{--}[u]^{1,  i}  \ar@{--}[r]^{2,  j}  & E \\&&&&}$$
  We have \begin{equation}\label{33}\chi_T \cong  \chi_{{A}}\chi_{{B\cup C\cup E}} \chi_{{D}}\ \text{mod.}\ {\xi_1=0}, \quad \chi_T \cong \chi_{{A\cup
B}}\chi_{{C\cup D}} \chi_{{E}}\ \text{mod.}\ {\xi_2=0},
\end{equation}   Suppose that $\chi_T$ is not irreducible,   then there is an irreducible factor $U$ congruent to either $\chi_{{A}}$ or $\chi_{{D}} $ or finally $\chi_{{A}} \chi_{{D}}$ modulo ${\xi_1=0}$.

Then   $U $ is  congruent to $\chi_{{E}}$ or  $\chi_{{A\cup
B}}$ or $\chi_{{C\cup D}}$ modulo ${\xi_2=0}$.

We now specialize $\xi_1=\xi_2=0$ and apply Lemma \ref{lem0} and we have several possibilities of two blocks giving the same characteristic polynomial. Of these possibilities some are excluded by the parity condition of  the indices 1,2 in the path joining them. 

We then see that we are left  with the ones listed which all produce an extra edge contradicting the assumption that $G=\bar T$ is a tree.   
$$ \xymatrix{&&D&&\\  &B\ar@{--}[r]^{2,  k}  &c\ar@{--}[u]^{1,  i}  \ar@{--}[r]^{2,  j}  & E \\&&0 \ar@{--}[lu]^{1,  h} \ar@{--}[u]_{1,  2}&&&}  \xymatrix{&&d&&\\A \ar@{--}[r]^{1,  h}  &b\ar@{--}[r]_{2,  k} \ar@{--}[ru]^{1,  2} &C\ar@{--}[u]_{1,  i}  \ar@{--}[r]_{2,  j}  & E \\&&&&}$$
$$ \xymatrix{&&d&&\\A \ar@{--}[r]^{1,  h}  &B\ar@{--}[r]^{2,  k}  &C\ar@{--}[u]^{1,  i}  \ar@{--}[r]^{2,  j}  & e\ar@{--}[lu]_{1,  2} \\&&&&}.$$ 
\qed  
%We often use these two Lemmas as follows:
% \begin{corollary}\label{clem5}
%If there exist two indices 1, 2,
%such that $T$ is of the form as in figure \eqref{pic32} and the two pairs   $B,D$  and $A,E$ are NOT in special position, or  $T$ is of the form as in figure \eqref{pic37} and the two pairs   $B,D$  and $A,C$, $D,E$ are NOT in special position  then $\chi_T$ is irreducible.
%\end{corollary}

\subsubsection{Figure
\eqref{pic34}}$$ \xymatrix{A \ar@{--}[r]^{1,  h}  &B\ar@{--}[r]^{2,  k}  &C\ar@{--}[r]^{2,  i}  &D\ar@{--}[r]^{1,  j}  & E
 }$$ 
 \begin{equation}\label{112}
\chi_T \cong  \chi_{{A}}\chi_{{B\cup
C\cup D}} \chi_{{E}}\ \text{mod.}\ {\xi_1=0}, \quad \chi_T \cong \chi_{{A\cup
B}}\chi_{{C}} \chi_{{D\cup
E}}\ \text{mod.}\ \xi_2=0 
\end{equation}  Suppose that $\chi_T$ is not irreducible,   then there is an irreducible factor $U$ such that $U $ is congruent, modulo $\xi_1=0$ to $\chi_{{A}}$ or $\chi_{{E}} $ or finally $\chi_{{A}} \chi_{{E}}.$ 

Then   $U $ is congruent, modulo $\xi_2=0$ to  either $\chi_{{C}}$ or  $\chi_{{A\cup
B}}$ or $\chi_{{D\cup E}}$. We reason as in previous cases, specializing $\xi_1=\xi_2=0$  we deduce that there are four possible applications of Lemma \ref{lem0} for the blocks $A,E$ and the blocks $C,B,D$. We exclude those for which an index 1,2 in the path connecting them occurs only once and the other 0 or 2. We then are left with the cases:
 \begin{equation}\label{114}
\chi_{{A}}\cong \chi_{{C}} 
,\ 
\chi_{{C}} \cong \chi_{{E}},\ \text{mod}\ {\xi_1=\xi_2=0}
\end{equation}

By  symmetry we need to consider only the first. \smallskip

Assume thus that $\chi_{{A}}\cong \chi_{{C}}$ modulo ${\xi_1=\xi_2=0}
$,  by lemma
\ref{lem0} we have $|C|=|A|=1,  C=\{c\},  A=\{0\}, c=\tau_{\pm
e_1\pm e_2} (0)$ ,    $  c=\pm
(e_1-e_2),-e_1-e_2$. Hence there is an edge marked $
(1, 2)$ connecting 0 and $c$. 
$$ \xymatrix{ &&&0 \ar@{--}[ld]_{1,  h}  \ar@{--}[d]^{1,2}\\&
&B \ar@{--}[r]^{2,  k}  &c\ar@{--}[r]^{2,
i} &D\ar@{--}[r]^{1, j}  & E
 } $$ and $G=\bar T$ is not a tree.
\qed         
\subsubsection{Figure \eqref{pic36}}   $$ \xymatrix{&C&&\\A \ar@{--}[r]^{1,  h}  &B\ar@{--}[r]^{1,  k}
\ar@{--}[u]^{2,  i}\ar@{--}[d]^{2,  j}  & E
\\&D&&}$$
We have:
\begin{gather}
\chi_T\cong\chi_{{A}}\chi_{{C\cup B\cup
D}} \chi_{{E}}\ \text{mod.}\ {\xi_1 =0},  \quad 
\chi_T\cong\chi_{{A\cup B\cup
E}}\chi_{{C}} \chi_{{D}}\ \text{mod.}\ { \xi_2=0}\label{132}.
\end{gather}
If $\chi_T$ is not irreducible by considering a suitable irreducible factor $U$ and   by a simple analysis we get   the following
subcases:$$
\chi_{{A}}\cong \chi_{{C}}, \,\chi_{{A}}\cong \chi_{{D}},\,\chi_{{E}} \cong \chi_{{D}} 
,\, \chi_{{E}}\cong \chi_{{C}} \ \text{mod.}\ {\xi_1=\xi_2=0}\label{133. 4}.
$$
By the symmetry of the tree in figure \eqref{pic36},   we need
consider only the first case.  We get easily by lemma
\ref{lem0} $|A|=|C|=1,  A=\{0\},   C=\{c\},  c=\pm
(e_1-e_2),-e_1-e_2$. So 0, $c$ are connected by an edge and $G=\bar T$ is not a tree..

$$ \xymatrix{&E  &&\\0 \ar@{--}[rd]_{1,  2} \ar@{--}[r]^{1,  h}  &B\ar@{--}[r]^{2, j}
\ar@{--}[u]^{1,  k}\ar@{--}[d]^{2,  i}  &D
\\&c&&} $$\qed
 
\subsubsection{Figure \eqref{pic35}}$$ \xymatrix{&&D&&\\A \ar@{--}[r]^{1,  h}  &B\ar@{--}[r]^{1,  k}  &C\ar@{--}[u]^{2,  i}  \ar@{--}[r]^{2,  j}  & E
 \\&&&&}$$
 We have:\begin{gather}
\chi_T|_{\xi_1=0}\cong \chi_{{A}}\chi_{{B}} \chi_{{C\cup
D\cup E}}\ \text{mod.}\ {\xi_1=0}, \quad \chi_T \cong \chi_{{A\cup
B\cup C}}\chi_{{D}} \chi_{{E}}\ \text{mod.}\ {\xi_2=0},
\label{119}
\end{gather}
Suppose that $\chi_T$ is not irreducible. The usual reasoning gives an irreducible factor $U$ so that $U\cong \chi_{ A},\chi_{ B},\chi_{ A}\chi_{ B}$ modulo $\xi_1=0$.

We may have $U\cong \chi_{ D},\chi_{ E},\chi_{ D}\chi_{ E},\chi_{{ D\cup E}}$ modulo $\xi_2=0$.

 Arguing as in the previous case we only have the possibility$$ \chi_{{B}}\cong \chi_{{D}},\quad \chi_{{B}}\cong \chi_{{E}}\quad\text{modulo}\quad \xi_1=\xi_2=0\label{123}.
$$
By  symmetry   we need to consider only the first case.  We
get by lemma \ref{lem0} $B=\{b\},   D=\{d\}, $ and $ d, b$ are joined by an edge $\pm (e_1-
e_2),-e_1-
e_2$. 

$$ \xymatrix{&&d\ar@{--}[ld]_{1,  2}&&\\A \ar@{--}[r]^{1,  h}  &b\ar@{--}[r]_{1,  k}  &C\ar@{--}[u]_{2,  i}  \ar@{--}[r]^{2,  j}  & E
 \\&&&&}$$ and $G=\bar T$ is not a tree.
\qed
\subsubsection{Figure \eqref{pic33},\eqref{pic32}}  We treat these two cases together.    $$I)\quad  \xymatrix{A \ar@{--}[r]^{1,  h}  &B\ar@{--}[r]^{2,  k}  &C\ar@{--}[r]^{1,  i}  &D\ar@{--}[r]^{2,  j}  & E
 }$$ 
$$II)\quad \xymatrix{A \ar@{--}[r]^{1,  h}  &B\ar@{--}[r]^{1,  k}  &C\ar@{--}[r]^{2,  i}  &D\ar@{--}[r]^{2,  j}  & E
 }$$  \begin{proof} I)  We have:\begin{gather}\chi_T\cong \chi_{{A}}\chi_{{B\cup C}} \chi_{{D\cup E}}\ \text{mod.}\ {\xi_1=0}, \quad \chi_T\cong\chi_{{A\cup B}}\chi_{{C\cup D}} \chi_{{E}}\ \text{mod.}\ {\xi_2=0}  \label{137. 2}. \end{gather}
Inspecting \eqref{137. 2}, by a simple analysis
we get the following possibilities:

If  $\chi_T$ is not irreducible it has a factor $U$  congruent, modulo 
$ {\xi_1=0}$ to i) $\chi_{ A}$ or  ii) 
$\chi_{B\cup C} $ or $  \chi_{{D\cup
E}} $. If $U\cong \chi_{ A}$ modulo 
$ {\xi_1=0}$ we must have
$U\cong \chi_{ E}$ modulo 
$ {\xi_2=0}$ and
\begin{equation}
\chi_{{A}}\cong \chi_{{E}}\ \text{mod.}\ {\xi_1=\xi_2=0}.\label{138. 4}
\end{equation}    Otherwise we have that $U $ is congruent to  $\chi_{{A\cup
B}} $ or $  \chi_{{C\cup D}} $ modulo 
$ {\xi_2=0}$.  %\begin{gather}\chi_{{A}}=\chi_{{C}}|_{\xi_1=\xi_2=0}\label{138.
%1}\\\mbox{or
%}\chi_{{C}}|_{\xi_1=\xi_2=0}=\chi_{{E}}|_{\xi_1=\xi_2=0}\label{138.
%2}\\\mbox{or
%}\chi_{{B}}|_{\xi_1=0}=\chi_{{D}}|_{\xi_1=\xi_2=0}\label{138.
%3}\\\mbox{or }  \chi_{{D}}|_{\xi_1=\xi_2=0} =
%\chi_{{A}}|_{\xi_1=\xi_2=0},\
%\chi_{{E}}|_{\xi_1=\xi_2=0}=
%\chi_{{B}}|_{\xi_1=\xi_2=0}. \end{gather}
\begin{gather}\chi_{{A}} \cong \chi_{{C}},\ \chi_{{C}} \cong \chi_{{E}} ,\ \chi_{{B}}  \cong \chi_{{D}},\  \chi_{{D}}\cong 
\chi_{{A}},\
\chi_{{E}}\cong 
\chi_{{B}}\ \text{mod.}\ {\xi_1=\xi_2=0}. \end{gather} The last two can be excluded by parity   of occurrences of 1,2 in their path. The first two are symmetric. Therefore we are left to consider three  cases   $\chi_{{A}}\cong \chi_{{E}},$ $ \chi_{{A}} \cong \chi_{{C}} ,\ \chi_{{B}}  \cong \chi_{{D}}$.\smallskip

If we are in case $ \chi_{{A}} \cong \chi_{{C}} ,\ \chi_{{B}}  \cong \chi_{{D}}$ by lemma
\ref{lem0} we get $|A|=|C|=1,  A=\{0\},   C=\{c\} $ (resp. $|B|=|C|=1,  A=\{b\},   C=\{c\} $) are joined by an edge
$ \pm (e_1-e_2),-e_1-e_2$ and $G=\bar T$ is not a tree.

If we have  $\chi_{{A}}\cong \chi_{{E}}$  always by  lemma
\ref{lem0} we get $|A|=|E|=1,  A=\{0\},   E=\{e\} , e\in \{\pm 2(e_1-e_2),-2e_1,-2e_2\}.$ 
\begin{equation}
\label{Bad1} I)\quad \xymatrix{0 \ar@{--}[r]^{1,  h}  &B\ar@{--}[r]^{2,  k}  &C\ar@{--}[r]^{1,  i}  &D\ar@{--}[r]^{2,  j}  & e
 }
\end{equation} II)\begin{equation}\label{134.1}
\chi_T\cong \chi_{{A}}\chi_{{B}} \chi_{  {C\cup
D\cup E}}\ \text{mod.}\ {\xi_1=0}, \quad 
\chi_T\cong \chi_{ {A\cup B\cup
C}}\chi_{{D}} \chi_{{E}}\ \text{mod.}\ {\xi_2=0},  \end{equation} If $\chi_T$ is not
irreducible,   one easily sees that there is a factor $U$ congruent modulo $\xi_1=0$ to
$\chi_{{A}}$ or $\chi_{{B}} $ or finally $\chi_{{A}} \chi_{{B}} $. Then $U$  modulo $\xi_2=0$ is congruent either to $\chi_{{D}} $ or
$\chi_{{E}} $ or $\chi_{{D}} \chi_{{E}} $. Applying Lemma \ref{lem0} a priori there are 4 possibilities that a block $A,B$ specializes to a block $D,E$, but in that Lemma we also have the   parity of 1,2 in a path joining the two blocks must be the same hence we only have two cases.  

i) $\chi_{{B}}\cong\chi_{{D}}$ modulo ${\xi_1=\xi_2=0}, $ and $ |B|=|D|=1,
B=\{b\},  D=\{d\}, $ and     $d,
b$ are joined by an edge $\pm (e_1-e_2),-e_1-e_2 $.  
In this case we contradict the fact that $G=T$ is a tree.

ii) $ \chi_{{A}}=\chi_{{E}}$ modulo ${\xi_1=\xi_2=0}$ and $ |A|=|E|=1,  A=\{0\},
E=\{e\}$. 
By inspection since $e\neq 0$ we must have $e=\pm
(2e_1-2e_2),-2e_1,-2e_2$.
All indices in the path from 0 to $e$ appear twice.   
\begin{equation}
\label{Bad2}II)\quad \xymatrix{0 \ar@{--}[r]^{1,  h}  &B\ar@{--}[r]^{1,  k}  &C\ar@{--}[r]^{2,  i}  &D\ar@{--}[r]^{2,  j}  & e
 }
\end{equation}
We now have to exclude in both cases  the second possibility \eqref{Bad1},\eqref{Bad2}.  \smallskip  

I)\quad Start from the first case.
If $k=h$ we have $$\xymatrix{0 \ar@{--}[r]^{1,  h}  &B\ar@{--}[r]^{2,  h}  &C\ar@{--}[r]^{1,  i}  &D\ar@{--}[r]^{2,  j}  & e
 }$$   If $i\neq j$ we must have that $i$ appears in one of the blocks $B,C,D$. For instance if $i$ is in $D$ we have 
 $$\xymatrix{0 \ar@{--}[r]^{1,  h}  &B\ar@{--}[r]^{2,  h}  &C\ar@{--}[r]^{1,  i}  &D_1\ar@{--}[r]^{s,  i}  &D_2\ar@{--}[r]^{2,  j}  & e
 }$$ we apply the previous analysis to the pair $h,i$ and deduce that  $|D_2\cup e|=1$ a contradiction. 
 $$\xymatrix{&&& \ar@{--}[d]^{u,  i} D'\\ 0 \ar@{--}[r]^{1,  h}  &B\ar@{--}[r]^{2,  h}  &C\ar@{--}[r]^{1,  i}  &D\ar@{--}[r]^{2,  j}  & e
 }$$ is like Picture \eqref{pic37} for indices $2,i$.
 
  The other cases are similar to this or to the previous case of \eqref{pic33}. If $i=j$ we have 
$$\xymatrix{0 \ar@{--}[r]^{1,  h}  &B\ar@{--}[r]^{2,  h}  &C\ar@{--}[r]^{1,  i}  &D\ar@{--}[r]^{2,  i}  & e
 }.$$ We   apply the previous analysis to the pair $h,i$ deducing  $e=\pm 2(e_i- e_h),$$-2e_i,-2e_h$ clearly a contradiction since we already have $e=\pm
(2e_1-2e_2),-2e_1,-2e_2$.
\smallskip
 
 If $k\neq h$ consider the positions of $k$. If $k\in B\cup C$
 $$\xymatrix{0 \ar@{--}[r]^{1,  h}  &B_1\ar@{--}[r]^{s,  k}&B_2\ar@{--}[r]^{1,  k}  &C\ar@{--}[r]^{2,  i}  &D\ar@{--}[r]^{2,  j}  & e
 }$$ $$\xymatrix{0 \ar@{--}[r]^{1,  h}  &B \ar@{--}[r]^{1,  k}  &C_1\ar@{--}[r]^{s,  k}&C_2\ar@{--}[r]^{2,  i}  &D\ar@{--}[r]^{2,  j}  & e
 }$$ by the previous discussion applied  to $k,2$  we  have that $|0\cup B_1|=1$ or $|0\cup B |=1$ a contradiction.   
 If $k\in D$
 $$\xymatrix{0 \ar@{--}[r]^{1,  h}  &B \ar@{--}[r]^{1,  k}  &C\ar@{--}[r]^{2,  i}  &D_1\ar@{--}[r]^{s,  k}&D_2\ar@{--}[r]^{2,  j}  & e
 }$$ we are in the previous case of \eqref{pic33} for the indices $k,2$ deducing $|0\cup B|=1$ again a contradiction. \smallskip
 
II).\quad We now finish the second case.  If $k=h$ we have $$\xymatrix{0 \ar@{--}[r]^{1,  h}  &B\ar@{--}[r]^{1,  h}  &C\ar@{--}[r]^{2,  i}  &D\ar@{--}[r]^{2,  j}  & e
 }$$ we are in the same situation but for the pair $h,2$. We deduce that $e=-2e_2$.  Now if $i=j$  we are in the same situation for the pair $h,j$ (or $1,j$) and deduce that $e= \pm 2 (e_1-e_j),- 2e_1,-2e_j$ a contradiction.  If $i\neq j$ we must have that $i$ appears in one of the blocks $B,C,D$. For instance if $i$ is in $D$ we have 
 $$\xymatrix{0 \ar@{--}[r]^{1,  h}  &B\ar@{--}[r]^{1,  h}  &C\ar@{--}[r]^{2,  i}  &D_1\ar@{--}[r]^{s,  i}  &D_2\ar@{--}[r]^{2,  j}  & e
 }$$ we apply the previous analysis to the pair $1,i$ and deduce that  $|D_2\cup e|=1$ a contradiction.  The other cases are similar to this or to the previous case of \eqref{pic33}.
 
 If $k\neq h$ consider the positions of $k$. If $k\in B\cup C$
 $$\xymatrix{0 \ar@{--}[r]^{1,  h}  &B_1\ar@{--}[r]^{s,  k}&B_2\ar@{--}[r]^{1,  k}  &C\ar@{--}[r]^{2,  i}  &D\ar@{--}[r]^{2,  j}  & e
 }$$ $$\xymatrix{0 \ar@{--}[r]^{1,  h}  &B \ar@{--}[r]^{1,  k}  &C_1\ar@{--}[r]^{s,  k}&C_2\ar@{--}[r]^{2,  i}  &D\ar@{--}[r]^{2,  j}  & e
 }$$ we apply  the previous discussion  to $k,2$  and  have that $|0\cup B_1|=1$ or $|0\cup B |=1$ a contradiction.   
 If $k\in D$
 $$\xymatrix{0 \ar@{--}[r]^{1,  h}  &B \ar@{--}[r]^{1,  k}  &C\ar@{--}[r]^{2,  i}  &D_1\ar@{--}[r]^{s,  k}&D_2\ar@{--}[r]^{2,  j}  & e
 }$$ we are in the previous case of \eqref{pic33} for the indices $k,2$ and again have $|0\cup B|=1$ a contradiction.\end{proof}

  \bibliographystyle{plain}

\bibliography{/lavori-in-corso/bibliografia}

 \end{document}